\documentclass{article}

\usepackage[english]{babel}

\usepackage[letterpaper,top=2cm,bottom=2cm,left=3cm,right=3cm,marginparwidth=1.75cm]{geometry}

\usepackage{algorithm2e}
\usepackage{algpseudocode}
\usepackage{booktabs}
\usepackage{multirow}
\usepackage{amsmath}
\usepackage{amsthm}
\usepackage{graphicx}
\usepackage[colorlinks=true, allcolors=blue]{hyperref}
\usepackage{flushend, cuted}
\usepackage[title]{appendix}
\usepackage{chngcntr}
\usepackage{apptools}
\AtAppendix{\counterwithin{lemma}{section}}

\usepackage{amsfonts}
\usepackage{bbm}

\usepackage{amssymb}
\usepackage{float}
\usepackage{color,amsmath,amssymb, amsfonts,amstext,amsthm}
\usepackage{xcolor}
\usepackage{epsfig, graphicx, graphics}
\usepackage{longtable}
\usepackage{cite}
\usepackage{subfigure}

\usepackage{caption}
\usepackage{bm}
\usepackage{epstopdf}

\usepackage{color, amsmath,amssymb, amsfonts, amstext,latexsym}

\usepackage{amssymb, epsfig, amssymb, latexsym}

\usepackage{longtable}

\usepackage{authblk}

\usepackage[mathscr]{euscript}

\usepackage{color}

\renewcommand{\phi}{\varphi}

\newtheorem{theorem}{Theorem}
\newtheorem{lemma}{Lemma}
\newtheorem{remark}{Remark}

\newtheorem{proposition}{Proposition}
\newtheorem{assumption}{Assumption}

\title{Reservoir Computing with Error Correction: Long-term Behaviors of Stochastic Dynamical Systems} 

\author[a]{Cheng Fang \thanks{fangcheng1@hust.edu.cn}}
\author[b]{Yubin Lu \thanks{Corresponding author: ylu117@iit.edu}}
\author[a]{Ting Gao \thanks{tgao0716@hust.edu.cn }}
\author[c]{Jinqiao Duan \thanks{duan@gbu.edu.cn}} 
\affil[a]{School of Mathematics and Statistics \& Center for Mathematical Sciences, Huazhong University of Science and Technology, Wuhan, Hubei 430074, China}
\affil[b]{Department of Applied Mathematics, College of Computing, Illinois Institute of Technology, Chicago, IL 60616, USA}
\affil[c]{Department of Mathematics and Department of Physics, Great Bay University, Dongguan, Guangdong 523000, China}

\date{June 30, 2023}

\begin{document}
\maketitle

\begin{abstract}
The prediction of stochastic dynamical systems and the capture of dynamical behaviors are profound problems. In this article, we propose a data-driven framework combining Reservoir Computing and Normalizing Flow to study this issue, which mimics error modeling to improve traditional Reservoir Computing performance and integrates the virtues of both approaches. With few assumptions about the underlying stochastic dynamical systems, this model-free method successfully predicts the long-term evolution of stochastic dynamical systems and replicates dynamical behaviors. We verify the effectiveness of the proposed framework in several experiments, including the stochastic Van der Pal oscillator, El Ni\~no-Southern Oscillation simplified model, and stochastic Lorenz system. These experiments consist of Markov/non-Markov and stationary/non-stationary stochastic processes which are defined by linear/nonlinear stochastic differential equations or stochastic delay differential equations. Additionally, we explore the noise-induced tipping phenomenon, relaxation oscillation, stochastic mixed-mode oscillation, and replication of the strange attractor.
\end{abstract}

\section{Introduction} \label{sec: introduction}

Stochastic dynamical systems describe the evolution of complex phenomena in the presence of noise. This noise may arise from a range of factors such as external fluctuations, internal agitation, uncertainty in initial/boundary conditions and parameters. The interplay between deterministic governing laws and noise results in corresponding stochastic dynamical behaviors. For example, stochastic (delay) differential equations (SDEs/SDDEs) \cite{duan2015introduction, mao2007stochastic, rihan2021delay} are two significant mathematical models that accurately describe the evolution of deterministic dynamical systems in the presence of noise. These models have demonstrated successful applications in various fields, including biology, chemistry, physics, and meteorology \cite{duan2015introduction, baxendale2007stochastic, rihan2021delay, weinan2021applied}.

For a given dynamical system, what we are concerned about are not only the state of the system, but also the dynamical behaviors it exhibits. Practical problems are usually complex and difficult to be described by concise mathematical models. As a result, their corresponding dynamical systems can be so complex that are difficult to solve. Therefore, one of the main strategies to study dynamical systems is to design appropriate tools that can characterize their behaviors, rather than solving them directly. Over the past century, researchers have developed various kinds of qualitative methods to characterize dynamical systems, such as bifurcations, chaos, attractors and their fractal dimensions\cite{strogatz2015nonlinear}. Furthermore, the dynamical behaviors of stochastic dynamical systems are richer and more challenging to quantify and analyze due to the impact of noise. For instance, for a multistable dynamical system, when noise is present, the state variable stays close to a metastable equilibrium for a long time and may sometimes cross the stable manifold and fall into the basin of attraction of another metastability. This phenomenon is also a kind of noise-induced tipping \cite{ashwin2012tipping}.

In spite of this principle-based research paradigm, one may encounter complex phenomena that lack sufficient scientific understanding. While it may not be possible to establish a complete mathematical model for highly complex phenomena, the ever-advancing fields of computer technology and observation techniques have made the data-driven research paradigm not only possible but also flourishing. In the last ten to twenty years, many data-driven methods for studying stochastic dynamical systems have been proposed and have shown a thriving development in various fields. For example, extracting stochastic governing laws from data by Koopman operators \cite{KLUS2020132416, Lu2020koopman}, Kramers-Moyal formulas \cite{Boninsegna2018sparse, lu2022extracting}, maximum likelihood estimations \cite{Dietrich2023sde, Fang2022}, physics-informed neural networks \cite{chen2021inverse, CHEN2023133691}, or estimating the probability distribution for the corresponding stochastic dynamical system \cite{LU2022nf}. Methods developed in the area of deep learning for analyzing time series are also used to study stochastic dynamical systems, for instance, forecasting the evolution of stochastic dynamical systems by Recurrent Neural Networks (RNN) \cite{HARLIM2021109922}, Neural SDE \cite{li2020scalable}, or statistics-informed neural networks \cite{ZHU2023111819}, and discovering dynamical behaviors from data \cite{HARLIM2021109922, ZHU2023111819}.

The term ``Reservoir Computing" (RC) \cite{nakajima2021reservoir} refers to a broad computational framework derived from RNN, which was developed by Jaeger \cite{Jaeger2004} and Maass et al. \cite{maass2002real, maass2011liquid}. These papers, in turn, introduced the ideas of Echo State Networks (ESN) and Liquid State Machines (LSM), respectively. The benefits of RC include its simple structure, easy training, and fixed reservoirs, which also make it effective for large-scale or high-dimensional datasets. RC has exhibited strong prediction capabilities that can guarantee prediction accuracy over a long period of time for chaotic systems \cite{Jaeger2004, yperman2016bayesian, pathak2018PhysRevLett, Griffith2019, gauthier2021next}. When concentrating on the dynamical behaviors of dynamical systems, RC replicates strange attractors, the basin of attraction, stable/unstable manifolds, etc \cite{pathak2018PhysRevLett, Griffith2019, mousumi2022multi, dagobert2022delaybased, gauthier2022unseen}. This method has been employed with success for more challenging tipping and critical transition phenomena in dynamical systems \cite{Lim2020critical, Patel2021climate, patel2023post}. As for its interpretability, the properties, the structure of its latent representation and the universality are also in progress \cite{JAEGER2007335, yildiz2012re, Griffith2019, Lukas2020, GONON202110}. Several new architectures of RC are being developed \cite{GALLICCHIO201787, gauthier2021next, dagobert2022delaybased}. Physical RC, another branch of the RC research field, refers to hardware implementation using various physical systems, substrates, and devices \cite{TANAKA2019100}. Besides, there are few studies utilizing RC for dynamical systems with noise, which is discussed in detail in Sec. \ref{sec: related work}.

Reservoir computing has manifested powerful approximation capabilities for deterministic dynamical systems, but due to the lack of an effective stochastic structure, it is still insufficient for stochastic dynamical systems. This is manifested by the rapid accumulation of errors with rolling predictions (see Fig. \ref{fig: subfig: ou rc} in Sec. \ref{sec: exper: ou}). In other words, proper probabilistic modeling of errors is a key aspect in dealing with stochastic dynamical systems \cite{HARLIM2021109922, GOTTWALD2021132911, levine2022framework}. Fortunately, the probabilistic generative models in the machine learning community are flourishing and have shown effective approximation capabilities for high-dimensional and complex probability distributions. The most popular models in recent years include Generative Adversarial Networks (GANs) \cite{NIPS2014_5ca3e9b1}, Variational Auto-Encoders (VAEs) \cite{Kingma2014}, Diffusion Models (DMs) \cite{ho2020denoising} and Normalizing Flows (NFs) \cite{NEURIPS2019_7ac71d43, papamakarios2021normalizing, Kobyzev2021,lu2022extracting, LU2022nf}, which have been successfully applied in areas such as image generation and speech synthesis.


Reservoir computing is an effective approximator for deterministic dynamical systems, while probabilistic generative models are designed to learn the underlying probability distribution of a dataset. Inspired by this, we combine Reservoir Computing and Normalizing Flow (RC-NF for short) to realize long-term prediction of stochastic dynamical systems and capture various dynamical behaviors of stochastic dynamical systems. This framework is model-free and imposes fewer constraints on the studied system. The main contributions of this work include:

\begin{itemize}
    \item Methodological innovation: the improved RC method (i.e., RC-NF introduced below) is successfully applied in stochastic dynamical systems, both defined by stochastic differential equations and stochastic delay differential equations, which can be extended to more complex stochastic systems.
    \item Dynamical behaviors capturing: we successfully capture the dynamical behaviors of stochastic dynamical systems, such as noise-induced tipping, relaxation oscillation, stochastic mixed-mode oscillation, and chaos. For stochastic chaotic systems, RC-NF is more effective than traditional RC.
    \item Robustness: we successfully predict the long-term evolution of various kinds of stochastic dynamical systems, including linear/non-linear systems, Markov/non-Markov processes, and stationary/ non-stationary models.
\end{itemize}
From a theoretical perspective, we also demonstrate the universality of RC-NF to further explain the role of the various components of the framework. We further illustrate the effectiveness of the approach using criteria such as Wasserstein distance, Kullback-Leibler divergence, transition rate, maximum Lyapunov exponent, close returns, autocorrelation function, and cross-correlation function in different experiments.

The remainder of this article is arranged as follows: Sec. \ref{sec: preliminary work} introduces stochastic dynamical systems defined by stochastic (delay) differential equations, related works, and unresolved issues. In Sec. \ref{sec:data-driven}, we begin by presenting the general frameworks of reservoir computing and normalizing flow. We then combine these frameworks to analyze stochastic dynamical systems using data. At the end of this section, a thorough analysis of the approximation ability is provided. In Sec. \ref{sec: experiments}, we demonstrate the effectiveness of our method through validation on a variety of examples, including the Ornstein–Uhlenbeck process, Double-Well system, stochastic Van der Pol oscillator, stochastic mixed-mode oscillation, linear SDDE, El Ni\~no-Southern Oscillation simplified model and stochastic Lorenz system. In Sec. \ref{sec: discussion}, we discuss our findings and potential directions for improving this work.

\section{Preliminary works} \label{sec: preliminary work}

In this section, we briefly introduce stochastic dynamical systems and highlight some related data-driven approaches that have been applied to them. Based on this, we point out the goals of this work. In the following, bold letters denote multidimensional vectors or random variables, while the subscript `t' indicates time (continuous or discrete, depending on the context). The superscript `$i$' represents different dimensions, and the superscript `$(m)$' denotes different sample trajectories, unless otherwise specified.

\subsection{Stochastic dynamical systems} \label{sec: sde}

Stochastic differential equations and stochastic delay differential equations are two types of models with significantly different characteristics. One of the key differences between the two is that the solutions of the former are Markov processes, whereas the solutions of the latter are typically non-Markov processes. This makes it difficult for general data-driven models to handle both types of systems simultaneously. 

\noindent$\bullet$ \textbf{Stochastic differential equations}

We consider the following stochastic dynamical systems defined by stochastic differential equations (SDEs) in $\mathbb{R}^d$ \cite{duan2015introduction}:
\begin{equation}
    d \bm{X}_t = \bm{f}(\bm{X}_t) dt + \bm{g} d\bm{B}_t, \label{eq: sde}
\end{equation}
where $t \in \mathbb{R}^+$, $\bm{X}_t = [X^1_t, X^2_t,\dots, X^d_t]^T \in \mathbb{R}^d$ (the superscript $T$ stands for transpose of a vector or a matrix), the drift coefficient $\bm{f}(\bm{X}_t)\in \mathbb{R}^d$  and the noise intensity (or diffusion coefficient matrix) $\bm{g}\in \mathbb{R}^d\times\mathbb{R}^d$. In this article, we assume that $\bm{g}$ is a diagonal matrix and that the elements on the diagonal are positive constants. A discussion of $\bm{g}$ as a non-diagonal matrix can be found in \cite{Fang2022}. The Brownian motion $\bm{B}_t$, taking values in $\mathbb{R}^d$, is a Gaussian process on an underlying probability space $(\Omega, \mathcal{F}, \mathit{P})$. More specifically, the stochastic process $\bm{B}_t$ satisfies: (a) $\bm{B}_0 = 0$, a.s.; (b) $\bm{B}_t$ has continuous paths, a.s.; (c) $\bm{B}_t$ has independent and stationary increments, and $\bm{B}_t - \bm{B}_s \sim N(\bm{0}, (t-s)\bm{I})$, for $t>s\ge 0$, where $N$ represents Gaussian distribution and $\bm{I}$ is the $d \times d$ identity matrix. Moreover, we should notice that the solution $\bm{X}_t$ of (\ref{eq: sde}) is a Markov process. 

It is well known that the probability density function for a given SDE satisfies a Fokker-Planck equation. That is, the probability density function $\bm{p}$ for the SDE (\ref{eq: sde}) satisfies
\begin{equation}
    \frac{\partial}{\partial t} \bm{p}(\bm{x},t) = A^* \bm{p}(\bm{x},t), \label{eq: fp equation}
\end{equation}
where $A^*$ is often called the Fokker-Planck operator, $A^*\bm{p} = - \sum_{i=1}^d \frac{\partial}{\partial x^i} (f^i \bm{p}) + \frac{1}{2} \sum_{i,j=1}^d \frac{\partial^2}{\partial x^i \partial x^j} ((\bm{gg^T})^{ij} \bm{p})$.

\noindent$\bullet$ \textbf{Stochastic delay differential equations} 

Stochastic delay differential equations (SDDEs) are appropriate models for describing stochastic dynamical systems with memory. A stochastic delay differential equation with delay $\tau$ in $\mathbb{R}^d$ is described by 
\begin{equation} \label{eq: sdde}
    \begin{cases}
    d \bm{X}_t=\bm{f}(\bm{X}_t,\bm{X}_{t-\tau})dt+\bm{g} d\bm{B}_t, \quad\text{for} \ t>0,\\ 
    \bm{X}_t=\bm{\gamma}_{-t},\quad\text{for}\ t\in[-\tau,0],
    \end{cases}
\end{equation}
where $\tau \in \mathbb{R}^+$ is the time delay, $\bm{\gamma}: [0, \tau] \rightarrow \mathbb{R}^d$, $\bm{f}: \mathbb{R}^d \times \mathbb{R}^d \rightarrow \mathbb{R}^d$, $\bm{g} \in \mathbb{R}^d \times \mathbb{R}^d$ is a diagonal matrix and the elements on the diagonal are positive constants. The non-Markov property resulting from delays makes it challenging to obtain governing equations for probability densities associated with SDDEs. Despite this, as demonstrated below, our framework is suitable for stochastic dynamical systems induced by SDDEs.

Based on the theory of stochastic analysis, powerful mathematical tools have been developed for the study of stochastic dynamical systems. However, due to the presence of randomness, many analytical tasks have become particularly challenging. For example, there are fewer examples of stochastic differential equations that can be solved explicitly compared to ordinary differential equations, and the definition of the dynamical behavior of stochastic dynamical systems is more complex. Therefore, it is necessary to develop data-driven methods for analyzing stochastic dynamical systems, especially for studies that can simultaneously handle both SDEs and SDDEs.

\subsection{Related works and unresolved issues} \label{sec: related work}

Various data-driven approaches have been proposed for the analysis of stochastic dynamical systems. However, most existing works heavily rely on specific model assumptions. For instance, the methods utilizing the Koopman operator \cite{KLUS2020132416, Lu2020koopman} or Kramers-Moyal formulas \cite{Boninsegna2018sparse, lu2022extracting} necessitate the assumption that the systems under study are stochastic differential equations. Maximum likelihood estimations \cite{Dietrich2023sde, Fang2022} are constructed based on numerical schemes of stochastic differential equations. The physics-informed neural networks are used to learn Fokker-Planck equations corresponding to the stochastic dynamical systems \cite{chen2021inverse, CHEN2023133691}. In other words, these methods overly rely on assumptions of the model, which limits their extension. These methods are not applicable to SDDEs, which serves as a concrete example.

As a type of Recurrent Neural Network, Reservoir Computing is a model-free method that can be directly applied to time series modeling, while its structure is simpler than other methods \cite{HARLIM2021109922, li2020scalable, ZHU2023111819}. Moreover, it has demonstrated strong approximation capabilities for deterministic dynamical systems, regardless of the presence of time-delay effects. Therefore, it has the potential to become a powerful tool for studying stochastic dynamical systems. There are a few applications of RC for stochastic systems. For example, the effect of stochasticity is also addressed in \cite{Patel2021climate,patel2023post}. They assume the noise originates from observations and obeys a uniform distribution. However, the randomness is treated as external noise. SDEs or SDDEs defined in Sec. \ref{sec: sde} treat the randomness as intrinsic noise and are more widely studied. Moreover, the systems studied in \cite{Patel2021climate,patel2023post} have similar dynamical behaviors (eg., intermittency, chaos) whether noise is present or not. While the systems we studied show significantly different dynamical behaviors compared to their deterministic counterparts since the noise is treated as an intrinsic part. In \cite{Lim2020critical}, authors investigate stochastic systems and select a specific trajectory with constraints for training, while our method works on trajectory data without these constraints. Besides, in \cite{khovanov2021sto}, authors use RC to study the stochasticity caused by numerical schemes for determining systems. The method mentioned in \cite{LIAO2022137} requires the cooperation of physical hardware to realize the mask of data on stochastic resonance systems. In contrast, we directly use trajectories of stochastic dynamical systems without mask operations.

The primary objective of this study is to propose a novel method to make long-term predictions regarding the evolution of stochastic dynamical systems. To be more precise, the aim is to use finite-length trajectories generated or observed from a stochastic process in order to predict probability distributions consisting of predicted trajectories in future instances. Not only can the trained model generate trajectories, but it can also predict the dynamical behaviors of stochastic dynamical systems, including transitions between metastabilities, noise-induced tippings, relaxation oscillations, stochastic mixed-mode oscillations, strange attractors, and more.

\section{Data-driven forecasting for stochastic dynamical systems} \label{sec:data-driven}

Stochastic dynamical systems induced by stochastic differential equations or stochastic delay differential equations could exhibit extremely different dynamical behaviors, such as converging to an equilibrium distribution, transition between metastable states, time-delay behavior, and chaotic behavior. In order to predict the future states or capture the dynamical behavior of a stochastic dynamical system from observed time series data, we combine Reservoir Computing and a deep generative model (e.g., Normalizing Flow) to establish a unified framework for forecasting long-term evolution and capturing the dynamical behavior. In Sec. \ref{sec: rc}, we briefly introduce the general description of Reservoir Computing. In Sec. \ref{sec: NF}, we present how to compensate for the approximated errors coming from the surrogate model (Reservoir Computing). In Sec. \ref{sec: rc-nf}, we illustrate how to combine these two key ingredients (Reservoir Computing and Normalizing Flow) for studying stochastic dynamical systems. In Sec. \ref{sec: convergence}, we analyze the universality of the proposed approach.

\subsection{Reservoir Computing for long-term forecasting} \label{sec: rc}

The umbrella term called Reservoir Computing (RC) is a type of Recurrent Neural Network, which is well known for its long prediction time, few training parameters, and low computational cost. We provide an outline framework of Reservoir Computing as follows.

Consider a time-discrete stochastic dynamical system (\ref{eq: sde}) or (\ref{eq: sdde}) $\{\bm{X}_t\}_{t=0}^{T-1}$, where $\bm{X}_t\in\mathbb{R}^d$ and $T$ is the length of time series. The time step size is denoted by $\Delta t$. Assume that the reservoir has $N$ nodes, then the $d$ dimension input vector at each iteration is coupled into the reservoir via an $N \times d $ dimensional input-to-reservoir coupling matrix $W_{in}$. We construct this matrix by randomly sampling each matrix element from a uniform distribution on the interval $[-\chi/2, \chi/2]$. The state of the reservoir at time $t$ is denoted by the $N$ dimensional vector $\bm{r}_t = [r^1_t,r^2_t,\dots,r^N_t]^T$. The time evolution of the reservoir state $\bm{r}_t$ is determined by \cite{nakajima2021reservoir}
\begin{equation} \label{eq: rc update}
    \bm{r}_{t+1} = (1 - \alpha) \bm{r}_t + \alpha \tanh (A \bm{r}_t+ W_{in} \bm{X}_t + \bm{\zeta}),
\end{equation}
when $\Delta t \rightarrow 0$, equation (\ref{eq: rc update}) can be approximated as
\begin{equation}
    \Dot{\bm{r}}_t \thickapprox \frac{\bm{r}_{t+1}-\bm{r}_{t}}{\Delta t} = - \frac{\alpha}{\Delta t} \bm{r}_t + \frac{\alpha}{\Delta t} \tanh (A \bm{r}_t+ W_{in} \bm{X}_t + \bm{\zeta}),
\end{equation}
where $A$ (an $N \times N$ matrix) is the adjacency matrix of the reservoir, $\alpha$ defines the leakage factor, the hyperbolic tangent function is applied element-wise to the vector, and the elements of the column vector $\bm{\zeta} \in \mathbb{R}^N$ are also sampled from the uniform distribution on the interval $[-\chi/2, \chi/2]$. The initial condition for the hidden state can be conveniently set to be $\bm{r}(t=0) = \bm{0}$. The adjacency matrix $A$ defines a random network of size $N$ and average degree $\langle k \rangle$.  Here, the average degree $\langle k \rangle$ of a network is the average number of links that a node has. Furthermore, matrix $A$ is rescaled such that the spectral radius of the network is $\rho$. The spectral radius $\rho$ of a network is the largest absolute value of the eigenvalues of its adjacency matrix. An intriguing and significant aspect of RC is that, once the aforementioned matrices $A$, $W_{in}$ and vector $\bm{\zeta}$ are created, they are never trained and are always fixed. This means that input vectors enter the reservoir in a fixed manner and that the connections between neurons within the reservoir do not change.

The input sequence and the desired output sequence in our work are essentially products of the same system and share the same dimension (for more general tasks, they are inconsistent). Subsequently, using a readout matrix, the estimation of RC at time $t+1$ is obtained,
\begin{equation} \label{eq: rc readout}
    \bm{\hat{X}}_{t+1} = W_{out} [1; \bm{X}_t; \bm{r}_{t+1}]^T,
\end{equation}
where output layer $W_{out}$ is a $d \times (1+ d + N)$ matrix with trainable weights. By minimizing the $\mathcal{L}_2$ error between the target states $\bm{X}_t$ and the estimate states $\bm{\hat{X}}_t = W_{out} [1; \bm{X}_t; \bm{r}_{t+1}]^T$, the matrix $W_{out}$ is trained. More specifically, the loss function is 
\begin{equation} \label{eq: loss}
    \mathcal{L}_{RC} = \sum_{t= 1}^{T-1} ||\bm{\hat{X}}_t - \bm{X}_t||^2 + \lambda ||W_{out}||^2,
\end{equation}
where $\lambda > 0$ is the regularization hyperparameter, the term $\lambda ||W_{out}||^2$ is added to prevent overfitting. Using the Tikhonov transformation, also called ridge regression, the above loss function (\ref{eq: loss}) can be solved in closed form:
\begin{equation} \label{eq: Tikhonov}
    W_{out} = \bm{Y} \bm{R}^T (\bm{R} \bm{R}^T + \lambda \bm{I})^{-1},
\end{equation}
where $\bm{R}$ is the states matrix of dimension $(1+d+n) \times (T-1)$, built using combinations of column vectors $[1; \bm{X}_{t-1}; \bm{r}_t]^T$ for every $t = 1, \dots, T-1$. The matrix $\bm{Y}$ is built in the same way using $\bm{X}_t$, for $t=1, \dots, T-1$. $\bm{I}$ is an identity matrix of dimension $1 + d + N$. 

The prediction phase uses the same update equations as those above (formulae (\ref{eq: rc update}), (\ref{eq: rc readout})) after the matrix $W_{out}$ obtained, but the input $\bm{X}_t$ is represented by the calculated output $\bm{\hat{X}}_t$ of the preceding step.


The construction of the RC framework involves several hyperparameters, and in this article, we focus on the impact of five specific hyperparameters on RC: the spectral radius $\rho$ and average degree $\langle k \rangle$ of adjacency matrix $A$; the sampling range $\chi$ of elements of matrix $W_{in}$; leakage factor $\alpha$; regularization parameter $\lambda$. Selecting suitable hyperparameters is crucial because the trainable component of RC is only readout matrix $W_{out}$, its fixed part is less malleable and therefore relies heavily on the appropriate choice of hyperparameters. The primary methods for hyperparameter tuning include grid search \cite{Rodan2011}, gradient descent \cite{JAEGER2007335}, and others.

Bayesian Optimization (BO) \cite{yperman2016bayesian, Griffith2019, karkkainen2022optimal} is a sequential design strategy for global optimization of black-box functions that does not assume any functional forms. It is more effective than grid search and performs well with minimal tuning. At each iteration, the BO algorithm chooses a set of hyperparameters from a selectable range of hyperparameters based on the posterior distribution of a loss function. Exploit uncertainty to balance exploration against exploitation. Inspired by \cite{yperman2016bayesian, Griffith2019, karkkainen2022optimal}, we select the range of each hyperparameter as shown in Table \ref{tab: hyper}, considering the efficiency of the algorithm and the reasonableness of the hyperparameters.

\begin{table}[!ht] 
\caption{The value range and data type of hyperparameters of RC in the Bayesian Optimization algorithm}
\label{tab: hyper}
\centering
\begin{tabular}{ccc} 
\bottomrule
\textbf{Hyperparameters} & \textbf{Min - Max} & \textbf{Data type}\\ 
\hline
$\rho$ & 0.3 - 1.5 & Real number\\
$k$ & 1 - 5 & Integer\\
$\chi$ & 0.3 - 1.5 & Real number \\
$\alpha$ & 0.05 - 1 & Real number\\
$\lambda$ & $10^{-10} - 1$ & Real number\\
\bottomrule
\end{tabular}
\end{table}

\subsection{Error modelling} \label{sec: NF}

It is a fundamental fact that when a model generates multi-step rolling predictions, errors can accumulate rapidly. However, for deterministic systems, reservoir computing can ensure a certain level of time prediction accuracy. For example, traditional Reservoir Computing has achieved 5–6 Lyapunov times in the prediction of large-scale spatiotemporal chaotic sequences \cite{pathak2018PhysRevLett}. The presence of stochasticity in dynamical systems renders Reservoir Computing alone inadequate for accurately predicting their long-term behavior (as seen in Sec. \ref{sec: experiments}). Due to the rolling prediction characteristics of RC, we are inspired by Kalman filtering \cite{Kalman1960, GOTTWALD2021132911} and error modeling \cite{levine2022framework, HARLIM2021109922} to compensate for the performance of RC in stochastic systems. RC can be regarded as a surrogate model for the original system. After modeling the error between the two, a strategy similar to Kalman Filtering is used to correct the subsequent predictions. Unlike deterministic systems, given the inherently random nature of stochastic dynamical systems, it is appropriate to assess the accuracy of predictions in terms of their probability distributions.

For a trajectory of a stochastic dynamical system, the single-step error at time $t$ between the RC prediction $\hat{\bm{X}}_t$ and true trajectory data $\bm{X}_t$ is $\tilde{\bm{\varepsilon}}^{(m)} := \bm{X}_t^{(m)} - \hat{\bm{X}}_t^{(m)}$, where $m =1, \dots, M$ represents $M$ different trajectories from the same stochastic dynamical system. The notation $\bm{p}(\tilde{\bm{\varepsilon}})$ represents the probability density formed by error $\tilde{\bm{\varepsilon}} := \{\tilde{\bm{\varepsilon}}^{(m)}\}$, $\mu_{\tilde{\bm{\varepsilon}}}$ stands for its probability measure. Under appropriate assumptions (Assumption \ref{assump: fix error} in Sec. \ref{sec: convergence}), we claim that the probability density $\bm{p}(\tilde{\bm{\varepsilon}})$ of the single-step error $\tilde{\bm{\varepsilon}}$ does not change with time. It might be a strong assumption, but we will see that it is reasonable for the studied systems below. Further exploration of this issue is left for future research. Various probabilistic generative models can be employed to estimate the probability distribution of the error $\tilde{\bm{\varepsilon}}$ from data including Generative Adversarial Networks \cite{NIPS2014_5ca3e9b1}, Variational Auto-Encoders \cite{Kingma2014}, Diffusion Models\cite{ho2020denoising} and Normalizing Flows \cite{NEURIPS2019_7ac71d43, papamakarios2021normalizing, Kobyzev2021, lu2022extracting, LU2022nf}. Unlike high-dimensional data such as images, our work focuses primarily on lower-dimensional physical systems. Hence, we employ Normalizing Flows to estimate error distributions. It is worth mentioning that alternative probabilistic generative models could potentially be utilized here. However, they do not have any essential differences. 

The main idea of Normalizing Flow (NF) is to find an invertible, differentiable transformation mapping the target random variable $\tilde{\bm{\varepsilon}} \in \mathbb{R}^d$ to a base random variable $\bm{u} \in \mathbb{R}^d$:
\begin{equation} \label{eq: nf h_theta}
    \tilde{\bm{\varepsilon}}=\bm{h}_\theta^{-1}(\bm{u}) \quad \bm{u}=\bm{h}_\theta(\tilde{\bm{\varepsilon}}),
\end{equation}
where the map $\bm{h}_\theta:\mathbb{R}^d\to\mathbb{R}^d$ is invertible, differentiable, and 
$\theta$ represents all parameters to be trained. The change of variables formula is that
\begin{equation} \label{eq: nf p trans}
\underbrace{\bm{p}(\tilde{\bm{\varepsilon}})}_{\text{over} \tilde{\bm{\varepsilon}}} = \underbrace{\bm{p}(\bm{h}_\theta(\tilde{\bm{\varepsilon}}))}_{\text{over}\bm{u}}\left|\det\left(\frac{\partial \bm{h}_{\theta} (\tilde{\bm{\varepsilon}})}{\partial \tilde{\bm{\varepsilon}}}\right)\right|,
\end{equation}
where $\frac{\partial \bm{h}_{\theta} (\tilde{\bm{\varepsilon}})}{\partial \tilde{\bm{\varepsilon}}}$ denotes the $d \times d$ Jacobian matrix.

Usually, we select a simple distribution as the base distribution, $\bm{u} \sim N(\bm{0},\bm{I})$. Suppose we compose $k$ transformations $\bm{h}_{\theta}(\tilde{\bm{\varepsilon}})=\bm{h}_{1} \circ \bm{h}_{2} \circ \ldots 
\circ \bm{h}_{k}(\tilde{\bm{\varepsilon}};\theta)$. The log-likelihood can be well decomposed,
\begin{equation} \label{eq: nf lpg p trans}
    \log \bm{p}(\tilde{\bm{\varepsilon}})=\log \bm{p} \left(\bm{h}_{\theta}(\tilde{\bm{\varepsilon}})\right) + \sum\limits_{i=1}^{k}\log \left| \det\dfrac{\partial \bm{h}_{i}(\tilde{\bm{\varepsilon}};\theta)}{\partial \tilde{\bm{\varepsilon}}}\right|.
\end{equation}

For the prediction results $\{\hat{\bm{X}}_t\}_{t=1}^{T-1}$ of RC, it can be regarded as single-step predictions of trajectory data $\{{\bm{X}}_t\}_{t=0}^{T-2}$. Since we assume that $\bm{p}(\tilde{\bm{\varepsilon}})$ does not change over time, we can obtain a total of $M \times (T-1)$ samples of $\tilde{\bm{\varepsilon}}$. To learn the transformation $\bm{h}_\theta$, we minimize the following negative log-likelihood,
\begin{equation} \label{eq: loss nf}
\begin{split}
    \mathcal{L}_{NF} =& - \sum_{j=1}^{M(T-1)} \log \bm{p}(\tilde{\bm{\varepsilon}}^{(j)}) \\
    =&  - \sum_{j=1}^{M(T-1)} \left( \log \bm{p} (\bm{h}_{\theta}(\tilde{\bm{\varepsilon}}^{(j)}) ) + \sum\limits_{i=1}^{k}\log 
\left| \det\dfrac{\partial \bm{h}_{i}(\tilde{\bm{\varepsilon}}^{(j)};\theta)}{\partial \tilde{\bm{\varepsilon}}^{(j)}} \right| \right).
\end{split}
\end{equation}
That is, maximizing the log-likelihood function of $\tilde{\bm{\varepsilon}}$. After training, it is easy to sample from the target distribution as long as $\bm{h}_{\theta}^{-1}$ is tractable. It is only necessary to sample from the Gaussian distribution $\bm{u}$ and apply $\bm{h}_{\theta}^{-1}$ to these samples. We note that the prediction $\tilde{\bm{X}}_t$ is the RC one-step prediction $\hat{\bm{X}}_t$ corrected by NF error distribution $\tilde{\bm{\varepsilon}}$, i.e, $\tilde{\bm{X}}_t = \hat{\bm{X}}_t + \tilde{\bm{\varepsilon}}$ for $t \ge T$. In this way, we are able to forecast the long-term evolution of stochastic process $\bm{X}$.

For the invertible and differentiable mapping $\bm{h}_{\theta}$, general options include Real NVP, ActNorm, Masked Autoregressive Flow, and others \cite{Kobyzev2021}. The rational-quadratic neural splines flow which may feature autoregressive layers (RQ-NSF (AR)) \cite{NEURIPS2019_7ac71d43} is used in this article. Rational Quadratic Splines (RQS) denote $\bm{h}_\theta$ as a monotone rational-quadratic spline on an interval as the identity function otherwise. The spline is defined by $K+1$ knots and the derivatives at the $K-1$ inner points. The positions of these knots and their corresponding derivatives are represented by the output of a neural network. An autoregressive layer is utilized to generate each entry of the output based on the previous entries of the input. We write RQ-NSF (AR) as:
\begin{equation}
h^i_\theta(\tilde{\varepsilon}^i)=\operatorname{RQS}_{\Theta(\tilde{\bm{\varepsilon}}^{1:i-1})}(\tilde{\varepsilon}^i),\quad \text{for}\  i = 1,\dots,d,
\end{equation}
where RQS represents the Rational Quadratic Splines method, the symbol $\Theta$ is a neural network, and the vector $\tilde{\bm{\varepsilon}}^{1:i-1} = (\tilde{\varepsilon}^1,\dots,\tilde{\varepsilon}^{i-1})^T$.

\subsection{Combining Reservoir Computing and Normalizing Flow} \label{sec: rc-nf}

Reservoir computing is good at capturing the evolutionary trend of the system, making it suitable as a surrogate model for long-term system prediction. Unlike the prediction of a deterministic system, normalizing flows are introduced as a compensatory mechanism for error modeling due to the presence of inherent noise. In this section, we provide a framework for modeling and predicting stochastic dynamical systems by combining Reservoir Computing and Normalizing Flow. For the sake of convenience, we shall henceforth refer to this method as RC-NF.

$\bullet$ \textbf{Data}: We generate $M$ trajectories with a total length of $T+T_{valid}+T_{test}$ from SDE (\ref{eq: sde}) or SDDE (\ref{eq: sdde}). To train the model, we utilize the initial $T$ states, validate it using the intermediate $T_{valid}$ states, and test it using the final $T_{test}$ states. 

$\bullet$ \textbf{Training (Reservoir Computing)}: Select a set of hyperparameters to generate matrices $A$, $W_{in}$ and vector $\bm{\zeta}$ to compose the reservoir structure. $M$ trajectories entered the reservoir by matrix $W_{in}$ in chronological order and obtained the reservoir state $\bm{r}_t$ by formula (\ref{eq: rc update}). The first $T_{warm}$ states are used to warm up the RC framework in order to eliminate the influence of the given initial value $\bm{r}_{0}$. After the warm-up, trajectory data with a total length of time $T-T_{warm}-1$ is used to train $W_{out}$ through Tikhonov transformation (\ref{eq: Tikhonov}). Utilizing the matrix $W_{out}$, one-step prediction $\hat{\bm{X}}_t$ is obtained by formula (\ref{eq: rc readout}).

$\bullet$ \textbf{Validation (Reservoir Computing)}: Based on the trained readout matrix $W_{out}$, rolling predictions are made at a few $T_{valid}$ steps. The $\mathcal{L}_2$ error between the forecast $\hat{\bm{X}}_t$ and the data $\bm{X}_t$, $\mathcal{L}_2 = \sum_{t=T+1}^{T+T_{valid}}(\hat{\bm{X}}_t- \bm{X}_t)^2$, is calculated as a loss function of Bayesian Optimization. BO refines the set of hyperparameters by utilizing the posterior distribution of the loss function, before resuming the training phase with the updated parameters. We select the set of hyperparameters that correspond to the minimum $\mathcal{L}_2$ loss during the BO iteration process as the optimal hyperparameters for RC. 

$\bullet$ \textbf{Training (Normalizing Flow)}: Based on the readout matrix $W_{out}$ and single-step prediction $\hat{\bm{X}}_t$ obtained by the optimal RC hyperparameters, we collect $M\times(T-T_{warm}-1)$ samples of the single-step error $\tilde{\bm{\varepsilon}}$ during the training phase. Subsequently, Normalizing Flow is used to estimate the probability density function $\bm{p}(\tilde{\bm{\varepsilon}})$ from these samples, obtaining the optimal mapping $\bm{h}_\theta$ between the target distribution and a pre-specified base distribution.

$\bullet$ \textbf{Testing (long-term prediction)}: For the $m$-th trajectory, the compensated single-step prediction can be expressed as $\tilde{\bm{X}}_t^{(m)}:= \hat{\bm{X}}_t^{(m)} + \tilde{\bm{\varepsilon}}^{(m)}$ where $\tilde{\bm{\varepsilon}}^{(m)}$ is a new sample generated by NF, $t \in [T+T_{valid}+1,T+T_{valid}+T_{test}]$. Regarding $\tilde{\bm{X}}_t^{(m)}$ as the input of the RC framework at time $t+1$, our strategy results in rolling long-term predictions for the $m$-th trajectory. If there are enough trajectories of the stochastic process $\bm{X}$, the probability density $\bm{p}(\bm{X}_t)$ of $\bm{X}$ at time $t$ can be approximated effectively.

$\bullet$ \textbf{Generating}: For the reservoir structure $A$, $W_{in}$, $W_{out}$ and the probability density $\bm{p}(\tilde{\bm{\varepsilon}})$ of error are learned by the fixed stochastic dynamical system, a trajectory of any length can be generated by rolling a short warm-up trajectory. More specifically, assume a short trajectory $\{\bm{X}_t\}_{t=0}^{T_{warm}-1}$ for warm-up, the input of RC at the time interval $[0,T_{warm}-1]$ is given by this trajectory. After $T_{warm}$ steps, the RC input is given by the corrected $\tilde{\bm{X}}_t$, which is obtained by plugging the compensated error of NF into the single-step prediction of RC.

Fig. \ref{fig: subfig: RC-NF} presents the overall framework in diagram form, while Fig. \ref{fig: subfig: time} depicts the time allocation for testing and generation tasks. Additionally, the pseudocode for RC-NF is outlined in Algorithm \ref{alg: rc-nf}.

\begin{figure}[!ht]
  \centering
   \subfigure[RC-NF framework]{
    \label{fig: subfig: RC-NF} 
    \includegraphics[width=0.85\textwidth]{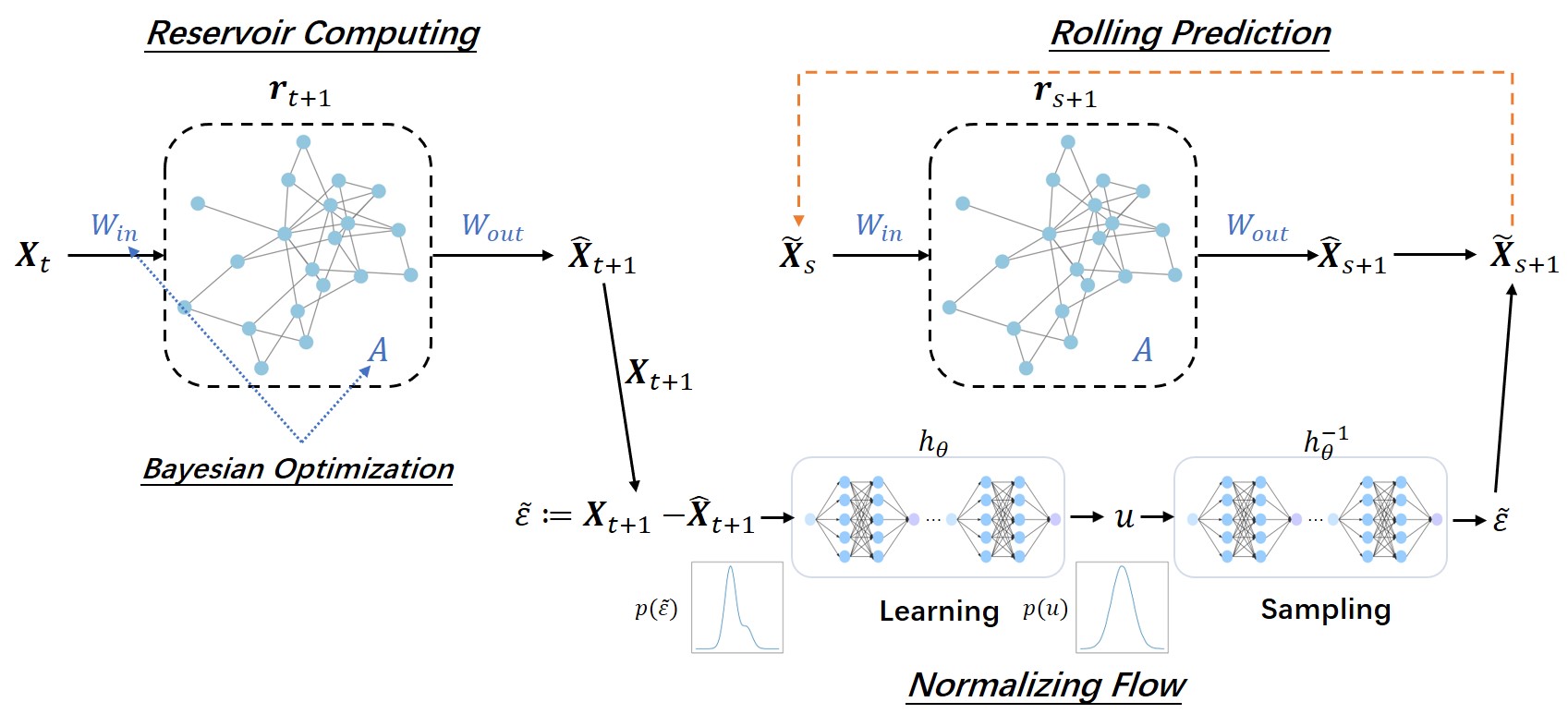}}\\
  \subfigure[Time allocation]{
    \label{fig: subfig: time} 
    \includegraphics[width=0.85\textwidth]{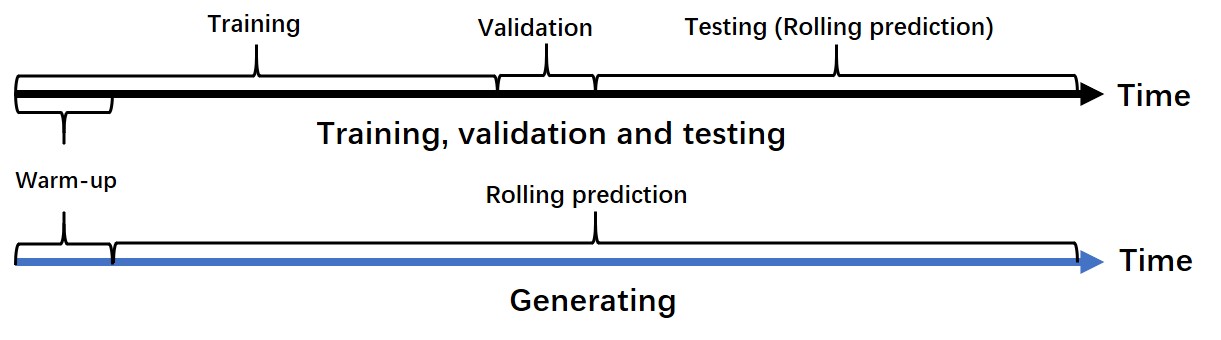}}
  \caption{(a) The flowchart of RC-NF. In the training phase, the RC readout matrix $W_{out}$ is trained according to multiple trajectory data and Tikhonov transformation, and the hyperparameters of RC fixed structure are searched by BO. NF is used to learn the probability density of the single-step prediction error $\tilde{\bm{\varepsilon}}$. Rolling predictions $\tilde{\bm{X}}_s$ are obtained by plugging the compensated error $\tilde{\bm{\varepsilon}}$ into the RC single-step prediction $\hat{\bm{X}}_s$. Without loss of generality, we assume $t \leq s$ in (a). (b) Time allocation for tasks of testing and generating new trajectories. Both require a short warm-up phase.}
  \label{fig: framework} 
\end{figure}

\RestyleAlgo{ruled}
\begin{algorithm}[!ht]
\caption{Reservoir Computing with Normalizing Flow}\label{alg: rc-nf}
\KwIn{trajectories $\bm{X}_t^{(m)}$, $m = 1,\dots,M$, $t = 0,\dots,T-1$;}
\KwOut{RC readout matrix $W_{out}$, probability density of the single-step error $\bm{p}(\tilde{\bm{\varepsilon}})$;}
\textbf{Initialization}:\\
hyperparameters selected by BO: the spectral radius $\rho$ and average degree $\langle k \rangle$ of adjacency matrix $A$, the sampling range $\chi$ of elements of $W_{in}$, leakage factor $\alpha$, regularization parameter $\lambda$\;
random variable $u \sim N(\bm{0},\bm{I})$\;

\textbf{Construct}:\Comment{Reservoir Computing}\\
matrices $A$, $W_{in}$, vector $\bm{\zeta}$\;
\For{$m = 1:M$}{
\For{$t = 1:T-1$}{
calculate and record reservoir states $\bm{r}_t^{(m)}$\;
}
}
\textbf{Construct}: \\
states matrix $\bm{R} = [\dots, [1; \bm{X}^{(m)}_{t-1}; \bm{r}^{(m)}_t]^T, \dots]$, $m = 1,\dots,M$, $t = 1,\dots,T-1$\;
observation matrix $\bm{Y} = [\dots, \bm{X}^{(m)}_{t}, \dots]$, $m = 1,\dots,M$, $t = 1,\dots,T-1$\;  
readout matrix $W_{out} = \bm{Y} \bm{R}^T (\bm{R} \bm{R}^T + \lambda \bm{I})^{-1}$\;
\For(\Comment{Collection of single-step error samples}){$m = 1:M$}{
\For{$t = 1:T-1$}{
calculate the one-step prediction $\hat{\bm{X}}_{t}^{(m)} = W_{out} [1; \bm{X}_{t-1}^{(m)}; \bm{r}_{t}^{(m)}]^T$\;
calculate the one-step error sample $\tilde{\bm{\varepsilon}}^{(m)}_t = \hat{\bm{X}}_{t}^{(m)} - \bm{X}_{t}^{(m)}$ \;
}
}
\While(\Comment{Normalizing Flow}){$i \leq \text{iterations}$}{
calculate the loss function $\mathcal{L}_{NF}$\;
updata the parameters of function $\bm{h}_\theta$\;
}
\textbf{Obtain}:\\
probability density of the single-step error $\bm{p}(\tilde{\bm{\varepsilon}}) = \bm{p}(\bm{h}^{-1}_\theta(\bm{u}))$\;
\KwResult{Single-step prediction $\tilde{\bm{X}}_{s}^{(m)}:= \hat{\bm{X}}_{s}^{(m)} + \tilde{\bm{\varepsilon}}^{(m)}$ for $s\geq T$, where the sample $\tilde{\bm{\varepsilon}}^{(m)}$ is sampled from the probability density function $\bm{p}(\tilde{\bm{\varepsilon}})$.}
\end{algorithm}

\subsection{Universality of RC-NF} \label{sec: convergence}

In this section, we will analyze the universality of RC-NF mentioned in Sec. \ref{sec: rc-nf}. For $d$-dimensional SDEs (\ref{eq: sde}) or SDDEs (\ref{eq: sdde}), assume the drift and diffusion coefficients satisfy the Lipschitz and growth conditions so that solutions of them exist and are unique \cite{duan2015introduction, weinan2021applied,rihan2021delay}.

\noindent$\bullet$ \textbf{Convergence of numerical scheme}

When dealing with real or simulation data, we can obtain trajectory data through discrete-time observations or through numerical simulations of SDE (\ref{eq: sde}) or SDDE (\ref{eq: sdde}). Let \{$\bm{X}_t^{\delta t}$\} be discrete-time observations or the numerical solution of SDE (\ref{eq: sde}) or SDDE (\ref{eq: sdde}) with (maximal) time step size $\delta t$ in this section. \{$\bm{X}_t^{\delta t}$\} and \{$\bm{X}_t$\} are not distinguished in the other sections. 

The most commonly used numerical scheme for SDE (\ref{eq: sde}) is the Euler-Maruyama scheme:
\begin{equation} \label{euler-maruyama}
    \bm{X}_{t+\delta t} = \bm{X}_t + \bm{f}(\bm{X}_t) \delta t + \bm{g}(\bm{X}_t) \delta \bm{B}_t,
\end{equation}
where $\delta \bm{B}_t \sim N(\bm{0}, \delta t \bm{I})$. Then the Euler-Maruyama scheme has the following convergence result.

\begin{lemma} \label{lemma: strong convergence} (Weinan E \cite{weinan2021applied}, Proposition 7.22) The Euler-Maruyama scheme is of strong order 1/2. More specifically, 
\begin{equation} 
    \max_{0 \leq t \leq T} \mathbb{E}| \bm{X}_t^{\delta t} - \bm{X_t}|^2 \leq C (\delta t),
\end{equation}
where $C$ is a constant independent of $\delta t$.
\end{lemma}

The numerical scheme and convergence of SDDEs (\ref{eq: sdde}) can be found in \cite{rihan2021delay}. These convergence results mean that as the time step $\delta t \rightarrow 0$, the numerical solutions converge to the corresponding solutions of SDE (\ref{eq: sde}) or SDDE (\ref{eq: sdde}).

\noindent$\bullet$ \textbf{Universality of Reservoir Computing}

We introduce some notations and background knowledge on filters and reservoir systems, from discrete-time setup. For more details, see \cite{Boyd1985, Lukas2020, GONON202110}. Let $\mathbb{Z} = \{\dots,-1,0,1,\dots\}$ and $\mathbb{Z}_{-} = \{\dots,-1,0\}$. The sets $(\mathbb{R}^d)^\mathbb{Z}$ and $(\mathbb{R}^d)^{\mathbb{Z}_{-}}$ are composed of infinite $\mathbb{R}^d$-valued sequences of the type $(\dots, \bm{z_{-1}}, \bm{z_{0}}, \bm{z_{1}}, \dots)$ and $(\dots, \bm{z_{-1}}, \bm{z_{0}})$. Denote the space of real $n \times m$ matrices by $\mathbb{M}_{n,m}$, $\mathbb{M}_{n}$ for the space of $n$-dimensional square matrices.

A filter is a map $U: (\mathbb{R}^d)^\mathbb{Z} \rightarrow \mathbb{R}^\mathbb{Z}$. If for any $\bm{z}, \bm{z}' \in (\mathbb{R}^d)^\mathbb{Z}$ which satisfy $\bm{z}_s = \bm{z}'_s$ for all $s \leq t$ for a given $t \in \mathbb{Z}$, one has that $U(\bm{z})_t = U(\bm{z}')_t$, then the filter $U$ is called \textbf{causal}. Define the time delay operator $T_{-s}: (\mathbb{R}^d)^\mathbb{Z} \rightarrow (\mathbb{R}^d)^\mathbb{Z}$ by $T_{-s}(\bm{z})_t := \bm{z}_{t+s}$ for any $s \in \mathbb{Z}$. A \textbf{time-invariant} filter $U$ refers that for all $s \in \mathbb{Z}$, $T_{-s} \circ U = U \circ T_{-s}$.

We refer to a map $H: (\mathbb{R}^d)^{\mathbb{Z}_{-}} \rightarrow \mathbb{R}$ as a functional. Causal and time-invariant filters can be equivalently described through the use of their naturally associated functionals. Given a causal and time-invariant $U$, its associated functional is defined by $H_U (\bm{z}) := U(\bm{z}^e)_0$, where $\bm{z}^e$ is an arbitrary extension of $\bm{z} \in (\mathbb{R}^d)^{\mathbb{Z}_{-}}$ to $(\mathbb{R}^d)^{\mathbb{Z}}$. Conversely, one may define a causal and time-invariant filter $U_H:(\mathbb{R}^d)^{\mathbb{Z}} \rightarrow \mathbb{R}^{\mathbb{Z}}$ through a given functional $H$ by setting $U_H(\bm{z})_t := H(\pi_{\mathbb{Z}_-} \circ T_{-t}(\bm{z}))$ where $\bm{z} \in (\mathbb{R}^d)^{\mathbb{Z}}$, $\pi_{\mathbb{Z}_-}: (\mathbb{R}^d)^{\mathbb{Z}} \rightarrow (\mathbb{R}^d)^{\mathbb{Z}_{-}}$ is the natural projection. The above maps $U$, $T_s$, $H$, and $\pi_{\mathbb{Z}_-}$ can be naturally extended to the measure space $((\mathbb{R}^d)^{\mathbb{Z}}, \otimes_{t \in \mathbb{Z}} \mathcal{B}(\mathbb{R}^d))$ (a general notation with no specific correspondence). We use the capital bold letter $\bm{Z}$ to indicate a sequence of a $d$-dimensional discrete stochastic process.

The echo state network (ESN) \cite{Jaeger2004} can be formulated as:
\begin{equation} \label{eq: esn}
    \begin{cases}
    \bm{r}_t=\sigma(A\bm{r}_{t-1}+W_{in}\bm{z}_t+\bm{\zeta}),\\ y_t=\bm{w}^{T}\bm{r}_t,
    \end{cases}
\end{equation}
where $(\bm{z}_{t})_{t \in \mathbb{Z}}$, $(y_{t})_{t \in \mathbb{Z}}$ represent $d$-dimensional input and one-dimensional output discrete time sequences respectively. The symbol $A$, $\bm{W_{in}}$, $\bm{\zeta}$, $\bm{r}_t$ are consistent with definitions in Sec. \ref{sec: rc}. The map $\sigma: \mathbb{R} \rightarrow \mathbb{R}$ is an activation function that is applied element-wise to the vector. The vector $\bm{w}$ is a readout matrix. A filter obtained by ESN (\ref{eq: esn}) satisfies the so-called echo state property (ESP) given by the following statement \cite{yildiz2012re, Lukas2020}: for any $\bm{z} \in (\mathbb{R}^d)^\mathbb{Z}$, there exists a unique $\bm{r} \in (\mathbb{R}^N)^\mathbb{Z}$ such that the first formula of (\ref{eq: esn}) holds. 

\begin{remark} \label{remark: esn and rc}
    There are slight differences between the ESN and RC mentioned in Sec. \ref{sec: rc}. Actually, the update rule of reservoir state $\bm{r}_t$ (\ref{eq: rc update}) contains an additional item $(1 - \alpha) \bm{r}_t$ that explicitly correlates $\bm{r}_t$ and $\bm{r}_{t+1}$ to represent memory and  use $\alpha$ to control the amplitude. Additionally, the additional $1+d$ dimensions of the row of $W_{out}$ (\ref{eq: Tikhonov}) are used to explicitly correlate inputs and outputs as well as to include a bias term. These are all to enhance the robustness of the model when used in practice. Appendix \ref{appendix: ESN} explores an experiment on how ESN and RC differ, and the performance of RC is slightly better than that of ESN.
\end{remark}

Based on the theorem in \cite{Lukas2020} (Lemma \ref{lemma: ortega} in Appendix \ref{appendix: convergence}), we show that if the time delay operator is bounded, then filters of the form (\ref{eq: esn}) are identical in the $L^p$-sense when the input sequences can be represented by each other using time delay operators.

\begin{theorem} \label{th: U lp uniqueness}
Let $U$ be a causal and time-invariant filter and its associated functional is $H$. Fix $p \in [1, \infty)$, let $\bm{Z}$ be a fixed $\mathbb{R}^d$-valued input process, and $H(\pi_{\mathbb{Z}_-}(\bm{Z})) \in L^p(\Omega, \mathcal{F}, \mathbb{P})$. Assume the time delay operator $T_{-s}$ is a bounded operator for any $s \in \mathbb{Z}$. Then there exists causal and time-invariant filters $U'$, $U^{s'}$ satisfying (\ref{eq: esn}), constructed by input sequences $\pi_{\mathbb{Z}_-}(\bm{Z})$, $\pi_{\mathbb{Z}_-} \circ T_{-s} (\bm{Z})$, respectively, that are identical in the $L^p$-sense. That is, for every $\varepsilon > 0$, these causal and time-invariant filters $U'$, $U^{s'}$ satisfy $\| U'(\bm{Z})_s - U^{s'}(\bm{Z})_s \|_p < (\| T_{-s}\|_p + 1)\varepsilon$.
\end{theorem}
The proof of this theorem is shown in Appendix \ref{appendix: convergence}.

For the $d$-dimensional stochastic process $(\bm{X}_t)_{t \in [0,T]}$ defined by SDE (\ref{eq: sde}) or SDDE (\ref{eq: sdde}), a sufficient condition for $\bm{X}_t \in L^2(\Omega, \mathcal{F}, \mathbb{P})$ is that the drift and diffusion coefficients satisfy the Lipschitz and growth conditions. We can construct countable infinite discrete sequences with respect to the process $\bm{X}$ for $t \in [0, T]$, denote as $(\bm{X}_t)_{t \in  \mathbb{Z}}:=(\dots, \bm{X}_{t-1}, \bm{X}_{t}, \bm{X}_{t+1}, \dots)$. The counterpart based on the Euler-Maruyama scheme is denoted as $(\bm{X}_t^{\delta t})_{t \in  \mathbb{Z}}:=(\dots, \bm{X}^{\delta t}_{t-1}, \bm{X}^{\delta t}_{t}, \bm{X}^{\delta t}_{t+1}, \dots)$. Lemma \ref{lemma: strong convergence} results in $\bm{X}_t^{\delta t} \in L^2(\Omega, \mathcal{F}, \mathbb{P})$ for all $t \in  \mathbb{Z}$. We generalize the conclusions of Lemma \ref{lemma: ortega} and Theorem \ref{th: U lp uniqueness} to discrete time series $(\bm{X}_t^{\delta t})_{t \in  \mathbb{Z}}$ based on the numerical scheme.


\begin{theorem} \label{th: main}
Define a filter $\bm{U}: ((\mathbb{R}^d)^{\mathbb{Z}}, \otimes_{t \in \mathbb{Z}} \mathcal{B}(\mathbb{R}^d)) \rightarrow ((\mathbb{R}^d)^{\mathbb{Z}}, \otimes_{t \in \mathbb{Z}} \mathcal{B}(\mathbb{R}^d))$ by $\bm{U}((\bm{X}_t^{\delta t})_{t \in  \mathbb{Z}})_t = \bm{X}_{t+1}^{\delta t}$. 
Let $(\bm{X}_t^{\delta t})^{(1)}:=(\dots, \bm{X}^{\delta t}_{t-1}, \bm{X}^{\delta t}_{t})$ be a fixed $\mathbb{R}^d$-valued input process. 
Then for arbitrary $\varepsilon > 0$, there exists $N \in \mathbb{N}$, $W_{in} \in \mathbb{M}_{N,d}$, $\bm{\zeta} \in \mathbb{R}^d$, $A \in \mathbb{M}_N$, $\bm{w} \in \mathbb{R}^N \times \mathbb{R}^d$ such that (\ref{eq: esn}) has the ESP, the corresponding filter is causal and time-invariant, the associated functional satisfies $\bm{H}^{A, W_{in}, \bm{\zeta}}_{\bm{w}}((\bm{X}_t^{\delta t})^{(1)}) \in L^2(\Omega, \mathcal{F}, \mathbb{P})$ and 
\begin{equation} \label{functional Lp x}
    \| \bm{H}((\bm{X}_t^{\delta t})^{(1)}) - \bm{H}^{A, W_{in}, \bm{\zeta}}_{\bm{w}}((\bm{X}_t^{\delta t})^{(1)}) \|_2 < \varepsilon.
\end{equation}
Additionally, for every $s \in \mathbb{Z}$, consider input sequences $(\bm{X}_t^{\delta t})^{(1)}$ and $T_{-s}((\bm{X}_t^{\delta t})^{(1)})$, there exists causal and time-invariant filters $U'$, $U^{s'}$ satisfying (\ref{eq: esn}), constructed by input sequences $(\bm{X}_t^{\delta t})^{(1)}$, $T_{-s}((\bm{X}_t^{\delta t})^{(1)})$, that are identical in the $L^2$-sense.
\end{theorem}
Appendix \ref{appendix: convergence} provides the proof for this theorem.

\begin{remark} \label{remark: training convergence}
As revealed in Theorem \ref{th: U lp uniqueness} and Theorem \ref{th: main}, for every $\varepsilon > 0$, we have $\| \bm{U}'((\bm{X}_t^{\delta t})_{t \in  \mathbb{Z}}))_s - \bm{U}^{s'}((\bm{X}_t^{\delta t})_{t \in  \mathbb{Z}}))_s \|_2 < \varepsilon$ for arbitrary $s \in \mathbb{Z}$, where $\bm{U}'$ and $\bm{U}^{s'}$ are filters of the form (\ref{eq: esn}) constructed by Lemma \ref{lemma: ortega} and the corresponding input sequences can be represented by time delay operators to each other. We use the notation $\hat{\bm{X}}_t$ in Sec. \ref{sec: rc} to indicate the output of ESN. The above conclusion means that for a semi-infinite input sequence of the form $(\bm{X}_t^{\delta t})^{(1)}$, $\| \hat{\bm{X}}_{t+1} - \bm{X}_{t+1}^{\delta t}\|_2 = \| \bm{U}'((\bm{X}_t^{\delta t})^{(1)})_0 - \bm{U}((\bm{X}_t^{\delta t})^{(1)})_0 \|_2 =\| \bm{H}'((\bm{X}_t^{\delta t})^{(1)}) - \bm{H}((\bm{X}_t^{\delta t})^{(1)})  \|_2 < \varepsilon$, for all $t \in \mathbb{Z}$. Furthermore, this shows that a filter of the form (\ref{eq: esn}) can approximate the numerical solution \{$\bm{X}_t^{\delta t}$\} in the $L^2$-sense, which is denoted as \{$\hat{\bm{X}}_t$\}. Lemma \ref{lemma: strong convergence} states that the numerical solution $\{\bm{X}_t^{\delta t}\}$ converges to solutions of SDEs (\ref{eq: sde}) or SDDEs (\ref{eq: sdde}) when the time step $\delta t \rightarrow 0$. The current theoretical results are derived from discrete time series, and further investigation is required to determine whether the theorems in this section hold for continuous-time stochastic processes in the time domain of $[0, T]$.
\end{remark}


\noindent$\bullet$ \textbf{Universality of Normalizing Flow}

According to Theorem \ref{th: U lp uniqueness} and Theorem \ref{th: main}, Reservoir Computing (RC) can serve as a universal approximator for the stochastic dynamical systems under study. However, the limitation of the inherent structure of RC, finite-time data, and the limited amount of sample trajectories, lead to approximation errors and generalization errors. It is worth noting that during the testing phase (prediction phase), there is no data available to provide us with information about the system. Specifically, at each step of the training phase, the input $\bm{X}_t^{\delta t}$ is from SDE (\ref{eq: sde}) or SDDE (\ref{eq: sdde}), with information on both drift and diffusion coefficients. In the testing phase, revisiting formulae (\ref{eq: rc update}), (\ref{eq: rc readout}) for RC, or (\ref{eq: esn}) for ESN, with the input $\bm{X}_t^{\delta t}$ replaced by the predicted output $\bm{\hat{X}}_t$ of the preceding step, the evolution laws of reservoir states $\bm{r}_t$ are polluted by prediction errors. Rolling predictions, on the other hand, necessarily lead to a rapid accumulation of errors. According to these observations, the RC model fails as the discrepancy between the stochastic process predicted by RC and the target stochastic process widens over time.

If we can guarantee that the first-step prediction is correct (without error, in an ideal situation), then RC is still a reasonable approximator for the second-step prediction. Otherwise, the phenomenon of the accumulation of errors appears, which is significant for long-term predictions. More specifically, our objective is to accurately predict the state  $\bm{X}_{T+k}^{\delta t}$ in the sense of distribution, where $k > 0$ is an integer, that is, to estimate the probability density $\bm{p}(\bm{X}_{T+k}^{\delta t})$ from data. Therefore, the task of NF introduced in Sec. \ref{sec: NF} is to correct the probability density of output $\hat{\bm{X}}_{T+1}$ to $\bm{p}(\bm{X}^{\delta t}_{T+1})$ in single-step prediction.

It should be noticed that the diffusion coefficient matrix $\bm{g}$ is a constant diagonal matrix, namely, the SDE/SDDE with additive noise. Besides, based on the structure of RC (formulae (\ref{eq: rc update}), (\ref{eq: rc readout})) and the observation of numerical experiments presented in Sec. \ref{sec: experiments}, we propose the following assumption. Further exploration of this issue is left for future research.

\begin{assumption} \label{assump: fix error}
If we fix the time step size $\delta t$ between $\bm{X}^{\delta t}_{t+1}$ and $\bm{X}^{\delta t}_{t}$ for all $t \in \mathbb{Z}$, then the single-step prediction error $\tilde{\bm{\varepsilon}} := \bm{X}_{t+1}^{\delta t} - \hat{\bm{X}}_{t+1}$ between the true filter $U$ and the filter $U'$ of the form (\ref{eq: esn}) obtained by Theorem \ref{th: main} has the same distribution at each time $t$, that is, $\bm{p}(\tilde{\bm{\varepsilon}})$ is fixed for all $t \in \mathbb{Z}$.
\end{assumption}

Based on Assumption \ref{assump: fix error}, we can learn the error density function $\bm{p}(\tilde{\bm{\varepsilon}})$ using NF from time-independent error data. For a special class of transformation $\bm{h}_\theta$ (\ref{eq: nf h_theta}), we can conclude the following universality with respect to NF. We start with some definitions. A mapping $\bm{h}_\theta = (h^{(1)}_\theta, \dots, h^{(d)}_\theta): \mathbb{R}^d \rightarrow \mathbb{R}^d$ is called \textbf{triangular} if $h^{i}_\theta$ is a function of $\tilde{\bm{\varepsilon}}^{1:i}$ for each $i = 1,\dots, d$, where $\tilde{\bm{\varepsilon}}^{1:i} = (\tilde{\varepsilon}^1,\dots, \tilde{\varepsilon}^i)^T \in \mathbb{R}^i$. Such a triangular map $\bm{h}_\theta$ is called \textbf{increasing} if $h^{i}_\theta$ is an increasing function of $\varepsilon^i$ for each $i$.

\begin{proposition}
\label{lemma: measure transform} (Kobyzev \cite{Kobyzev2021}, Proposition 4)
If $\mu$ and $\nu$ are absolutely continuous Borel probability measures on $\mathbb{R}^d$, then there exists an increasing triangular transformation $\bm{h}_\theta: \mathbb{R}^d \rightarrow \mathbb{R}^d$, such that $\nu = \mu(\bm{h}_\theta^{-1}):=\bm{h}_{\theta *} \mu$. This transformation is unique up to null sets of $\mu$. A similar result holds for measures on $[0,1]^d$. 
\end{proposition}

\begin{proposition}
\label{lemma: measure transform weak} (Kobyzev \cite{Kobyzev2021}, Proposition 5)
If $\mu$ is an absolutely continuous Borel probability measure on $\mathbb{R}^d$ and \{$\bm{h}_{\theta n}$\} is a sequence of maps $\mathbb{R}^d \rightarrow \mathbb{R}^d$ which converges pointwise to a map $\bm{h}_\theta$, then a sequence of measures $(\bm{h}_{\theta n})_* \mu$ weakly converges to $\bm{h}_{\theta *} \mu$.
\end{proposition}

\noindent$\bullet$ \textbf{Error correction in distribution}

Theorem \ref{th: main} states that there exists an RC structure approximating the single-step prediction with $L^2$-norm. However, based on the analysis of the previous subsection, Theorem \ref{th: main} fails to guarantee accurate long-term predictions. Fortunately, the universality of normalizing flow (Proposition \ref{lemma: measure transform} and Proposition \ref{lemma: measure transform weak}) provides a way to estimate the distribution of prediction error. We can compensate for the error of each RC prediction step using NF to effectively mitigate the accumulation of errors over time. We impose that the laws of $\hat{\bm{X}}_{t+1}$ and error $\tilde{\bm{\varepsilon}}$ are absolutely continuous with respect to Borel probability measures in order to ensure the existence of their densities. This assumption is given for NF and may be relaxed when using other generative models. If Assumption \ref{assump: fix error} holds, using Proposition \ref{lemma: measure transform} and Proposition \ref{lemma: measure transform weak}, we summarize these descriptions into the following theorem, namely, constructing an RC-NF framework to approximate the distribution of $\bm{X}_{t+1}^{\delta t}$.

\begin{theorem} \label{th: d-convergence}
Given a random variable $\bm{u} \sim N(\bm{0},\bm{I})$, there exists an increasing triangular transformation $\bm{h}_\theta: \mathbb{R}^d \rightarrow \mathbb{R}^d$ and a sequence of maps \{$\bm{h}_{\theta n}: \mathbb{R}^d \rightarrow \mathbb{R}^d$\} which converges pointwise to the map $\bm{h}_\theta$, such that $\hat{\bm{X}}_{t+1} + \bm{h}_{\theta n}^{-1}(\bm{u})$ convergences to $\bm{X}^{\delta t}_{t+1}$ in distribution.
\end{theorem}
\begin{proof}
According to Proposition \ref{lemma: measure transform} and Proposition \ref{lemma: measure transform weak}, there exists an increasing triangular transformation $\bm{h}_\theta: \mathbb{R}^d \rightarrow \mathbb{R}^d$ and a sequence of maps \{$\bm{h}_{\theta n}: \mathbb{R}^d \rightarrow \mathbb{R}^d$\} which converges pointwise to the map $\bm{h}_\theta$, such that $\mu_{\bm{u}} = \bm{h}_{\theta *}\mu_{\tilde{\bm{\varepsilon}}}$ and  a sequence of measures $(\bm{h}_{\theta n})_* \mu_{\tilde{\bm{\varepsilon}}}$ weakly converges to $\bm{h}_{\theta *} \mu_{\tilde{\bm{\varepsilon}}}$. Equivalently, $\mu_{\tilde{\bm{\varepsilon}}} = \bm{h}_{\theta *}^{-1}\mu_{\bm{u}}$ and  a sequence of measures $(\bm{h}_{\theta n}^{-1})_* \mu_{\bm{u}}$ weakly converges to $\bm{h}_{\theta *}^{-1} \mu_{\bm{u}}$.

For all $\bm{f} \in C_b(\mathbb{R}^d)$ (space of bounded continuous functions), we denote ${\tilde{\bm{\varepsilon}}}_n = \bm{h}_{\theta n}^{-1}(\bm{u})$, then
\begin{equation*}
\begin{split}
    & \lim_{n \rightarrow \infty} \int_{\mathbb{R}^d} \bm{f}(\bm{x}) \mu_{\hat{\bm{X}}_{t+1}+{\tilde{\bm{\varepsilon}}}_n}(d\bm{x})\\
    = & \lim_{n \rightarrow \infty} \int_{\mathbb{R}^d} \bm{f}(\bm{x}) \mu_{\hat{\bm{X}}_{t+1}}*\mu_{{\tilde{\bm{\varepsilon}}}_n}(d\bm{x}) \quad (\text{See Corollary 1.2.3 in \cite{applebaum2009levy}, p. 23})\\
    = & \lim_{n \rightarrow \infty} \int_{\mathbb{R}^d} \int_{\mathbb{R}^d} \bm{f}(\bm{y+z}) \mu_{\hat{\bm{X}}_{t+1}}(d\bm{y})\mu_{{\tilde{\bm{\varepsilon}}}_n}(d\bm{z}) \quad (\text{See Proposition 1.2.2 in \cite{applebaum2009levy}, p. 22})\\ 
    = & \int_{\mathbb{R}^d} \int_{\mathbb{R}^d} \bm{f}(\bm{y+z}) \mu_{\hat{\bm{X}}_{t+1}}(d\bm{y})\mu_{{\tilde{\bm{\varepsilon}}}}(d\bm{z})\\
    = & \int_{\mathbb{R}^d} \bm{f}(\bm{x}) \mu_{\hat{\bm{X}}_{t+1}}*\mu_{{\tilde{\bm{\varepsilon}}}}(d\bm{x}) = \int_{\mathbb{R}^d} \bm{f}(\bm{x}) \mu_{\bm{X}^{\delta t}_{t+1}}(d\bm{x}),
\end{split}
\end{equation*}
where, in the penultimate line, the function $\int_{\mathbb{R}^d} \bm{f}(\bm{y+z}) \mu_{\hat{\bm{X}}_{t+1}}(d\bm{y})$ is a bounded continuous function due to $\bm{f} \in C_b(\mathbb{R}^d)$ and the Borel probability measure $\mu_{\hat{\bm{X}}_{t+1}}(d\bm{y})$. Then the result follows, i.e., $\hat{\bm{X}}_{t+1} + \bm{h}_{\theta n}^{-1}(\bm{u})$ convergences to $\bm{X}^{\delta t}_{t+1}$ in distribution.
\end{proof}

\begin{remark}
    Theorem \ref{th: d-convergence} states that the predicted result of RC-NF for the next step will converge to the true one in distribution, which is slightly different from Theorem \ref{th: main}. The convergence result in Theorem \ref{th: main} is in the strong sense (in $L^2$-norm), whereas the convergence result in Theorem \ref{th: d-convergence} is in the weak sense (in distribution). The conclusion in Theorem \ref{th: d-convergence} is weak than Theorem \ref{th: main}, but it is enough to ensure long-term predictions in distribution due to the error correction mechanism. Although in the weak sense, Theorem \ref{th: d-convergence} is more suitable for realistic scenarios. 
\end{remark}

The universality described above could be proved for spline flows \cite{NEURIPS2019_7ac71d43} with splines for coupling functions. In fact, RQ-NSF (AR) (in Sec. \ref{sec: NF}) itself is a triangular and increasing map. The iterative training of NF is to construct a sequence of maps \{$\bm{h}_{\theta n}$\}, for $n=1,2,\dots$, to approximate the required transformation $\bm{h}_{\theta}$, and the result follows. For more specific explanations, we recommend \cite{NEURIPS2019_7ac71d43,huang2018neural}.

So far, we have shown that NF with the specific transformation $\bm{h}_{\theta}$ (RQ-NSF (AR)) can approximate the probability measure $\bm{\mu}_{\tilde{\bm{\varepsilon}}}$ or the probability density function $\bm{p}({\tilde{\bm{\varepsilon}}})$. Following a single-step error correction, the probability density function $\lim_{n \rightarrow \infty} \bm{p}(\hat{\bm{X}}_{T+1} + \tilde{\bm{\varepsilon}}_n)$ is identical to the target probability density function $\bm{p}(\bm{X}^{\delta t}_{T+1})$, that is, $\hat{{\bm{X}}}_{T+1} + \tilde{\bm{\varepsilon}}$ and $\bm{X}^{\delta t}_{T+1}$ have the same distribution. We follow the notation $\tilde{\bm{X}}_{T+1}:=\hat{\bm{X}}_{T+1} + \tilde{\bm{\varepsilon}}$ in Sec. \ref{sec: NF} to represent the RC prediction $\hat{\bm{X}}_{T+1}$ corrected by the NF single-step error modeling $\tilde{\bm{\varepsilon}}$. The prediction $\tilde{\bm{X}}_{T+1}$ is used as the next step input of RC and the RC output result is corrected by the NF error distribution again. Repeating this procedure, we can utilize RC-NF to achieve long-term predictions of stochastic dynamical systems.

In general, the single-step prediction result of RC strongly converges to the random variable at the corresponding moment of the target stochastic process. NF learns the error distribution to make it possible for the RC framework to predict the long-term evolution of stochastic dynamical systems in the sense of weak convergence. The difficulties of improving the weak convergence to the strong convergence may require further analysis of the structure of RC, which is out of the scope of this work. We refer to \cite{LiRNAVelocity} for further discussing the continuum limit of discrete transition probabilities, which partially answers our question. Despite this theoretic analysis, we draw some ``trajectories" in the experimental part (Sec. \ref{sec: experiments}), which seem reasonable and do reflect the dynamical behaviors of original stochastic dynamical systems. But, importantly, we cannot make sure these ``trajectories" converge to the true trajectories in the pathwise sense (almost surely). Loosely speaking, in a fixed hypothesis function space, RC can be thought of as the ``best'' approximation of the target stochastic process. Meanwhile, RC-NF somewhat expands the hypothesis space of RC, that is, RC-NF further reduces the approximation error between the RC model $\bm{\hat{X}}_t$ and the target stochastic process $\bm{X}_t^{\delta t}$.

\section{Experiments} \label{sec: experiments}

This section aims to demonstrate the effectiveness of our method, RC-NF, by exploring it through a series of experiments from different aspects. Specifically, our analysis considers linear/non-linear systems, Markov/non-Markov processes, stationary/non-stationary models, and an ergodic system. We summarize them as follows.
\begin{itemize}
    \item A fundamental experiment: the Ornstein-Uhlenbeck process (OU process) is a simple linear process to account for the ineffectiveness of conventional RC and the effectiveness of RC-NF.
    \item The capacity to capture complex dynamical behaviors:
    \begin{itemize}
        \item the Double-Well system (DW system) exhibits transition phenomena or noise-induced tipping phenomena in the presence of noise;
        \item the stochastic Van der Pol oscillator exhibits variances of state variables that fluctuate with time, due to whether the trajectories pass through a slow manifold or not;
        \item the dynamics of the stochastic mixed-mode oscillation (stochastic MMO) differ qualitatively from the corresponding deterministic system, even with small perturbations.
    \end{itemize}
    \item Exploration of stochastic differential equations with memory (i.e., SDDEs):
    \begin{itemize}
        \item a linear SDDE that can be solved explicitly to illustrate the capacity of the RC-NF framework on the fundamental non-Markovian stochastic system with time delays;
        \item the El Ni\~no-Southern Oscillation (ENSO) simplified model is a more complex time-delay system that is non-linear, non-Markovian, and non-stationary.
    \end{itemize}
    \item The stochastic Lorenz system, a chaotic system, is employed to demonstrate the substantial improvement of RC-NF over RC through the evaluation of multiple criteria (probability density function, maximum Lyapunov exponent, close returns, autocorrelation function, and cross-correlation function).
\end{itemize}

We use the Euler-Maruyama scheme to generate trajectory data for all experiments. For all systems, except for the stochastic Lorenz system, we assume that the numerical scheme step size $\delta t$ and the observation time step (or sampling time step) $\Delta t$ are equal. We choose 500 reservoir nodes for one-dimensional systems and 1000 reservoir nodes for multidimensional systems. The hyperparameters of RC chosen by BO for various experiments are displayed in Appendix \ref{appendix: hyper}.

We utilize RQ-NSF(AR) as the foundational transformation for normalizing flows and apply it twice through composition to obtain the final transformation $\bm{h}_\theta$ (\ref{eq: nf h_theta}). Our basic neural network consists of a fully connected neural network with two hidden layers, each containing 8 nodes. To train the model, we use the Adam optimizer with a learning rate of 0.005 and perform 500 iterations.

We list the common partial criteria across all experiments, along with their respective results in each experiment. The criteria for specific systems will be described in detail as needed. The Wasserstein distance is a distance function that measures the dissimilarity between probability distributions on a specific metric space. In the following, the Wasserstein 2-distance ($W_2$) is employed. The Kullback-Leibler divergence (KL divergence) is a measure of how one probability distribution $P$ differs from a second, reference probability distribution $Q$, denoted as $D_{KL}(P\|Q)$. We employ them to measure the discrepancy between the probability distributions of the data and the forecasts generated by our surrogate model, RC-NF. Tables \ref{tab: wasser} and \ref{tab: kl} show the Wasserstein distances $W_2(\bm{X}_t,\tilde{\bm{X}}_t)$ and KL divergences $D_{KL}(\tilde{\bm{X}}_t \| \bm{X}_t)$ at several predicted snapshots and the means of them across all predicted time.

\begin{table}[!ht] 
\caption{Wasserstein distances at several predicted snapshots and the means of Wasserstein distances across all predicted snapshots. Time $t_i$, $i = 1,\dots,5$ represent the different forecasting snapshots we select. Each of these times $10^{-3}$.}
\label{tab: wasser}
\centering
\begin{tabular}{l|cccccc} 
\bottomrule
\textbf{Experiments} & $t_1$ & $t_2$ & $t_3$ & $t_4$ & $t_5$ & \textbf{Mean} \\ 
\hline
OU process & $10.494$ & $11.959$ & $13.173$ & $14.349$ & $17.374$ & $12.385$\\
DW system & $5.3359$ & $5.4639$ & $ 5.4390$ & $5.1680$ & $5.1745$ & $5.4179$\\
Van del Pol oscillator & $8.1369$ & $8.7457$ & $9.8129$ & $11.198$ & $10.591$ & $9.4182$\\
stochastic MMO & $6.0690$ & $8.5117$ & $9.4959$ & $10.286$ & $11.074$ & $9.2720$\\
Linear SDDE & $8.9689$ & $7.9715$ & $8.6401$ & $7.1655$ & $7.9070$ & $7.9967$\\
ENSO simplified model & $7.8604$ & $5.7932$ & $6.4838$ & $7.9760$ & $6.1406$ & $6.4873$\\
\toprule
\end{tabular}
\end{table}

\begin{table}[!ht] 
\caption{KL divergences at a number of predicted snapshots and the means of KL divergences across all predicted snapshots. Time $t_i$, $i = 1,\dots,5$ represent our selected forecasting snapshots. Each of these times $10^{-3}$.}
\label{tab: kl}
\centering
\begin{tabular}{l|cccccc} 
\bottomrule
\textbf{Experiments} & $t_1$ & $t_2$ & $t_3$ & $t_4$ & $t_5$ & \textbf{Mean} \\ 
\hline
OU process & $2.3487$ & $2.2405$ & $5.2048$ & $3.4695$ & $4.1003$ & $2.6387$\\
DW system & $2.9123$ & $4.9392$ & $3.5568$ & $0.9387$ & $3.1716$ & $3.1879$\\
Van del Pol oscillator & $10.057$ & $24.635$ & $15.497$ & $23.901$ & $19.479$ & $14.887$\\
stochastic MMO & $9.7730$ & $27.429$ & $16.502$ & $23.862$ & $18.406$ & $17.692$\\
Linear SDDE & $1.4285$ & $1.6431$ & $1.7739$ & $2.0469$ & $3.0536$ & $1.7927$\\
ENSO simplified model & $3.9270$ & $1.5115$ & $7.4493$ & $3.5715$ & $2.2184$ & $2.8313$\\
\toprule
\end{tabular}
\end{table}

\subsection{Ornstein-Uhlenbeck process: a fundamental experiment for long-term prediction}\label{sec: exper: ou}

We describe the Ornstein-Uhlenbeck process (OU process) \cite{duan2015introduction} as follows:
\begin{equation}\label{eq: ouprocess}
    d X_t = b_0(\mu_0-X_t) dt + g dB_t,\quad \text{for}\  t \geq 0,
\end{equation}
where $b_0>0$ indicates the rate of mean reversion, $\mu_0 \in \mathbb{R}$ is the mean value, and the constant $g>0$ is the volatility. The stochastic process $B_t$ is a scalar Brownian motion. For a given initial state $X_0= x$, the analytical solution of the OU process is $X_t = (1-e^{-b_0t})\mu_0 + xe^{-b_0t} + g \int_0^te^{-b_0(t-s)}dB_s$. More specifically, $X_t \sim N(\mu_0 + (x-\mu_0)e^{-b_0t}, g^2\int_0^te^{-2b_0(t-s)}ds)$. The solution $X_t$ is a stationary process that admits a Gaussian distribution $N(\mu_0, g^2/2b_0)$ when $t \rightarrow \infty$.

Assuming that the parameters of the system are $b_0 = 0.15$, $\mu_0=1$, and $g=1$. We use the Euler-Maruyama scheme for solving this system with time step $\delta t=0.01$, and the observed data are recorded at every time step, i.e., $\Delta t=0.01$. The initial value is $X_0 = 0$. We generate 1000 trajectories with a total length of 4000, where the training length $T$, verification length $T_{valid}$, and prediction length $T_{test}$ are 2000, 100, and 1900 respectively, and the first 100 steps of training are used for warm-up.

First, Fig. \ref{fig: ou rc and rc-nf}  shows the significant difference between RC and RC-NF in their ability to predict the OU process. Since the OU process follows Gaussian distribution at arbitrary fixed time $t$, we can build confidence intervals using the ``three-sigma rule," which correspond to 68\%, 95\%, and 99.7\% confidence intervals, respectively (magenta dashed lines in Fig. \ref{fig: ou rc and rc-nf}, the middle line represents the mean $\mu_0 + (x-\mu_0)e^{-b_0t}$). The solid blue line denotes the sample mean, and the shaded blue portions from dark to light represent the intervals calculated from the trajectory data, which are made up of 68\%, 95\%, and 99.7\% of the data surrounding the median value for each time $t$. Rolling predictions using RC and RC-NF produce the red dotted dashes and shaded areas. RC alone is unable to provide an accurate forecast of the OU process distribution in the long term, while the RC-NF framework is capable of doing so.
\begin{figure}[!ht]
  \centering
   \subfigure[Rolling predictions of RC]{
    \label{fig: subfig: ou rc} 
    \includegraphics[width=0.49\textwidth]{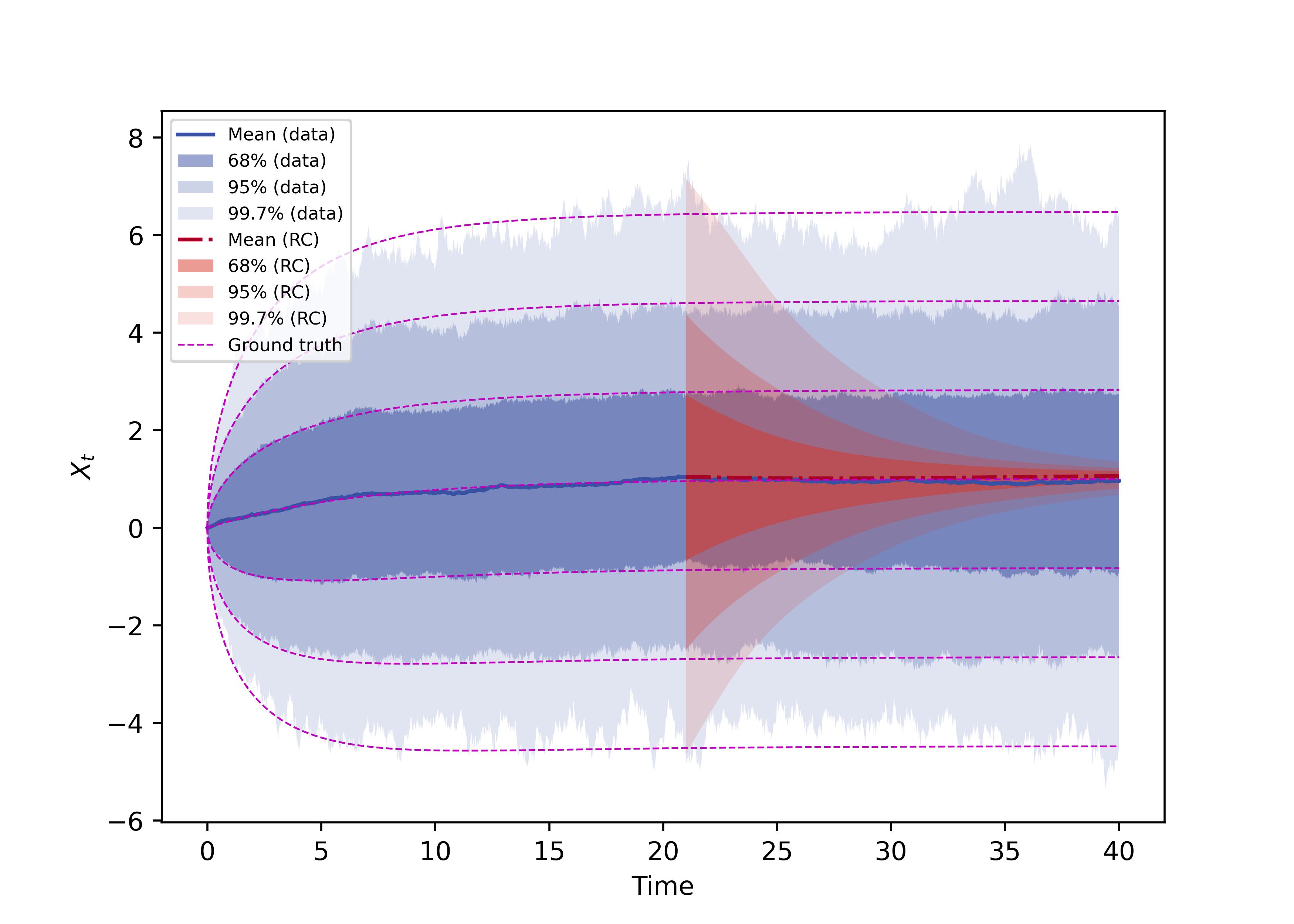}}
  \hspace{-0.1in}  
  \subfigure[Rolling predictions of RC-NF]{
    \label{fig: subfig: ou rc-nf} 
    \includegraphics[width=0.49\textwidth]{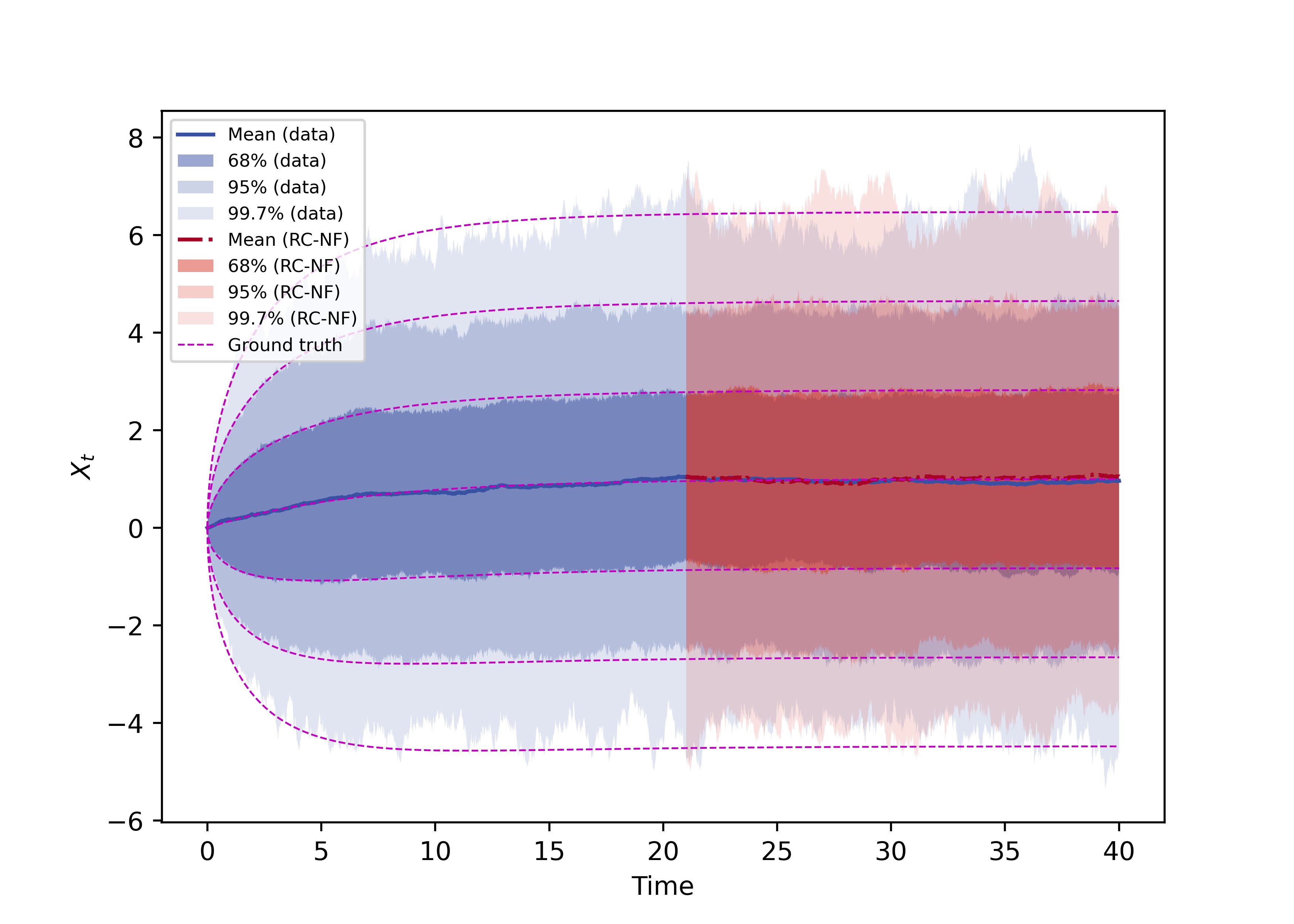}}
  \caption{Rolling predictions of RC and RC-NF. The magenta dashed lines represent the theoretical true values of the means and confidence intervals. The solid blue lines denote the sample means, and the shaded blue portions from dark to light represent the intervals calculated from the trajectory data. Rolling predictions using RC and RC-NF produce the red dotted dashes and shaded areas. (a) Rolling predictions of RC. (b) Rolling predictions of RC-NF.}
  \label{fig: ou rc and rc-nf} 
\end{figure}

To further confirm the effectiveness of the RC-NF, Fig. \ref{fig: subfig: ouheatmap} depicts the probability density functions of the trajectory data and the rolling predictions of RC-NF on the test dataset. Fig. \ref{fig: subfig: ou wasser and kl} displays Wasserstein distances and KL divergences between the reference distributions and the estimated distributions at different times on the test dataset; the red dashed lines represent the selected snapshots, whose probability density functions and specific values are displayed in Fig. \ref{fig: subfig: ou pdf select}, Tables \ref{tab: wasser} and \ref{tab: kl}, respectively. The quantity of trajectory data is also a factor in the discrepancy between RC-NF predictions and trajectory data. Despite all of this, the proposed RC-NF has shown to be effective in achieving long-term predictions of the OU process, as demonstrated by these results. Even though the OU process is a linear, Markov, stationary process, traditional RC fails, and the success of RC-NF further makes us wonder how RC-NF will behave on more complex systems.
\begin{figure}[!ht]
  \centering
  \hspace{-0.2in} 
  \subfigure[PDFs over time. Reference (left) and RC-NF (right)]{
    \label{fig: subfig: ouheatmap} 
    \includegraphics[height=1.98in]{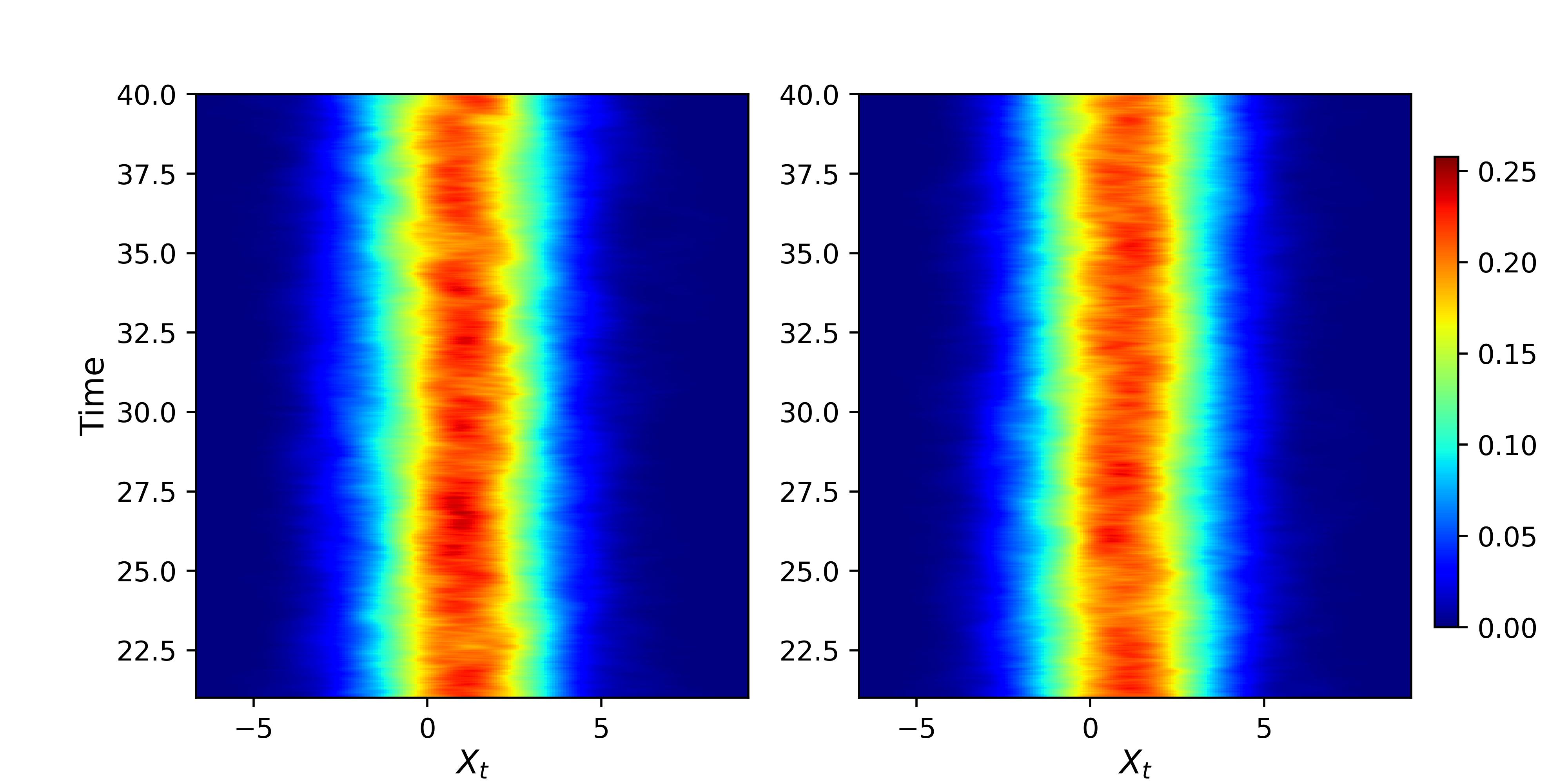}}
  \hspace{-0.1in}
  \subfigure[$W_2(\bm{X}_t,\tilde{\bm{X}}_t)$ and $D_{KL}(\tilde{\bm{X}}_t \| \bm{X}_t)$]{
    \label{fig: subfig: ou wasser and kl} 
    \includegraphics[height=2.04in]{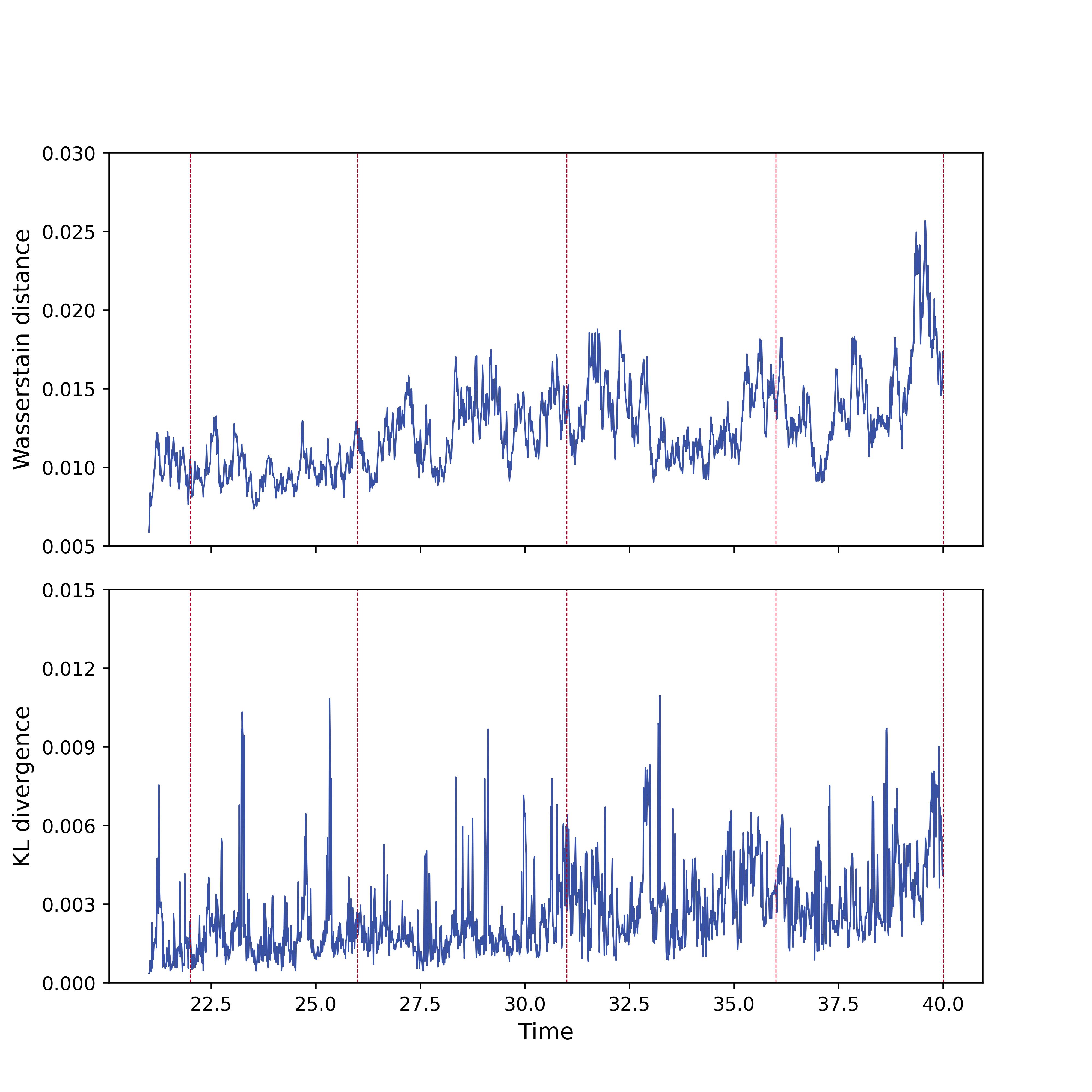}}
  \makebox[\textwidth][c]{\subfigure[PDFs. References (solid blue lines) and RC-NF results (red dashed lines) for several snapshots selected on the test dataset]{
    \label{fig: subfig: ou pdf select} 
    \includegraphics[width=1.12\textwidth]{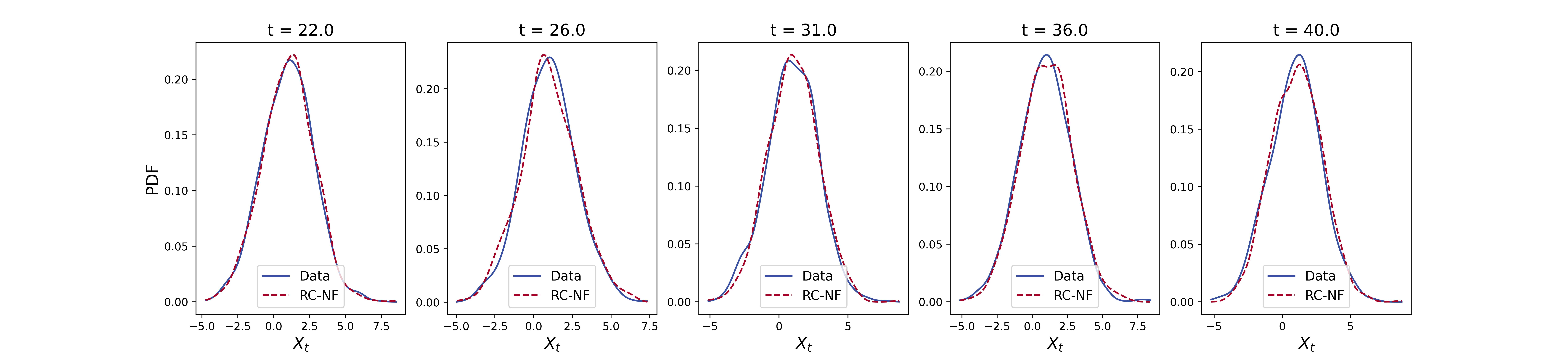}} } 
  \caption{Results of the OU process. PDF is the abbreviation for probability density function. PDFs of the trajectory data and the RC-NF rolling predictions are shown in (a). (b) Wasserstein distance (top) and KL divergence (bottom), the red dashed lines represent snapshots where we draw the PDFs in (c) and display the values in Tables \ref{tab: wasser} and \ref{tab: kl}.}
  \label{fig: ou heatmap,wass, kl, pdf} 
\end{figure}

\subsection{Exploration of complex dynamical behaviors}\label{sec: exper: richer dynamical behaviors}

We explore three experiments exhibiting complex dynamical behaviors. For the Double-Well system, the trajectory of this system possibly transitions between two metastable states under the influence of noise, while the trajectory of the corresponding deterministic system will be attracted to a certain stable point. Due to the presence of relaxation oscillation in the stochastic Van der Pol system, the variances of the state variables display time dependency. The trajectory passing through the folding singularity displays a pattern of stochastic mixed-mode oscillations even in the presence of small perturbations. We present the performance of RC-NF in these three experiments in turn.

\noindent$\bullet$ \textbf{Double-Well system: transition rate} \label{sec: exper: dw}

We consider a DW system \cite{duan2015introduction, HARLIM2021109922, ZHU2023111819} defined on $\mathbb{R}$, 
\begin{equation} \label{eq: dw}
    d X_t = (X_t-X_t^3) dt + g dB_t, \quad \text{for}\  t \geq 0,
\end{equation}
where the constant $g$ is a positive diffusion coefficient and $B_t$ is a scalar Brownian motion. This is a bistable system with its determined counterpart having two stable fixed points ($X_t=-1$, $X_t=1$) and one unstable fixed point ($X_t=0$). The trajectory from this system possibly transitions between two metastable states under the influence of noise, which is not possible in the absence of noise.

Assume that $g=0.5$ (larger noise makes it difficult to calculate the subsequent transition rate). We follow the same time step size for this system as for the OU process, that is, $\delta t = \Delta t = 0.01$. The initial values $X_0$ are drawn from a uniform distribution on $[-1.5,1.5]$. We generate 2000 trajectories with a total length of 4000 using the Euler-Maruyama scheme, where the training length $T$, verification length $T_{valid}$, and prediction length $T_{test}$ are 2000, 100, and 1900 respectively. The first 100 steps of the training dataset are used for warm-up. 

Fig. \ref{fig: subfig: dwheatmap} depicts the probability density functions of the trajectory data and the rolling predictions of RC-NF on the test dataset. Rolling predictions of RC-NF reproduce the bimodal probability density function of the DW system. The two peaks correspond to two metastabilities of the system. Fig. \ref{fig: subfig: dw wasser and kl} displays Wasserstein distances and KL divergences between the reference distributions and the estimated distributions at each time $t$ in the testing phase; the red dashed lines represent the selected snapshots, whose probability density functions and specific values are displayed in Fig. \ref{fig: subfig: dw pdf select}, Tables  \ref{tab: wasser} and \ref{tab: kl}, respectively. The results demonstrate the effectiveness of our proposed RC-NF framework and its ability to predict the DW system over the long term.
\begin{figure}[!ht]
  \centering
  \hspace{-0.2in} 
  \subfigure[PDFs over time. Reference (left) and RC-NF (right)]{
    \label{fig: subfig: dwheatmap} 
    \includegraphics[height=1.98in]{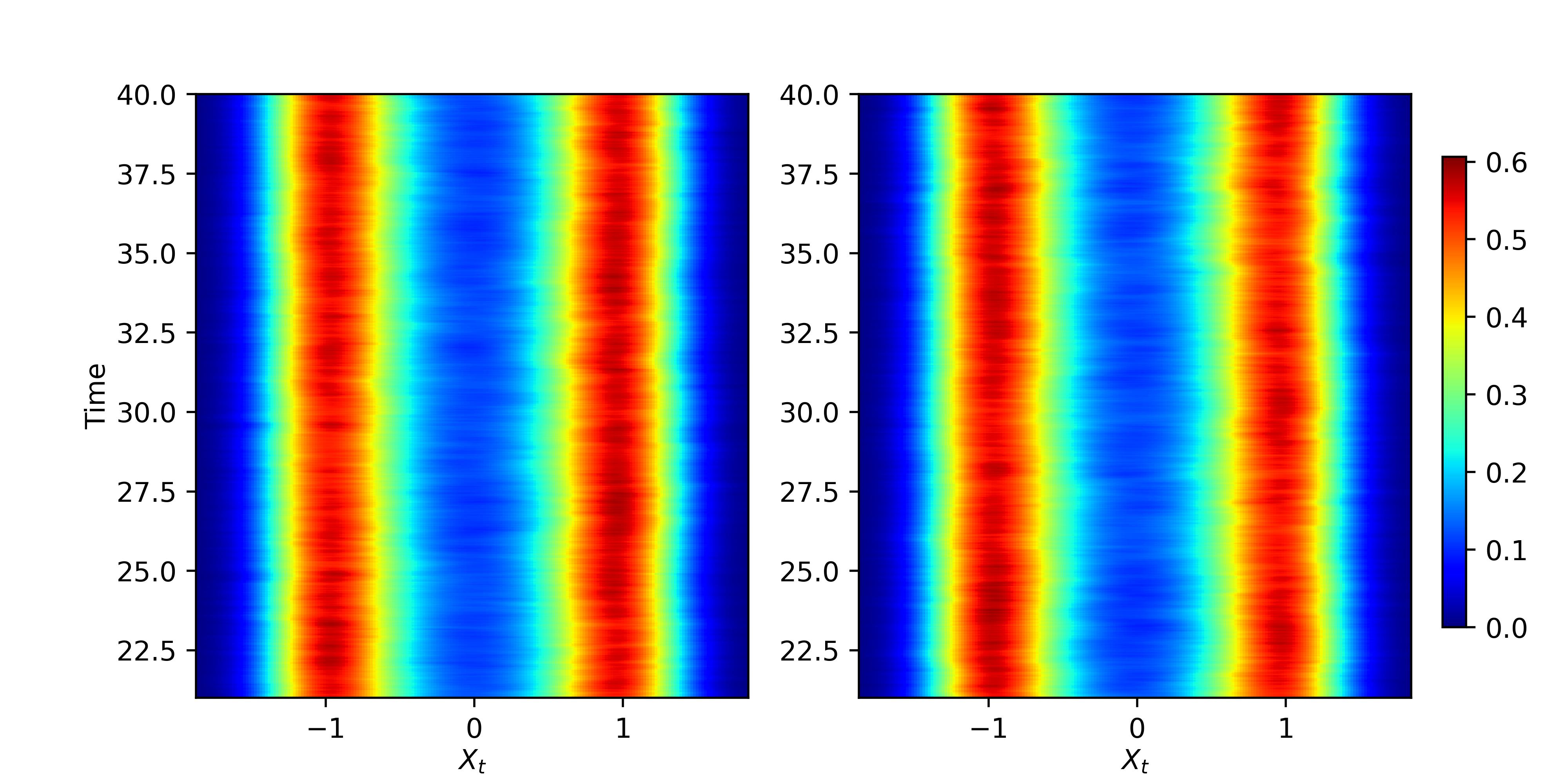}}
  \hspace{-0.1in}
  \subfigure[$W_2(\bm{X}_t,\tilde{\bm{X}}_t)$ and $D_{KL}(\tilde{\bm{X}}_t \| \bm{X}_t)$]{
    \label{fig: subfig: dw wasser and kl} 
    \includegraphics[height=2.04in]{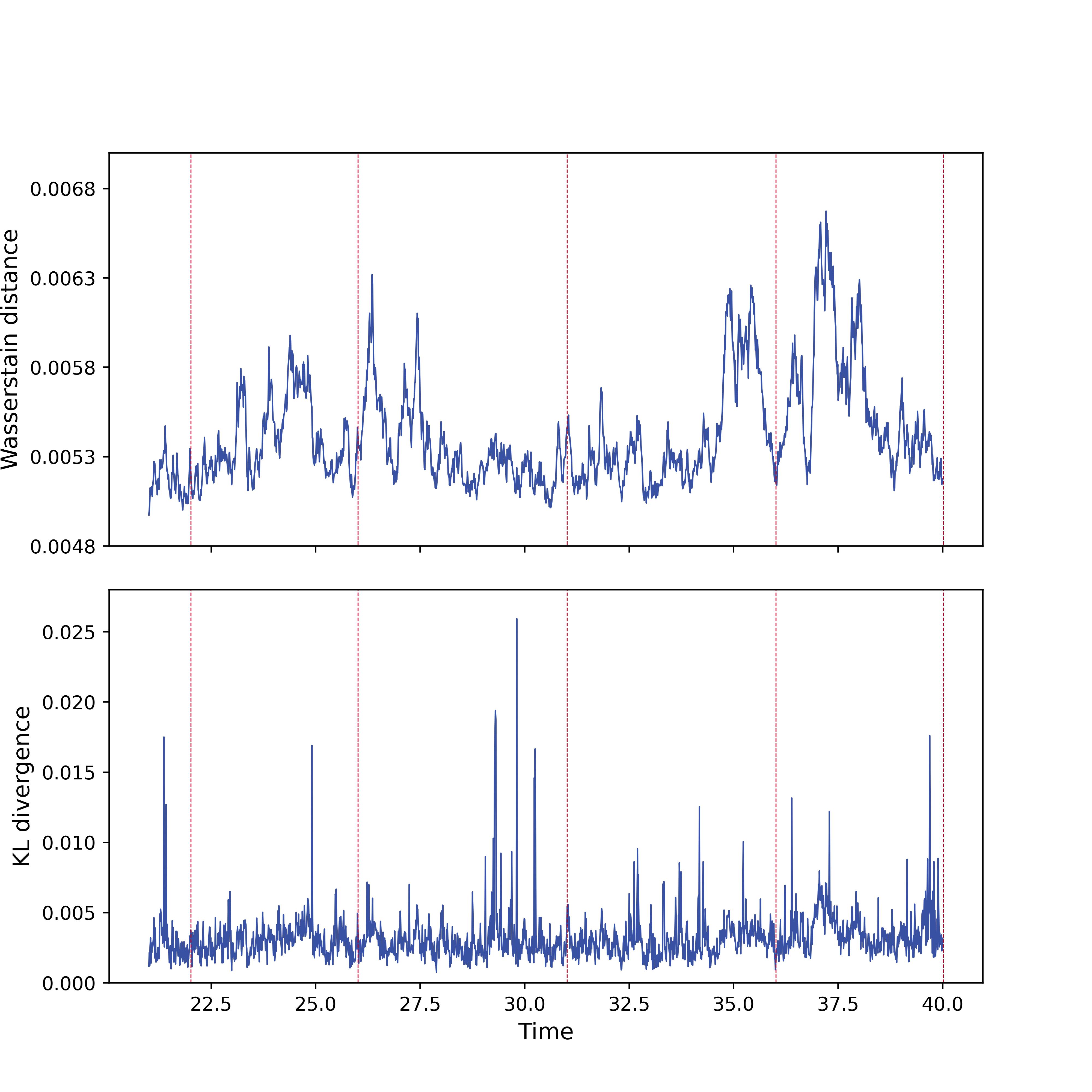}}
  \makebox[\textwidth][c]{\subfigure[PDFs. References (solid blue lines) and RC-NF results (red dashed lines) for several snapshots selected on the test dataset]{
    \label{fig: subfig: dw pdf select} 
    \includegraphics[width=1.12\textwidth]{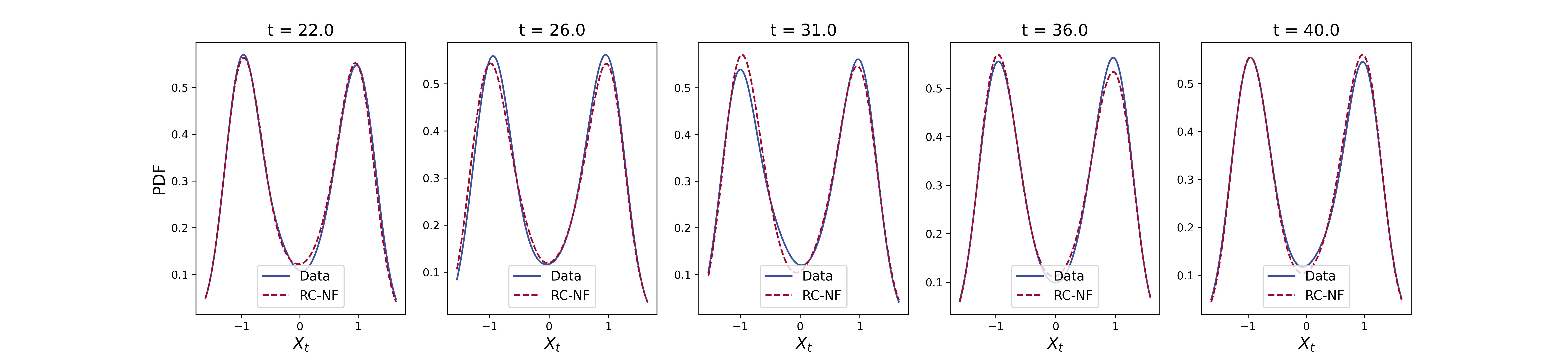}} } 
  \caption{Results of the DW system. PDFs of the trajectory data and the RC-NF rolling predictions are shown in (a). (b) Wasserstein distance (top) and KL divergence (bottom), the red dashed lines represent snapshots where we draw the PDFs in (c) and display the values in Tables \ref{tab: wasser} and \ref{tab: kl}.}
  \label{fig: dw heatmap,wass, kl, pdf} 
\end{figure}

The transition phenomenon of trajectories between two metastabilities in the DW system is a rare event, and it is also known as noise-induced tipping \cite{ashwin2012tipping}. We use the trained RC-NF model to generate new sample trajectories to verify whether it can approximate the transition rate of the rare event. Following \cite{ZHU2023111819, Bolhuis2002annurev.physchem, dellago1998transition}, we first separate the phase space for $X_t$ into $A = (\infty, 0]$ and $B = (0,+\infty)$ regions in order to calculate the transition rate between the two states. The time correlation function $C_{AB}(t)$ is defined as follows:
\begin{equation}\label{eq: Cab}
\frac{C_{AB}(t)}{C_{A}} = \frac{\langle I_A(X_0) I_B(X_t) \rangle}{\langle I_A(X_0) \rangle},
\end{equation}
where $I_A(\cdot)$ is an indicator function satisfying $I_A(X_t) = 1$ if $X_t \in A$, and $I_A(X_t) = 0$ if $X_t \notin A$. The indicator function $I_B(X_t)$ is defined similarly. The symbol $\langle \cdot \rangle$ denotes the ensemble average. When the system is originally in region $A$, the ratio (\ref{eq: Cab}) represents the probability of finding the system in region $B$ after time $t$. The ratio $C_{BA}(t)/C_B$ can be defined in a similar way. Furthermore, the transition rate from $A$ to $B$ can be calculated as:
\begin{equation}\label{eq: kab}
    k_{AB} = \frac{d}{dt}\frac{C_{AB}(t)}{C_{A}}, \quad \text{for}\ \tau_{mol} < t \ll \tau_{rxn},
\end{equation}
where $\tau_{mol}$ is a short transient time \cite{Bolhuis2002annurev.physchem,dellago1998transition} and $\tau_{rxn} = 1/(k_{AB}+k_{BA})$ is the exponential relaxation time. The rate $k_{AB}$ is actually the slope of the time correlation function $C_{AB}(t)$ for $\tau_{mol} < t \ll \tau_{rxn}$. For the calculation of transition rate $k_{AB}$, we set the initial value $X_0=-1$. The first 100 steps of data are used to warm up. A total of $10^4$ trajectories are generated by the Euler-Maruyama scheme and the trained RC-NF model, respectively. Using equations (\ref{eq: Cab}) and (\ref{eq: kab}), the transition rates of trajectory data (denoted as $k_{AB}$) and of generating trajectories by the RC-NF model (denoted as $\tilde{k}_{AB}$) can be calculated. The transition rate $k_{BA}$ is calculated similarly, except that the initial value $X_0=1$. Fig. \ref{fig: dw transition rate} shows the calculated results of $C_{AB}(t)/C_{A}$ and $C_{BA}(t)/C_B$, and the transition rates are determined by linear fitting. Specifically, $k_{AB} = 1.4684 \times 10^{-2}$, $\tilde{k}_{AB} = 1.4475 \times 10^{-2}$, $k_{BA} = 1.4812 \times 10^{-2}$, $\tilde{k}_{BA} = 1.5322 \times 10^{-2}$. The errors of transition rates between the trajectory data and the trajectories generated by the RC-NF model are both less than 0.001. Additionally, the exponential relaxation time $\tau_{rxn} \approx 33$s. In order to satisfy the condition $\tau_{mol} < t \ll \tau_{rxn}$ so that $k_{AB}$ and $k_{BA}$ remain constant, we choose the time interval $t \in [5,25]$ to calculate $k_{AB}$ and $k_{BA}$, and display the results in Fig. \ref{fig: dw transition rate}. The RC-NF model successfully reproduces the transition rates of the original stochastic DW system.
\begin{figure}[!ht]
    \centering
    \subfigure[Transition rate $k_{AB}$]{
    \includegraphics[width=0.49\textwidth]{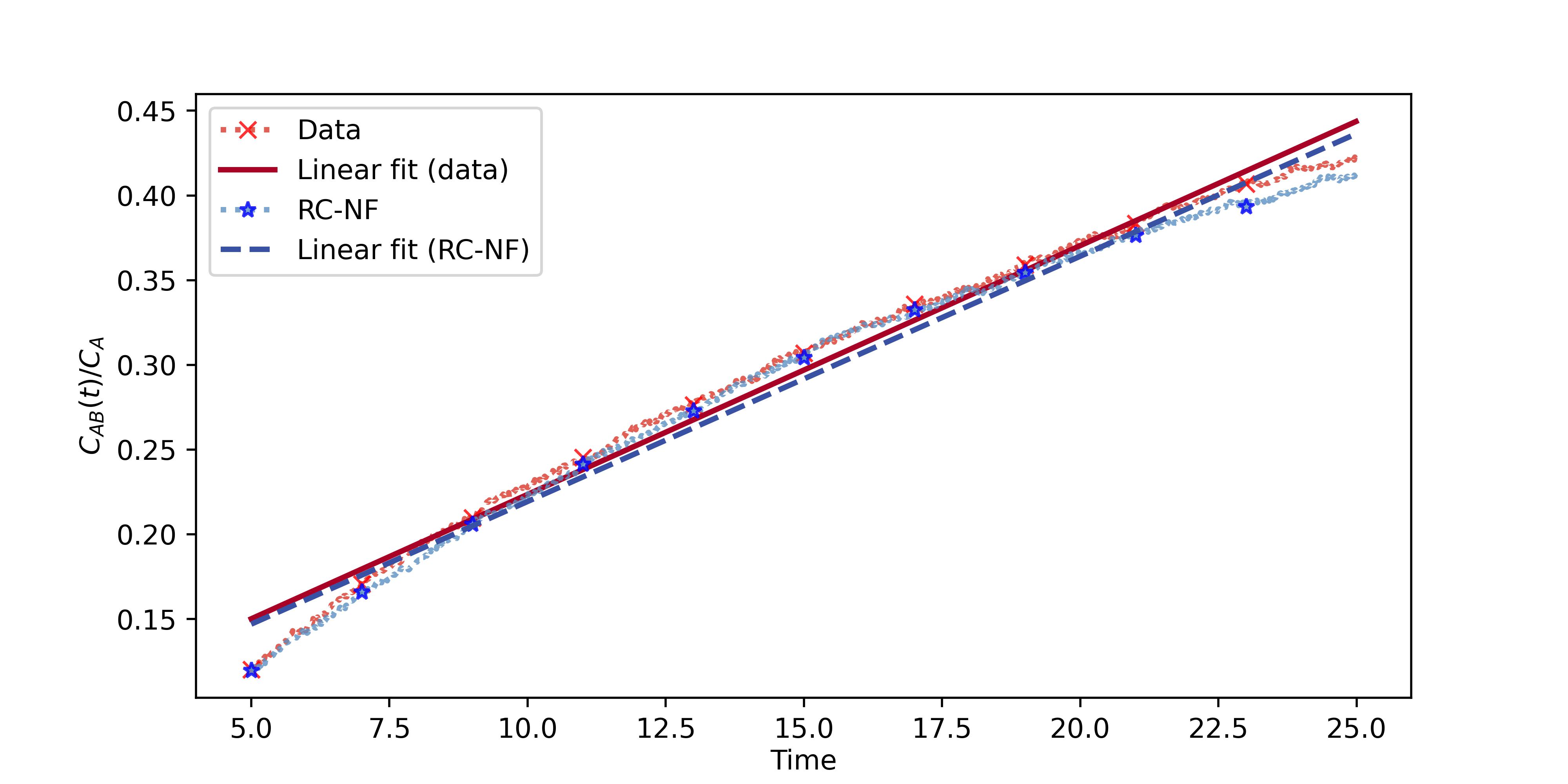}
    \label{fig: subfig: dwkab}
    }
    \hspace{-0.2in}
    \subfigure[Transition rate $k_{BA}$]{
    \includegraphics[width=0.49\textwidth]{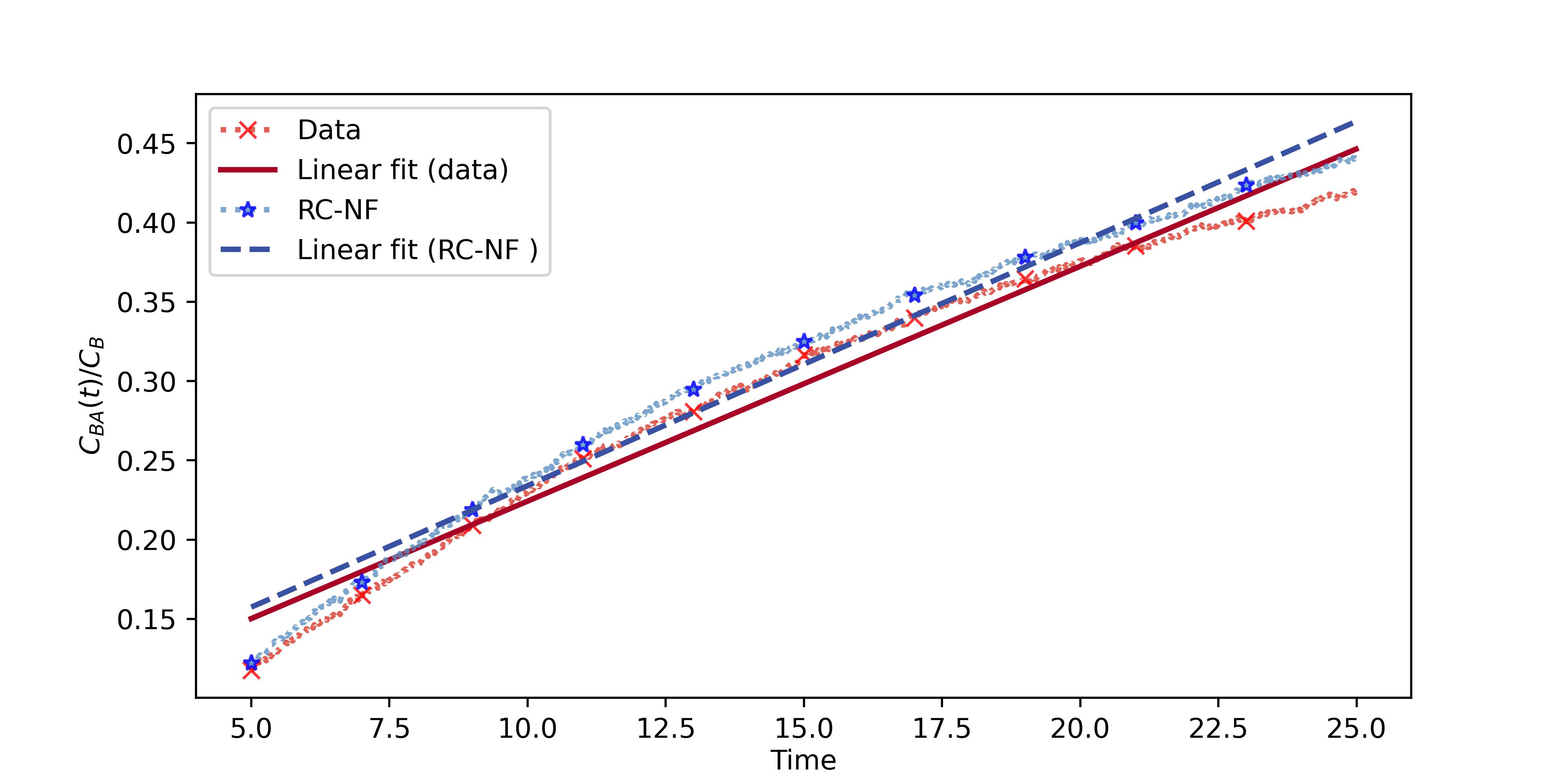}
    \label{fig: subfig: dwkba}
    }
    \caption{Transition rates (a) $k_{AB}$ and (b) $k_{BA}$ are obtained by trajectory data and RC-NF model generating trajectories. The red cross-dot lines are calculated from trajectory data, and the linear fitting produces the solid red lines. The blue asterisk dot lines are calculated from the RC-NF generating trajectories, while linear fitting produces the dashed blue lines.}
    \label{fig: dw transition rate}
\end{figure}

The RC-NF model can be successfully employed for DW systems, whether evaluated from the probability density function or the transition rate. In other words,  RC-NF could serve as a surrogate model for analyzing stochastic DW systems, allowing for a deeper understanding of their internal mechanisms and enabling early warning of rare events or critical transition phenomena.

\noindent$\bullet$ \textbf{Stochastic Van der Pol oscillator: the variances of the state variables fluctuate over time} \label{sec: exper: vdp}

We consider the Van der Pol system with additive noise on $\mathbb{R}^2$:

\begin{equation} \label{eq: vdp}
\begin{cases}
    dX_t = \mu_0(X_t - X_t^3/3 - Y_t)dt + g^1 dB^1_t,\\ 
    dY_t = (X_t/\mu_0)dt + g^2 dB^2_t,
\end{cases}
\end{equation}
where $\mu_0 \in \mathbb{R}$. The positive constants $g^1$, $g^2$ are diagonal elements of the diffusion coefficient matrix $\bm{g}$ and $[B^1_t, B^2_t]^T$ is a two-dimensional standard Brownian motion. We set $\mu_0 = 8$, $g^1=g^2=0.1$. The corresponding deterministic dynamical system of the model (\ref{eq: vdp}) has a unique, stable limit cycle for each $\mu_0 > 0$. When $\mu_0$ is large, the oscillator behaves in a slow buildup and fast release cycle, slowly moving up (moving down) the right branch (left branch) of the cubic curve $Y_t = X_t - X^3_t/3$ with time $O(\mu_0)$ and quickly moving to the left branch (right branch) of the cubic curve with time $O(\mu_0^{-1})$. This kind of oscillation is also called relaxation oscillation. The left and right branches of this limit cycle are two slow manifolds of this system. The limit cycle and the cubic curve $Y_t = X_t - X^3_t/3$ in Fig. \ref{fig: subfig: vdpmanifoldpdf} provide an intuitive understanding. In the presence of small additive noise, the variance of $X_t$ is small when the trajectory passes the slow manifold, and large when it ``jumps" between two slow manifolds. 

Following the previous experiments, we set the time step for this system to $\delta t = \Delta t = 0.01$. The initial value is $[X_0, Y_0]^T = [0.5,0]^T$. We generate 1000 trajectories with a total length of 4000 using the Euler-Maruyama scheme, where the training length $T$, verification length $T_{valid}$, and prediction length $T_{test}$ are 2000, 100, and 1900 respectively, and the first 100 steps of training are used for warm-up.

The results of rolling predictions of the RC-NF model are shown in Fig. \ref{fig: vdp heatmap,wass, kl, pdf}. Fig. \ref{fig: subfig: vdpheatmap} depicts the probability density functions of the trajectory data and the rolling predictions of RC-NF on the test dataset. Both exhibit similar evolutionary patterns, which means that these trajectories move around the limit cycle over time and spend a long time near the slow manifolds. Fig. \ref{fig: subfig: vdpmanifoldpdf} shows the distributions in two-dimensional space of data and prediction results of RC-NF in the test phase. Both indicate that a large number of data points and predicted values are gathered near the limit cycle, and the peaks are located on the slow manifolds. Fig. \ref{fig: subfig: vdp wasser and kl} displays Wasserstein distances and KL divergences between the reference distributions and the estimated distributions at different times $t$; the red dashed lines represent the selected snapshots, whose probability density functions and specific values are displayed in Fig. \ref{fig: subfig: vdp pdf select}, Tables  \ref{tab: wasser}  and \ref{tab: kl}, respectively. The empirical variances of the state variables $[X_t, Y_t]^T$ during the testing phase are shown in Fig. \ref{fig: subfig: vdpvariance}. The data and rolling predictions of RC-NF exhibit synchronous oscillations. These results demonstrate the effects of our proposed RC-NF framework.
\begin{figure}[!ht]
  \hspace{0.1in}
  \centering 
  \subfigure[PDFs over time. From left to right: Reference ($X_t$), RC-NF ($X_t$), Reference ($Y_t$) and RC-NF ($Y_t$)]{
    \label{fig: subfig: vdpheatmap} 
    \includegraphics[height=1.65in]{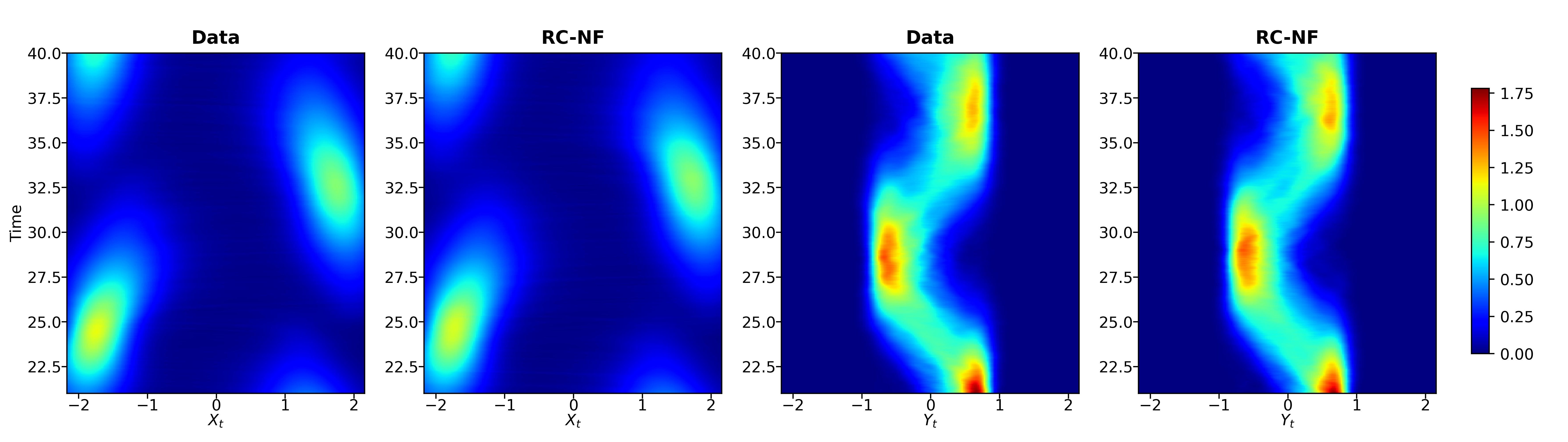}}
  \subfigure[Distributions of ($X_t$, $Y_t$) in testing. Reference (left) and RC-NF (right)]{
    \label{fig: subfig: vdpmanifoldpdf} 
    \includegraphics[height=1.98in]{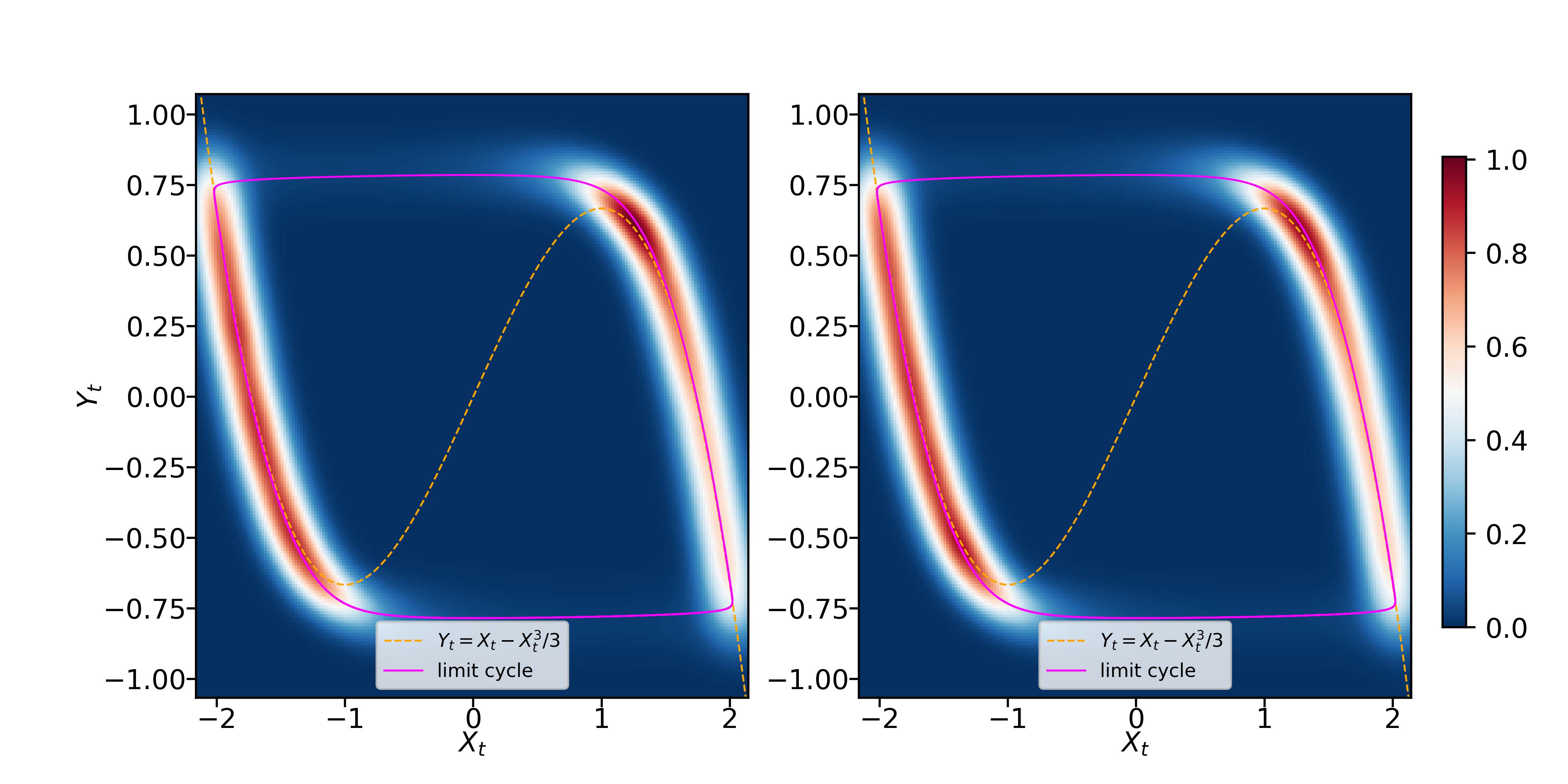}}
  \hspace{-0.2in}
  \subfigure[$W_2(\bm{X}_t,\tilde{\bm{X}}_t)$ and $D_{KL}(\tilde{\bm{X}}_t \| \bm{X}_t)$]{
    \label{fig: subfig: vdp wasser and kl} 
    \includegraphics[height=2.04in]{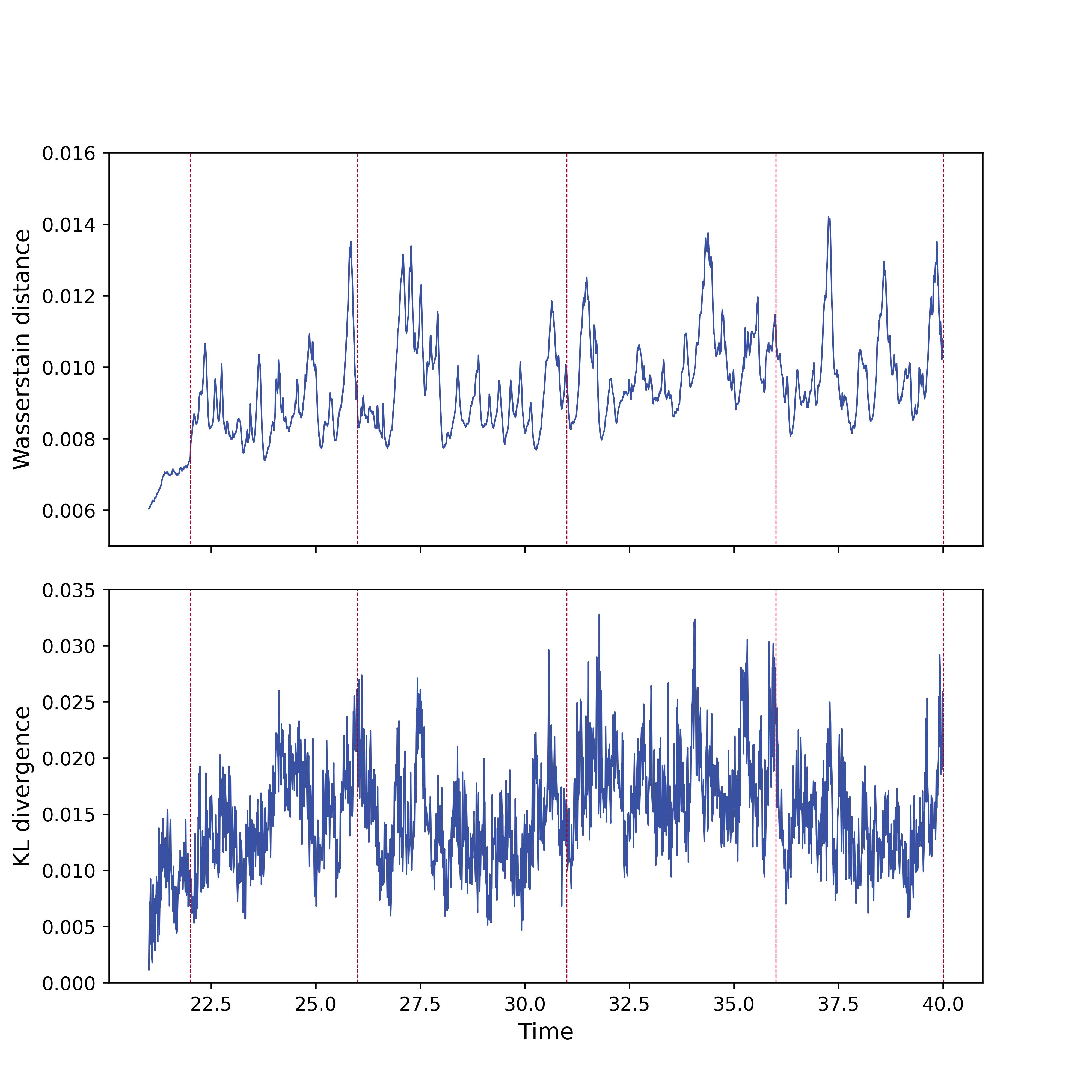}}
  \makebox[\textwidth][c]{\subfigure[PDFs. References (solid blue lines) and RC-NF results (red dashed lines) for several snapshots selected on the test dataset]{
    \label{fig: subfig: vdp pdf select} 
    \includegraphics[width=1.12\textwidth]{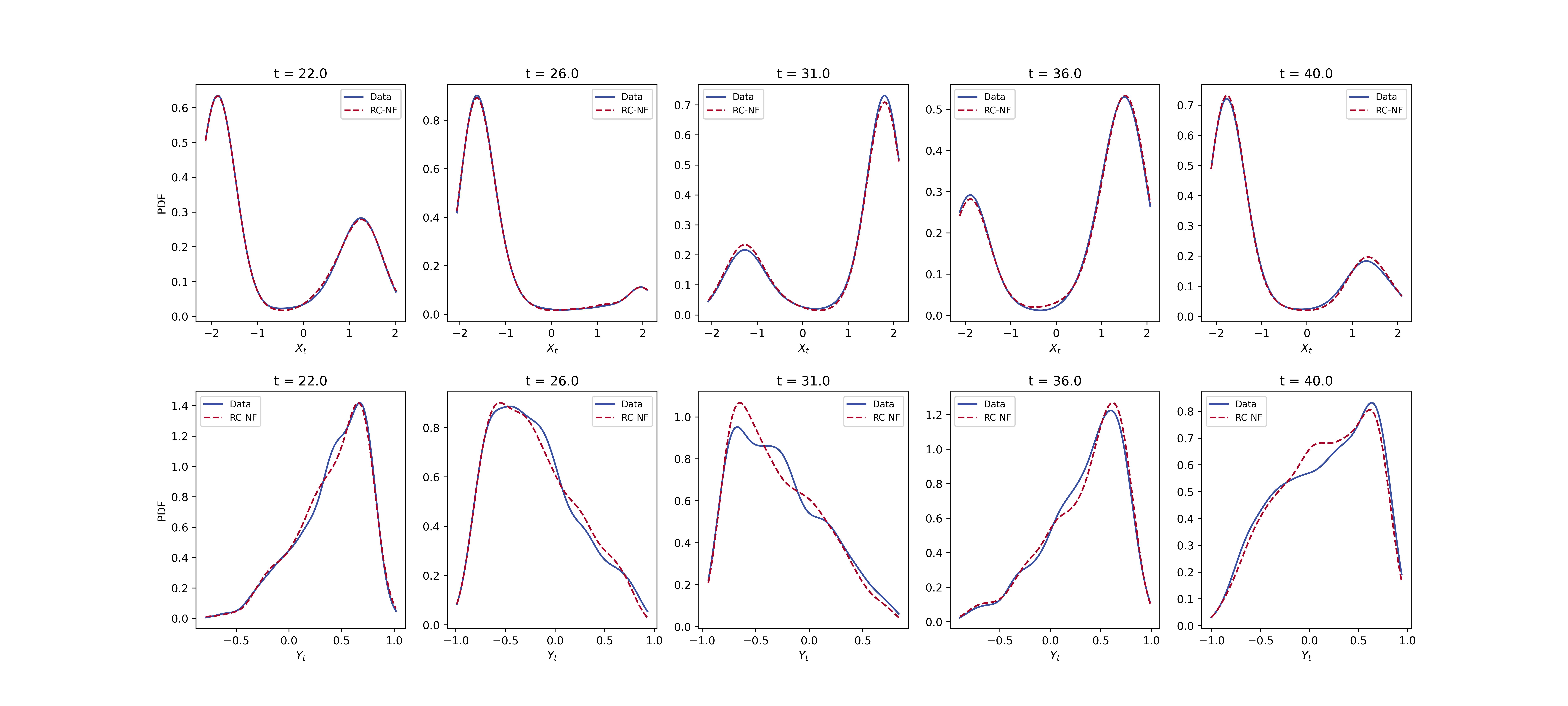}} } 
  \caption{Results of stochastic Van der Pol oscillator. PDFs of the trajectory data and the RC-NF rolling predictions are shown in (a). (b) Distributions of ($X_t$, $Y_t$) in testing. The solid magenta line represents the limit cycle, and the dashed yellow line represents the cubic nullcline $Y_t = X_t-X^3_t/3$. (c) Wasserstein distance (top) and KL divergence (bottom), where the notations $\bm{X}_t:= [X_t, Y_t]^T $ and $\tilde{\bm{X}}_t:= [\tilde{X}_t, \tilde{Y}_t]^T$. The red dashed lines represent snapshots where we draw the PDFs in (d) and display the values in Tables \ref{tab: wasser} and \ref{tab: kl}.}
  \label{fig: vdp heatmap,wass, kl, pdf} 
\end{figure}

Taking the initial value $[X_0, Y_0]^T = [1.5,1.0]^T$. Using the Euler-Maruyama scheme and the RC-NF model respectively, we generate a long trajectory with a length of 20500 and set observation time step $\Delta t=0.01$. The first 500 steps of the dataset are used to warm up for RC-NF. The results are displayed in Figs. \ref{fig: subfig: vdpx} and \ref{fig: subfig: vdpv}, both of which show trajectories of period-like motions.
\begin{figure}[!ht]
    \centering
    \subfigure[Empirical variances in test phase]{
    \includegraphics[width=0.33\textwidth]{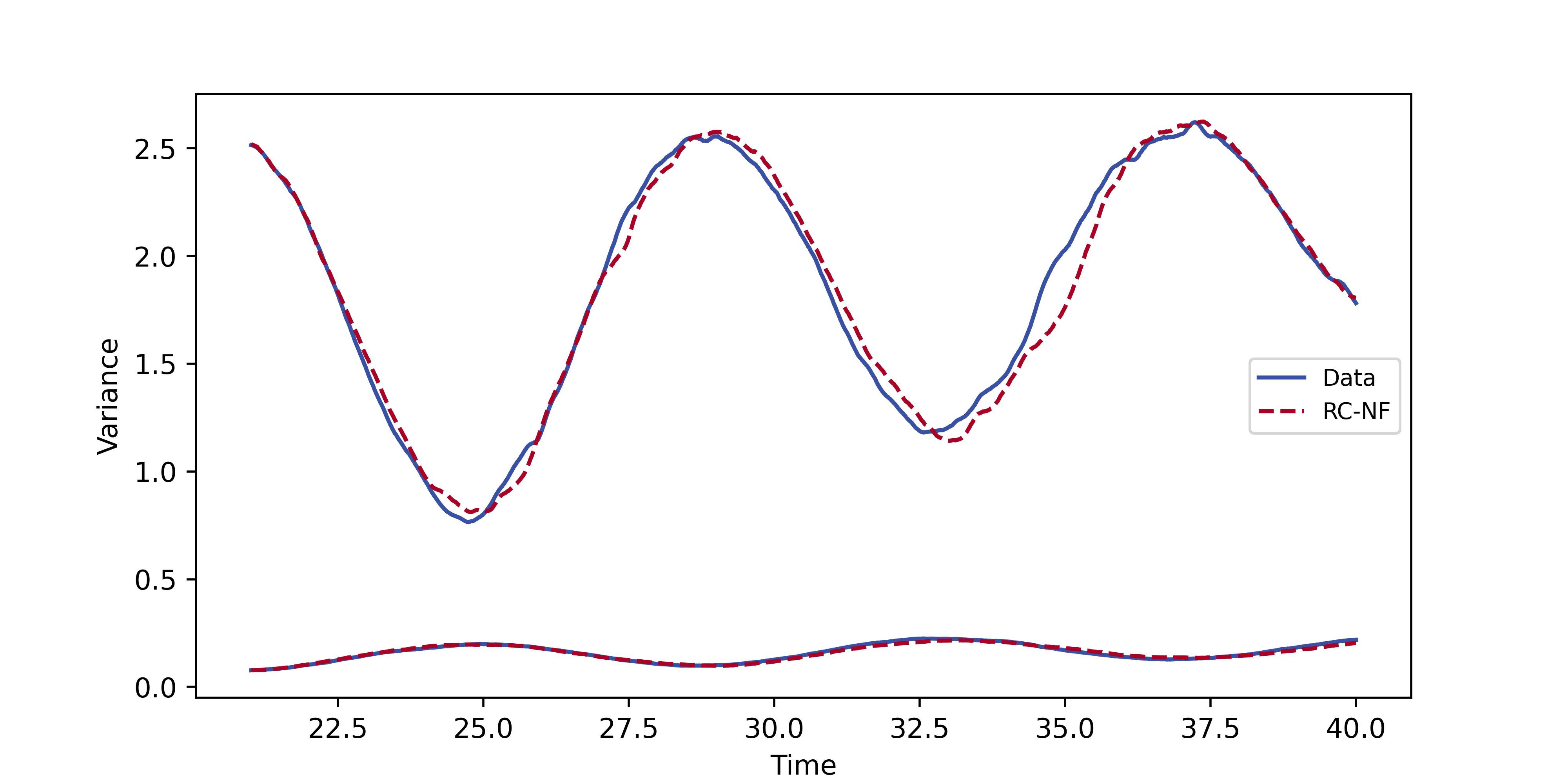}
    \label{fig: subfig: vdpvariance}
    }
    \hspace{-0.2in}
    \subfigure[The state variable $X_t$]{
    \includegraphics[width=0.33\textwidth]{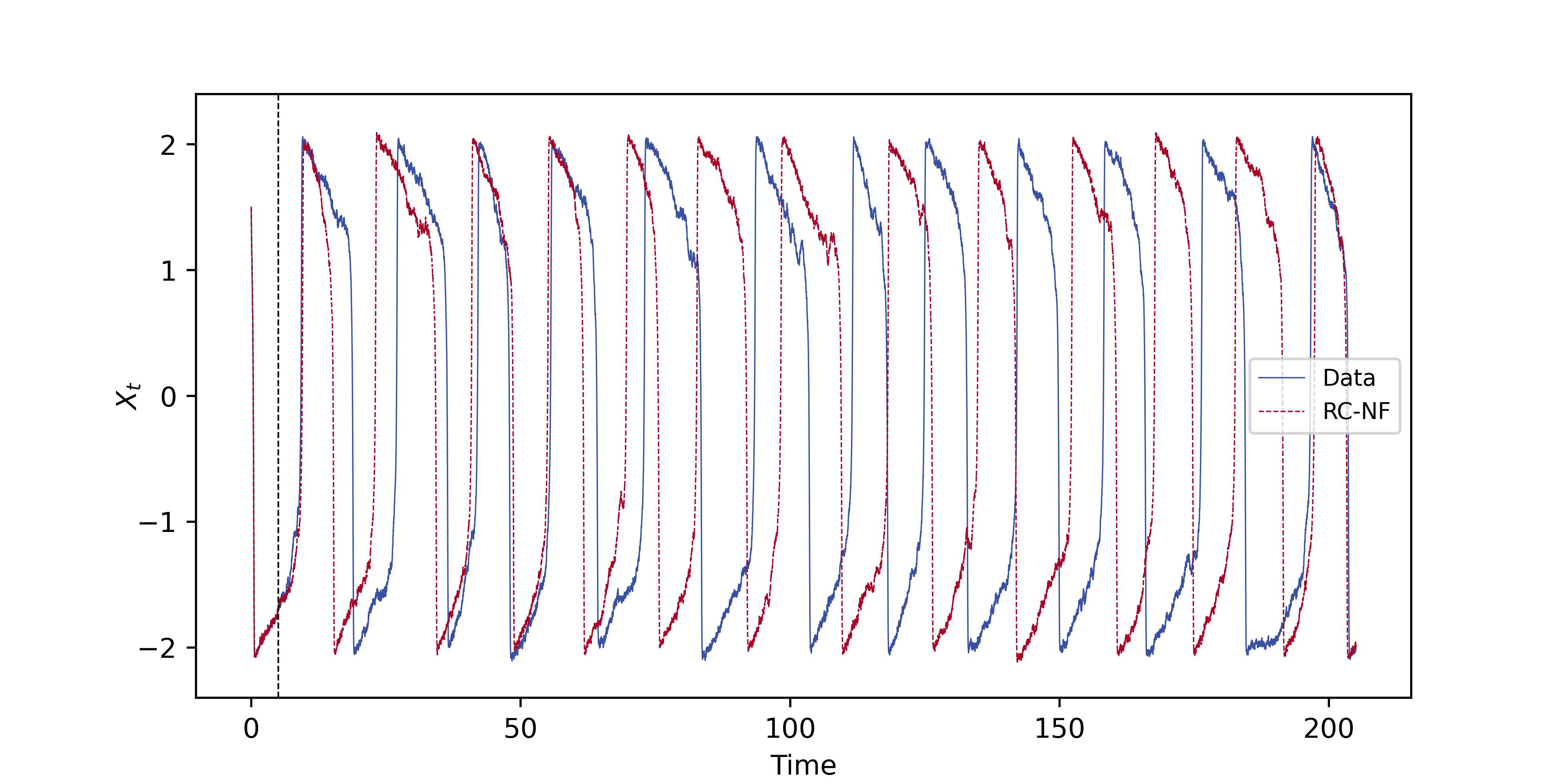}
    \label{fig: subfig: vdpx}
    }
    \hspace{-0.2in}
    \subfigure[The state variable $Y_t$]{
    \includegraphics[width=0.33\textwidth]{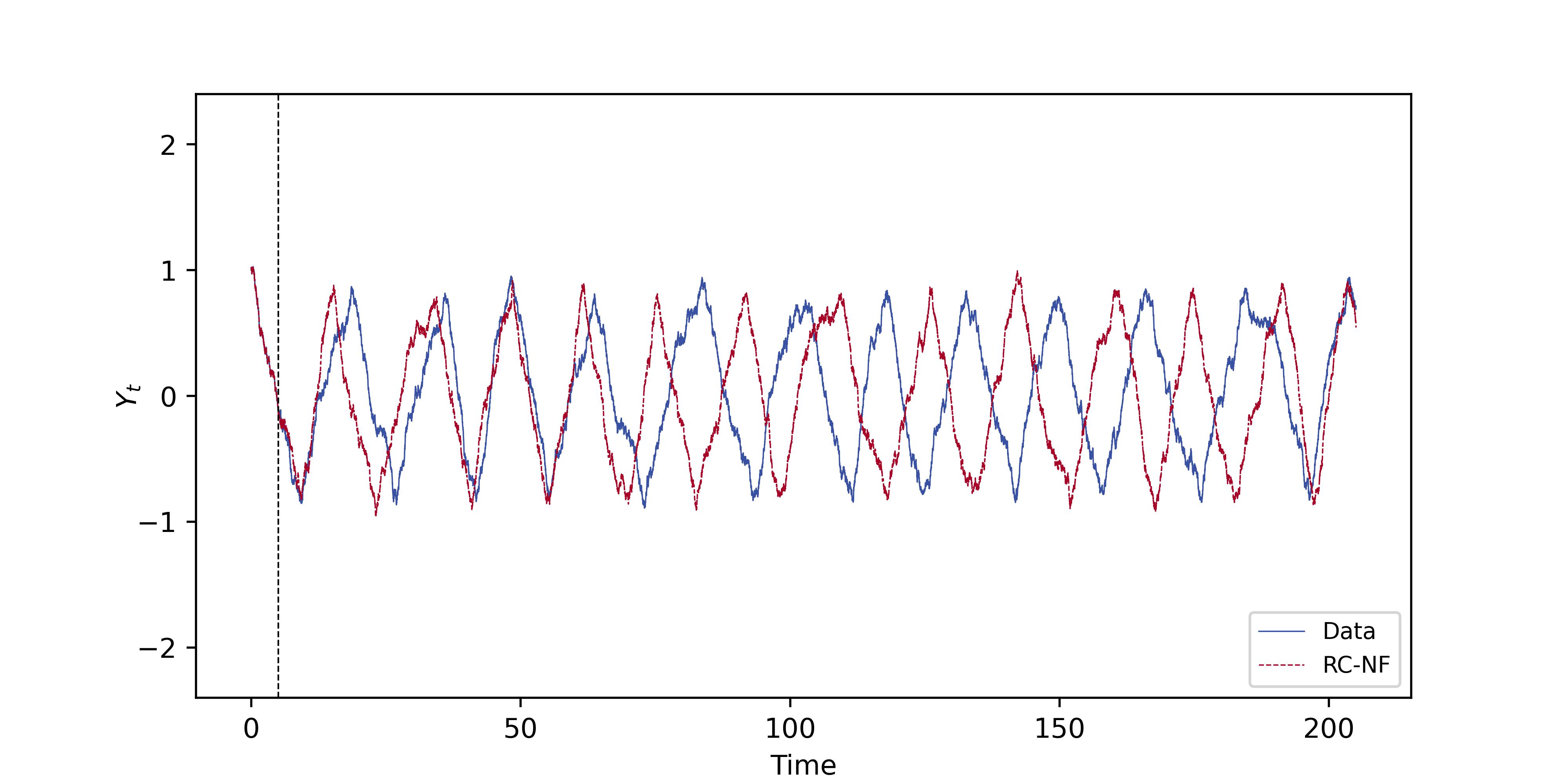}
    \label{fig: subfig: vdpv}
    }
    \caption{The solid blue lines represent data, and the red dashed lines represent the results of RC-NF. The gray dotted vertical line separates the trajectories into the warm-up stage and the generation stage. (a) Empirical variances during the testing phase. The larger empirical variances come from the state variable $X_t$. (b) Long trajectories of the state variable $X_t$. (c) Long trajectories of the state variable $Y_t$.}
    \label{fig: vdp one track}
\end{figure}

Due to the presence of noise, the trajectory of the Van der Pol system no longer strictly follows the periodic motion but still follows the pattern of relaxation oscillation. If we amplify the noise (e.g., taking diffusion coefficients $g^1=g^2=1.0$), we find that the variances of the state variables $[X_t, Y_t]^T$ no longer follow periodic oscillations, but trajectories of this system are still distributed around the limit cycle and the cubic nullcline $Y_t = X_t - X^3_t/3$. Nevertheless, the RC-NF framework still performs well (not shown in this paper).

\noindent $\bullet$ \textbf{Stochastic mixed-mode oscillation: the effects of small noise} \label{sec: exper: mmo}

Consider the following folding singularity with additive noise on $\mathbb{R}^2$:

\begin{equation} \label{eq: mmo}
\begin{cases}
    dX_t = 10(X_t - X_t^3/3 - Y_t)dt + g^1 dB^1_t,\\ 
    dY_t = (X_t+0.988)dt + g^2 dB^2_t,
\end{cases}
\end{equation}
where $g^1 = g^2 = 0.005$ are on the diagonal of the diffusion coefficient matrix $\bm{g}$ and $[B^1_t, B^2_t]^T$ is a two-dimensional standard Brownian motion. A stable limit cycle is present in the associated deterministic dynamics of the model (\ref{eq: mmo}). When the randomly perturbed trajectory passes a folding singularity, even small random perturbations cause the stochastic trajectory to ``jump" randomly between a small cycle and a large cycle. Such oscillation systems which are patterns consisting of alternating structures of small-amplitude and large-amplitude oscillations are called stochastic mixed-mode oscillations (stochastic MMOs). See Fig. \ref{fig: mmo one track}.

We set the time step for this system to be $\delta t = \Delta t = 0.01$ in accordance with the previous experiments. The initial value is $[X_0, Y_0]^T = [0.5, 0]^T$. Using the Euler-Maruyama scheme, we generate 1000 trajectories with a total length of 4000, where the training length $T$, verification length $T_{valid}$, and prediction length $T_{test}$ are 2000, 100, and 1900 respectively. The first 100 steps of training data are used for warm-up.

Fig. \ref{fig: mmo heatmap,wass, kl, pdf} displays the results of rolling predictions generated by the RC-NF model. The probability density functions of the trajectory data and the rolling predictions of RC-NF on the test dataset are shown in Fig. \ref{fig: subfig: mmoheatmap}. The predicted results of RC-NF successfully reproduce jagged-like probability density functions over time. Fig. \ref{fig: subfig: mmomanifoldpdf} displays the distributions of the data and prediction results of the RC-NF during the test phase in the two-dimensional space, which both indicate that most data points and predicted values are gathered near the large and small cycles. Fig. \ref{fig: subfig: mmo wasser and kl} displays Wasserstein distances and KL divergences between the reference distributions and the estimated distributions at different times $t$ in testing; the red dashed lines represent the selected snapshots, whose probability density functions and specific values are displayed in Fig. \ref{fig: subfig: mmo pdf select}, Tables  \ref{tab: wasser} and \ref{tab: kl}, respectively. These results demonstrate the effectiveness of our proposed RC-NF framework. 
\begin{figure}[!ht]
  \hspace{0.1in}
  \centering 
  \subfigure[PDFs over time. From left to right: Reference ($X_t$), RC-NF ($X_t$), Reference ($Y_t$) and RC-NF ($Y_t$)]{
    \label{fig: subfig: mmoheatmap} 
    \includegraphics[height=1.65in]{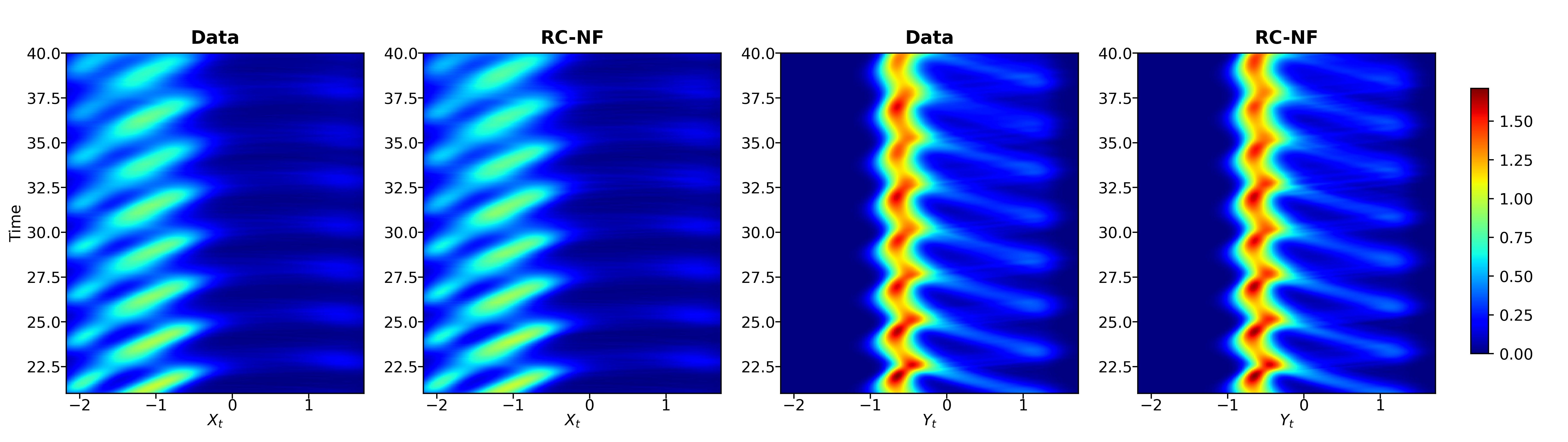}}
  \subfigure[Distributions of ($X_t$, $Y_t$) in testing. Reference (left) and RC-NF (right)]{
    \label{fig: subfig: mmomanifoldpdf} 
    \includegraphics[height=1.98in]{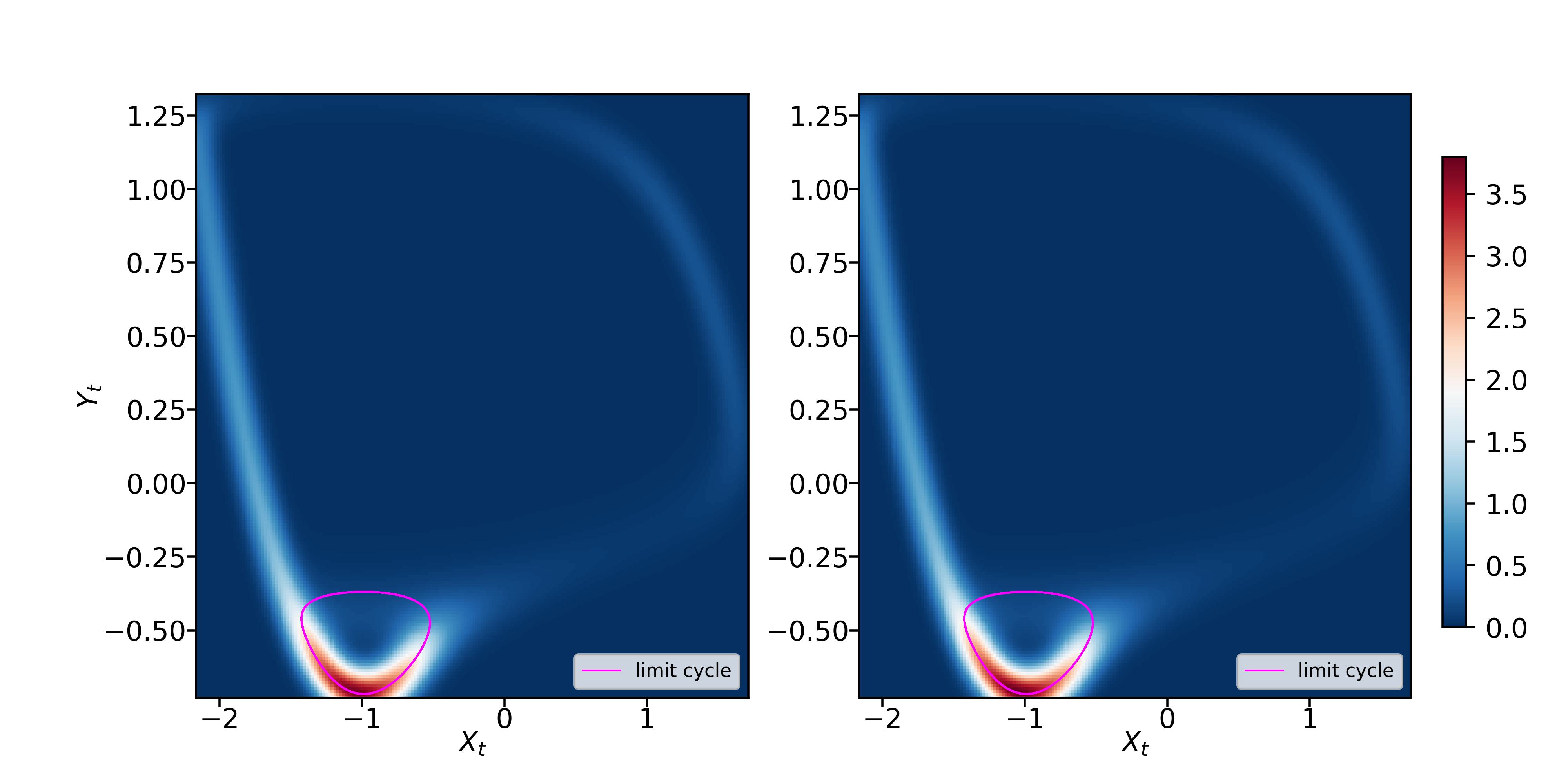}}
  \hspace{-0.2in}
  \subfigure[$W_2(\bm{X}_t,\tilde{\bm{X}}_t)$ and $D_{KL}(\tilde{\bm{X}}_t \| \bm{X}_t)$]{
    \label{fig: subfig: mmo wasser and kl} 
    \includegraphics[height=2.04in]{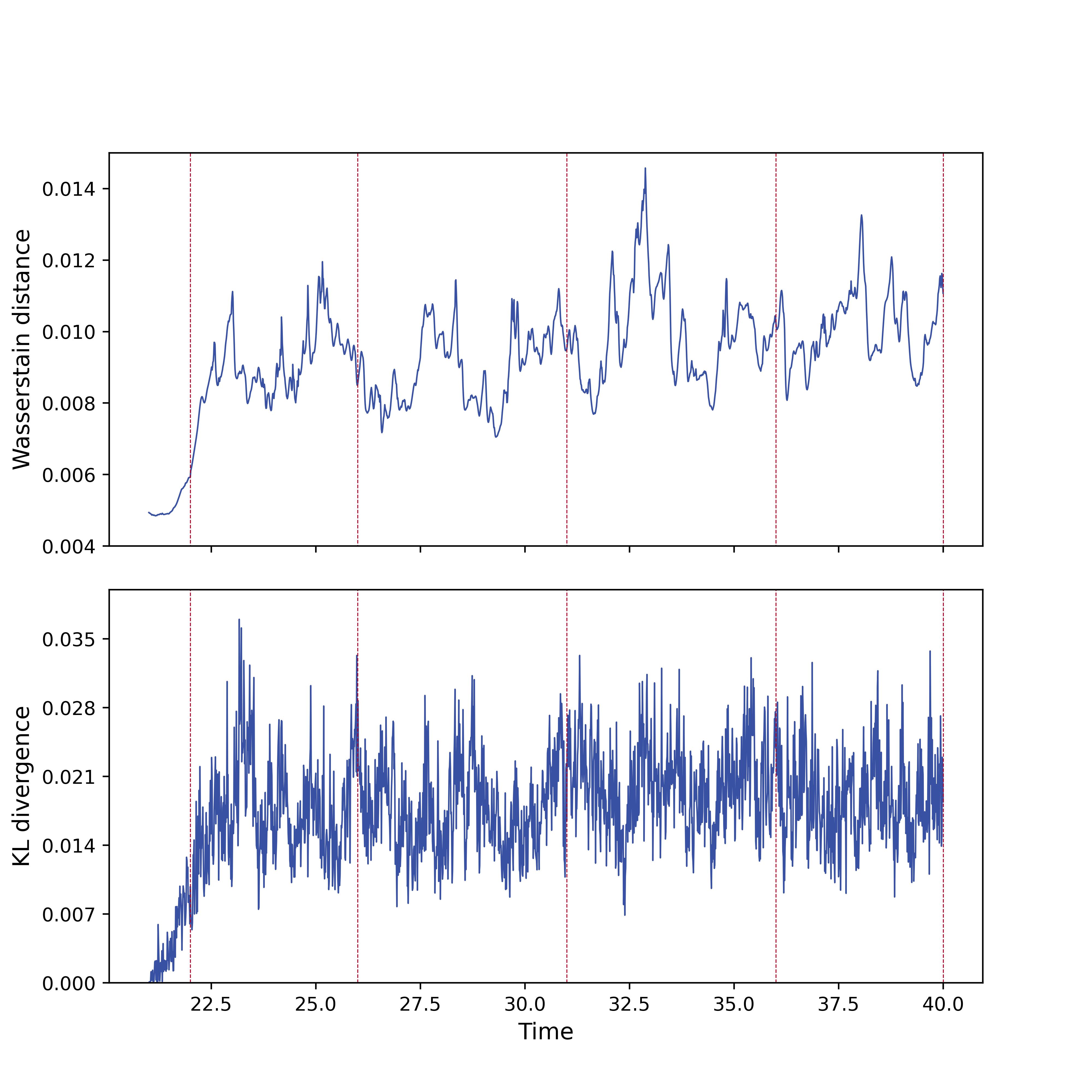}}
  \makebox[\textwidth][c]{\subfigure[PDFs. References (solid blue lines) and RC-NF results (red dashed lines) for several snapshots selected on the test dataset]{
    \label{fig: subfig: mmo pdf select} 
    \includegraphics[width=1.12\textwidth]{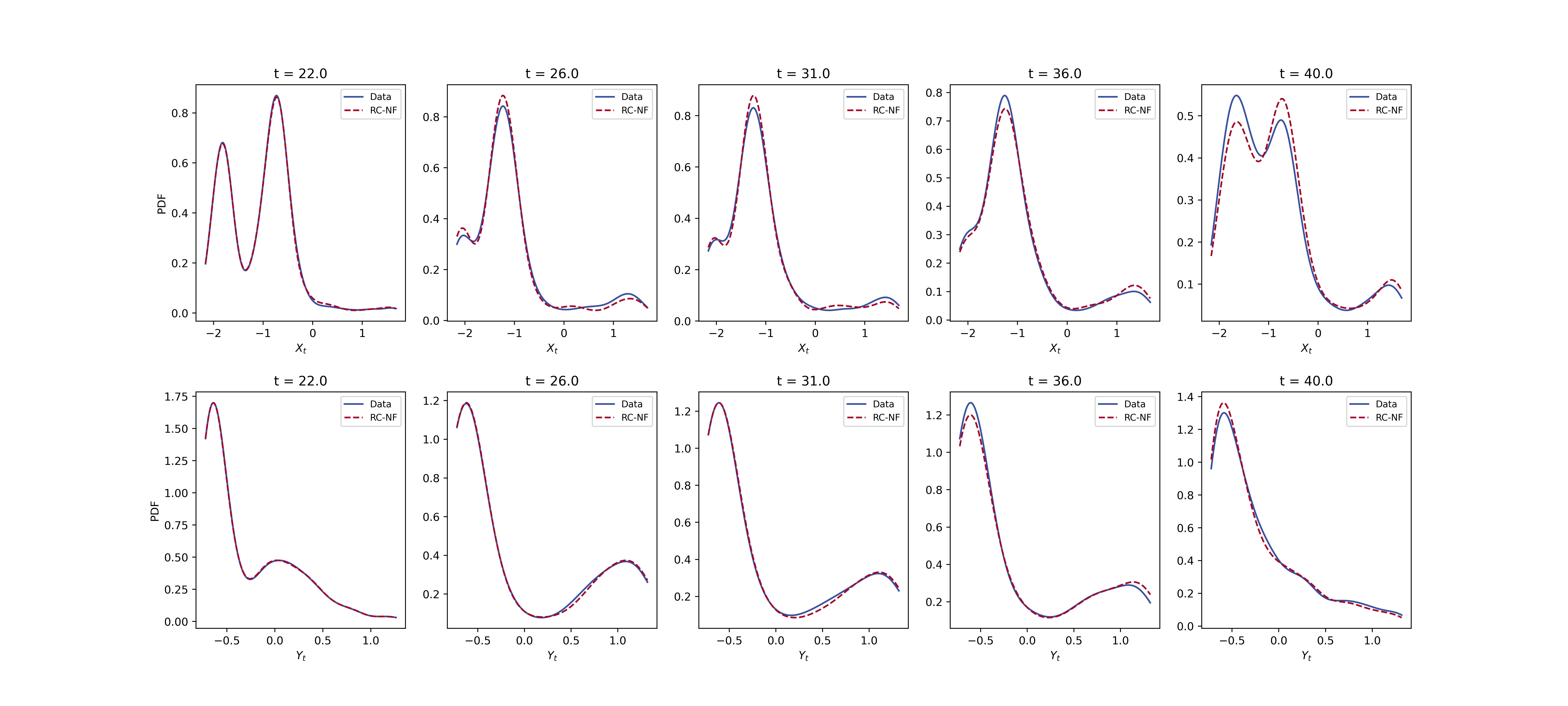}} } 
  \caption{Results of stochastic mixed-mode oscillation. PDFs of the trajectory data and the RC-NF rolling predictions are shown in (a). (b) Distributions of ($X_t$, $Y_t$) in testing. The solid magenta line represents the limit cycle. (c) Wasserstein distance (top) and KL divergence (bottom), where the notations $\bm{X}_t:= [X_t, Y_t]^T $ and $\tilde{\bm{X}}_t:= [\tilde{X}_t, \tilde{Y}_t]^T$. The red dashed lines represent snapshots where we draw the PDFs in (d) and display the values in Tables \ref{tab: wasser} and \ref{tab: kl}.}
  \label{fig: mmo heatmap,wass, kl, pdf} 
\end{figure}

Taking the initial value $[X_0, Y_0]^T = [-1,0.5]^T$. To visualize stochastic mixed-mode oscillation, we generate a long trajectory using the Euler-Maruyama scheme and the RC-NF model respectively with a length of 20500 and set observation time step $\Delta t=0.01$. The first 500 steps of the dataset are used to warm up for the RC-NF model. Fig. \ref{fig: mmo one track} shows a trajectory of the deterministic system, long trajectories of the state variable $X_t$, and long trajectories of the state variable $Y_t$, from left to right. On the one hand, there is a significant difference between Fig. \ref{fig: subfig: mmoode} and Fig. \ref{fig: subfig: mmox}, \ref{fig: subfig: mmoy}. On the other hand, both the data and the RC-NF model display patterns of stochastic mixed-mode oscillations.
\begin{figure}[!ht]
    \centering
    \subfigure[The deterministic system]{
    \includegraphics[width=0.33\textwidth]{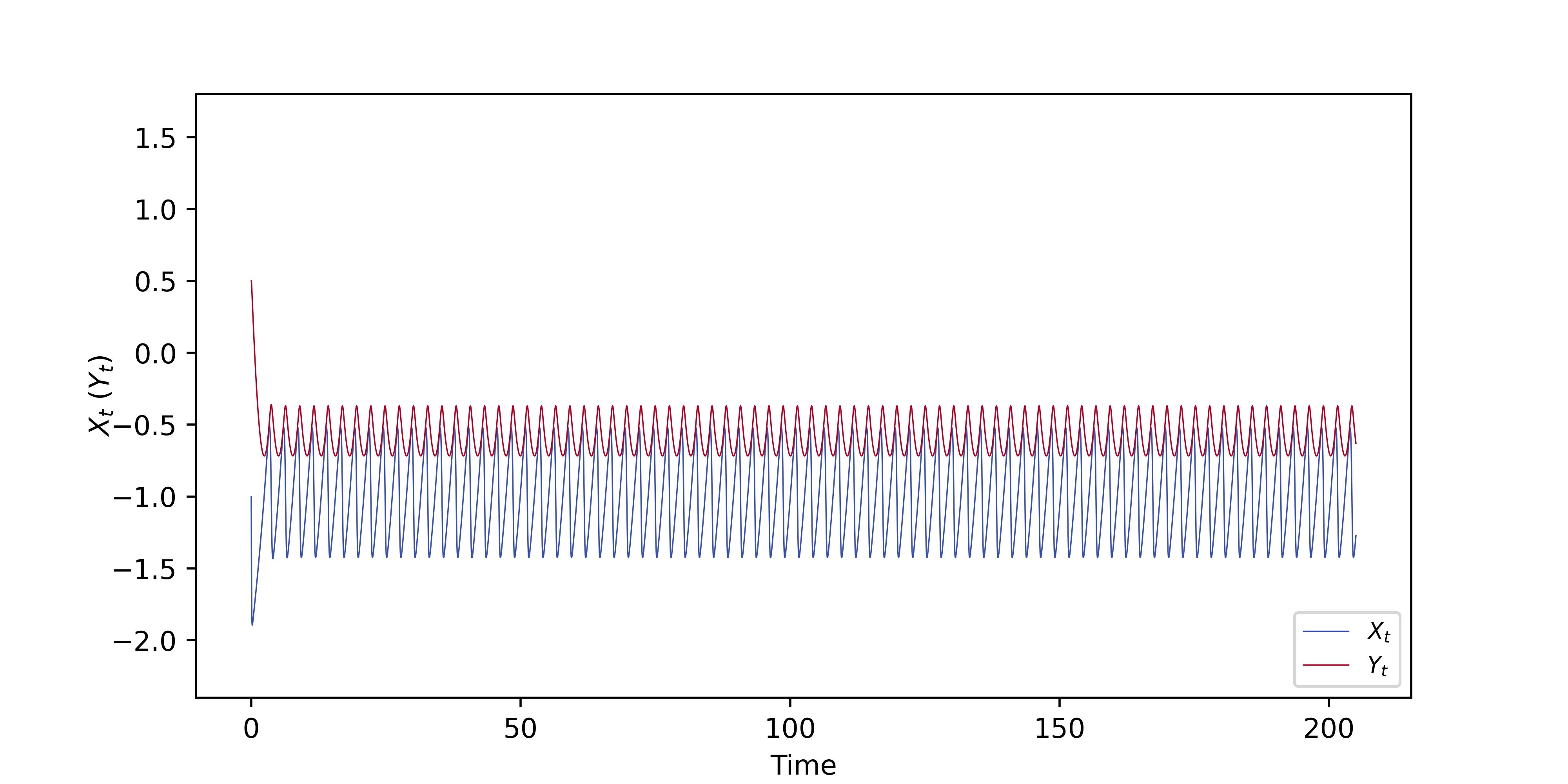}
    \label{fig: subfig: mmoode}
    }
    \hspace{-0.2in}
    \subfigure[The state variable $X_t$]{
    \includegraphics[width=0.33\textwidth]{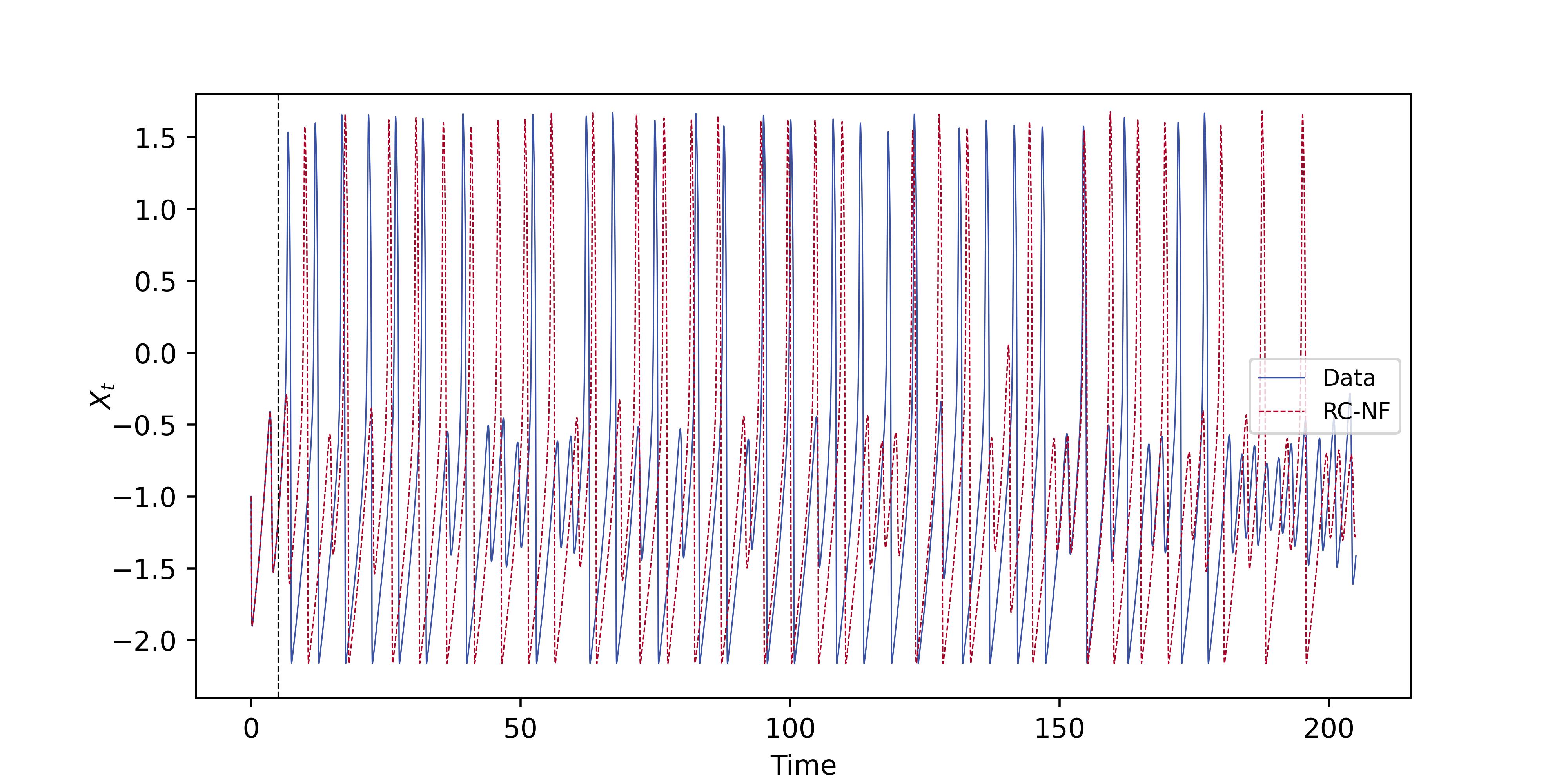}
    \label{fig: subfig: mmox}
    }
    \hspace{-0.2in}
    \subfigure[The state variable $Y_t$]{
    \includegraphics[width=0.33\textwidth]{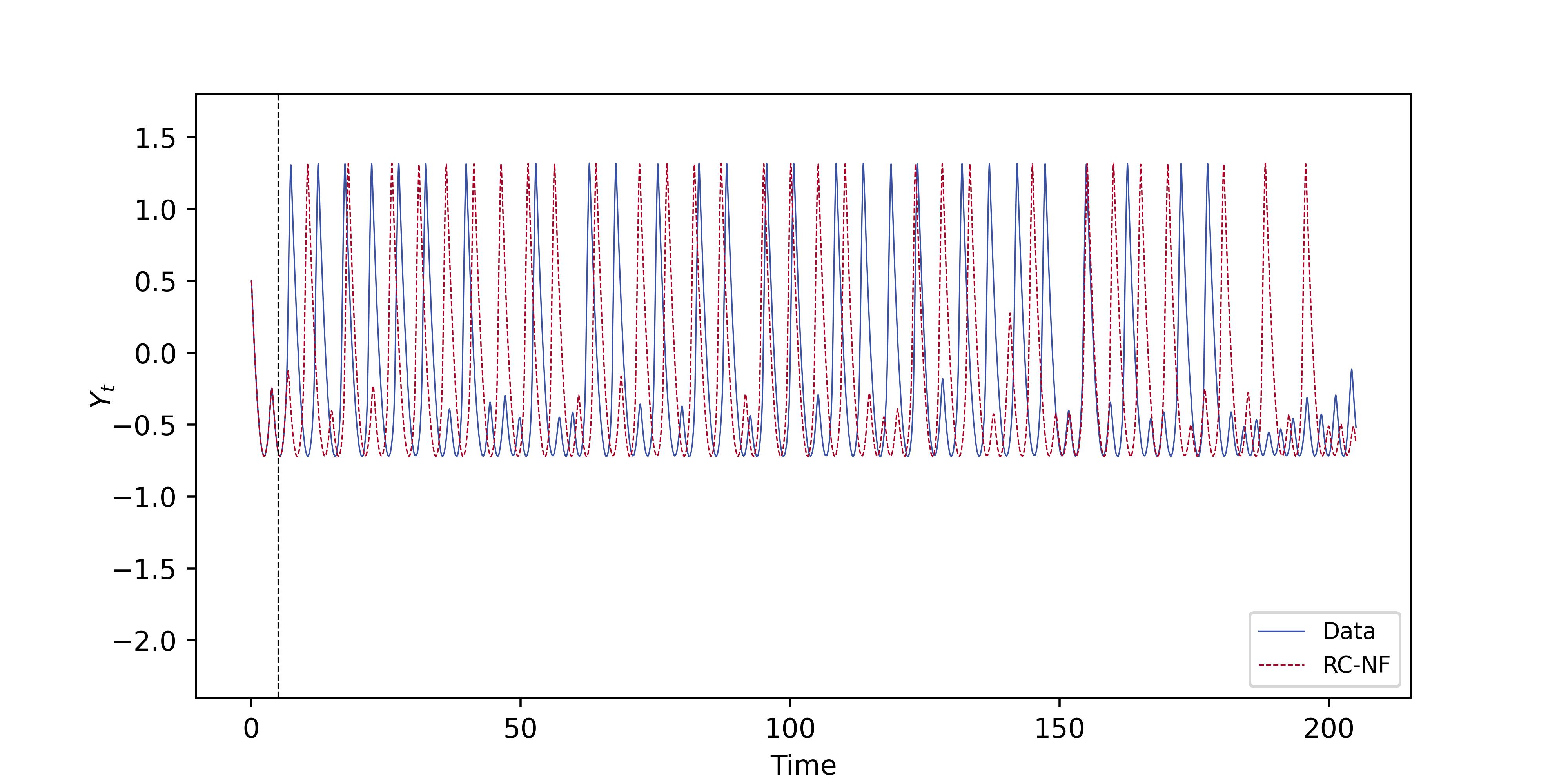}
    \label{fig: subfig: mmoy}
    }
    \caption{(a) A trajectory from the deterministic system. In (b) and (c), the solid blue lines represent data, and the red dashed lines represent the results of RC-NF. The warm-up stage and the generation stage of the trajectories are separated by the gray dotted vertical line. (b) Long trajectories of the state variable $X_t$. (c) Long trajectories of the state variable $Y_t$.}
    \label{fig: mmo one track}
\end{figure}

The RC-NF approach successfully reproduces the stochastic mixed-mode oscillation. For the examples of stochastic Van der Pol oscillator and stochastic mixed-mode oscillation, the probability density functions over time during the test phase are significantly more complex than in the other experiments. While the RC-NF model shows its effectiveness.

\subsection{The effects on stochastic differential equations with time delays}\label{sec: exper: sdde}
The solutions of the aforementioned examples (Sec. \ref{sec: exper: ou}, Sec. \ref{sec: exper: richer dynamical behaviors}) are Markov processes. Stochastic processes arising from solutions of SDDEs (\ref{eq: sdde}) typically lack the Markov property. As we mentioned in Sec. \ref{sec: related work}, various data-driven approaches have been proposed for the analysis of stochastic dynamical systems. However, due to the complexity of the system and specific model assumptions, most existing methods cannot be extended to analyze SDDEs from data. At present, there is relatively little research on this type of system. We will explore two experiments in this section to verify the learning ability of RC-NF for stochastic dynamical systems defined by SDDEs. We can solve explicitly the first linear SDDE to gain an intuitive understanding. The El Ni\~no-Southern Oscillation simplified model is more complex and has the background of real-world applications.

\noindent$\bullet$ \textbf{Linear SDDE: long-term prediction of a non-Markov process} \label{sec: exper: linear sdde}

Consider the following SDDE with a linear coefficient \cite{rihan2021delay}:

\begin{equation} \label{eq: linear sdde}
\begin{cases}
    dX_t = \mu_0 X_{t-\tau_0}dt + g dB_t,\quad \text{for}\  t \geq 0,\\
    X_t = t+1,\quad \text{for}\  t \in [-\tau_0, 0], 
\end{cases}
\end{equation}
where $\mu_0 \in \mathbb{R}$, time delay $\tau_0 > 0$, the diffusion coefficient $g$ is a positive constant, and $B_t$ is a scalar Brownian motion. We take $\mu_0 = -1.2$, $\tau_0 = 1$, and $g = 1$. Under our setting, the existence and uniqueness of the solution of the above SDDE (\ref{eq: linear sdde}) are guaranteed \cite{rihan2021delay}. Based on the above parameter assumptions, we use the It\^o formula to obtain the solution on the time interval $[0,2]$:
\begin{equation}\label{eq: linearsdde solution}
X_t =
\begin{cases}
    1 - 0.6 t^2 + B_t, \quad \text{for}\ t \in [0,1],\\
    0.2(t-1)^2 - t + 1.4 + \int_1^t B_{s-1}ds + B_t, \quad \text{for}\ t \in [1,2].
\end{cases}
\end{equation}
Note that although this linear SDDE (\ref{eq: linear sdde}) is formally similar to the OU process (\ref{eq: ouprocess}), their solutions are significantly different from each other.

For this system, we follow the time step setting from the previous examples, that is, $\delta t = \Delta t = 0.01$. The initial value is $X_0 = 1$, which is determined by equation (\ref{eq: linear sdde}). We generate 2000 trajectories with a total length of 2000 using the Euler-Maruyama scheme, where the training length $T$, verification length $T_{valid}$, and prediction length $T_{test}$ are 1000, 100, and 900 respectively. The first 100 steps of training data are used for warm-up. 

The results of rolling predictions of the RC-NF model are shown in Fig. \ref{fig: linearsdde heatmap,wass, kl, pdf}. Fig. \ref{fig: subfig: sddelinearheatmap} depicts the probability density functions of the trajectory data and the rolling predictions of RC-NF on the test dataset. The results of RC-NF rolling predictions mimic the fluctuation of the probability density function over time. Fig. \ref{fig: subfig: sddelinear wasser and kl} displays Wasserstein distances and KL divergences between the reference distributions and the estimated distributions at different times $t$ in the test phase; the red dashed lines represent the selected snapshots, whose probability density functions and specific values are displayed in Fig. \ref{fig: subfig: sddelinear pdf select}, Tables  \ref{tab: wasser} and \ref{tab: kl}, respectively. These results demonstrate the effectiveness of the RC-NF model we proposed and its ability to predict the non-Markov linear SDDE over a long time. 
\begin{figure}[!ht]
  \centering
  \hspace{-0.2in} 
  \subfigure[PDFs over time. Reference (left) and RC-NF (right)]{
    \label{fig: subfig: sddelinearheatmap} 
    \includegraphics[height=1.98in]{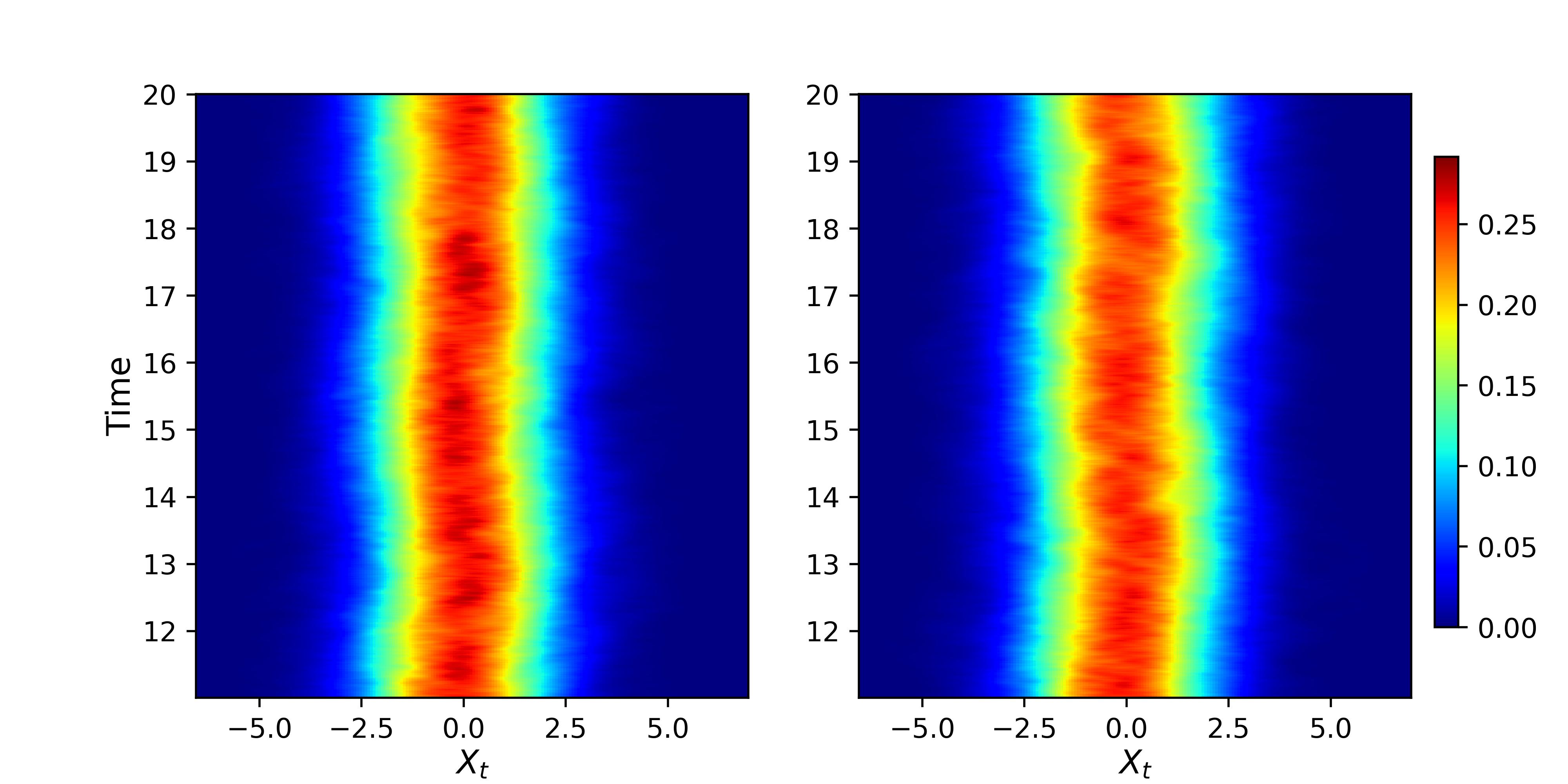}}
  \hspace{-0.1in}
  \subfigure[$W_2(\bm{X}_t,\tilde{\bm{X}}_t)$ and $D_{KL}(\tilde{\bm{X}}_t \| \bm{X}_t)$]{
    \label{fig: subfig: sddelinear wasser and kl} 
    \includegraphics[height=2.04in]{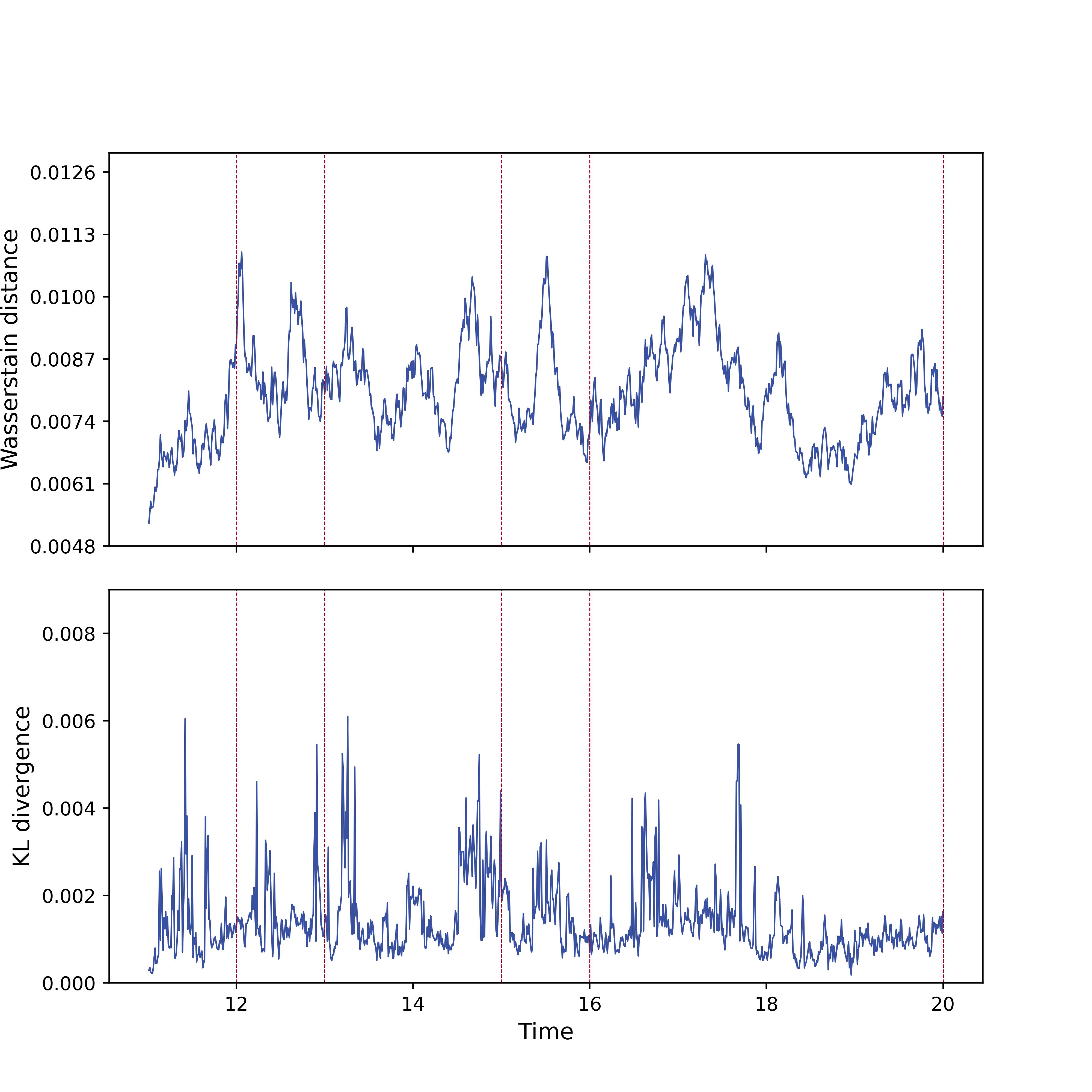}}
  \makebox[\textwidth][c]{\subfigure[PDFs. References (solid blue lines) and RC-NF results (red dashed lines) for several snapshots selected on the test dataset]{
    \label{fig: subfig: sddelinear pdf select} 
    \includegraphics[width=1.12\textwidth]{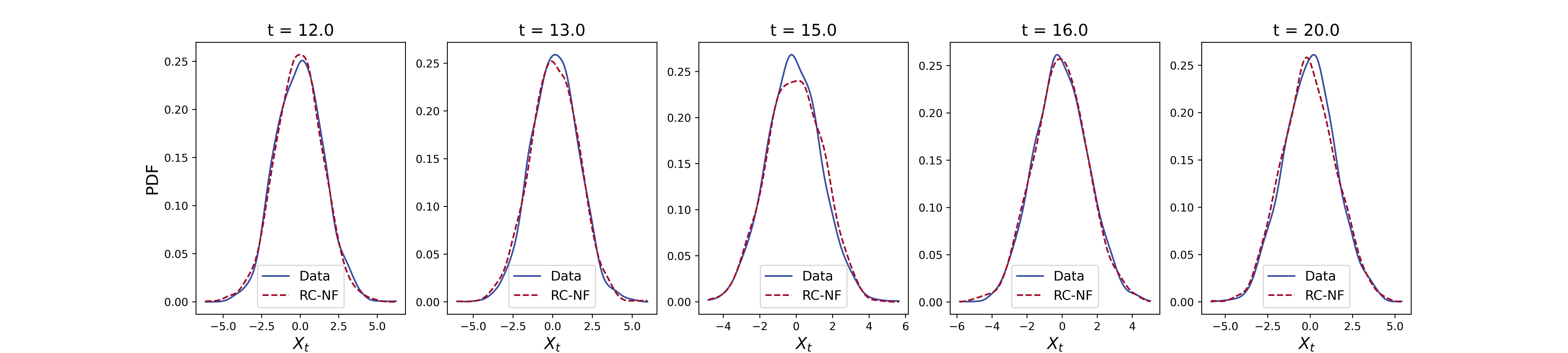}} } 
  \caption{Results of the linear SDDE. PDFs of the trajectory data and the RC-NF rolling predictions are shown in (a). (b) Wasserstein distance (top) and KL divergence (bottom), the red dashed lines represent snapshots where we draw the PDFs in (c) and display the values in Tables \ref{tab: wasser} and \ref{tab: kl}.}
  \label{fig: linearsdde heatmap,wass, kl, pdf} 
\end{figure}

Although it is a simple linear SDDE (\ref{eq: linear sdde}), it essentially differs from the previous examples defined by SDEs. However, our proposed RC-NF achieves almost the same level of accuracy in predicting the long-term evolution of the linear SDDE.

\noindent$\bullet$ \textbf{El Ni\~no-Southern Oscillation (ENSO) simplified model}\label{sec: exper: esno}

Consider the El Ni\~no-Southern Oscillation (ENSO) simplified model \cite{suarez1988delayed,Sun2022} with additive noise that is presented below,
\begin{equation}\label{eq: enso}
\begin{cases}
    dX_t = (X_t-X_t^3-\alpha_0 X_{t-\tau_0})dt+gdB_t,\quad \text{for}\ t \geq 0,\\
    X_t = C_0,\quad \text{for}\ t \in [-\tau_0, 0], 
\end{cases}
\end{equation}
where $|\alpha_0|<1$ measures the influence of the returning signal relative to that of the local feedback, $\tau_0$ is the time delay corresponding to wave transit time, the constant $C_0\in \mathbb{R}$, the diffusion coefficient $g$ is also a positive constant, and $B_t$ is a scalar Brownian motion. Note that the deterministic model corresponding to (\ref{eq: enso}) has two stable fixed points at $X_t=(1-\alpha_0)^{1/2}$ and $X_t=-(1-\alpha_0)^{1/2}$, and one unstable fixed point at $X_t=0$. In accordance with \cite{suarez1988delayed,Sun2022}, we set $\alpha_0 = 0.75$ and $\tau_0=6$.

For this system, we follow the time step setting from the previous examples, that is, $\delta t = \Delta t = 0.01$. The initial values $X_0 = C_0$, $C_0$ are drawn from a uniform distribution on $[-1,1]$. We generate 2000 trajectories with a total length of 4000 using the Euler-Maruyama scheme, where the training length $T$, verification length $T_{valid}$, and prediction length $T_{test}$ are 2000, 100, and 1900 respectively. The first 100 steps of training data are used for warm-up.

Noting that the probability distribution of state $X_t$ of this system depends on the initial value $C_0$, we fix the initial value $X_0 = -0.1$ and generate new trajectories to explore the performance of the proposed method. The trained RC-NF model and the Euler-Maruyama scheme both generate a total of 2000 trajectories, respectively. The first 500 steps of the dataset are used for warm-up.

The comparison results of the trajectory data and the trajectories generated by the RC-NF model are shown in Fig. \ref{fig: esno heatmap,wass, kl, pdf}. More specifically, Fig. \ref{fig: subfig: esnoheatmap} depicts the probability density functions of the trajectory data and trajectories generated by the RC-NF model. As can be seen from the figure, the RC-NF model successfully captures the change in probability density over time, with its highest peak alternating with time. Fig. \ref{fig: subfig: esno wasser and kl} displays Wasserstein distances and KL divergences between the reference distributions and the estimated distributions at different times $t$ in testing; the red dashed lines represent the selected snapshots, whose probability density functions and specific values are displayed in Fig. \ref{fig: subfig: esno pdf select}, Tables  \ref{tab: wasser} and \ref{tab: kl}, respectively. Despite there being a slight difference in peaks, we successfully learn the bimodal probability density function of the system at a fixed time. Although the transition rate $k_{AB}$ is no longer approximately constant for $\tau_{mol} < t \ll \tau_{rxn}$, as shown in Sec. \ref{sec: exper: dw},  we can still calculate the time correlation function $C_{AB}(t)/C_A$ to assess the quality of the RC-NF model, the relevant results are shown in Fig. \ref{fig: subfig: esnokba}. The reference time correlation function and the estimated function by the RC-NF model exhibit a similar trend. Our proposed method is still effective for the ESNO simplified model.
\begin{figure}[!ht]
  \centering
  \hspace{-0.2in} 
  \subfigure[PDFs over time. Reference (left) and RC-NF (right)]{
    \label{fig: subfig: esnoheatmap} 
    \includegraphics[height=1.98in]{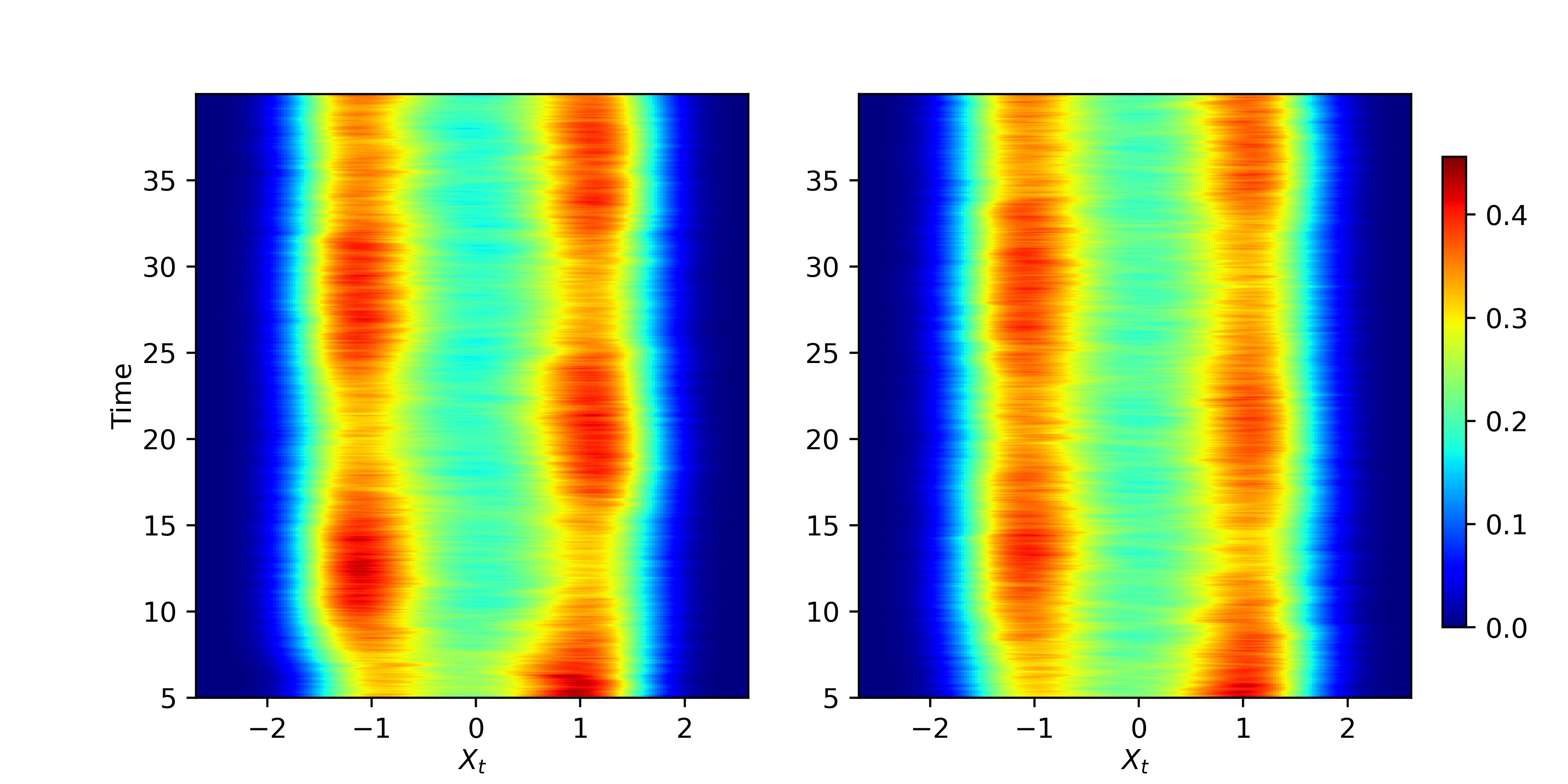}}
  \hspace{-0.1in}
  \subfigure[$W_2(\bm{X}_t,\tilde{\bm{X}}_t)$ and $D_{KL}(\tilde{\bm{X}}_t \| \bm{X}_t)$]{
    \label{fig: subfig: esno wasser and kl} 
    \includegraphics[height=2.04in]{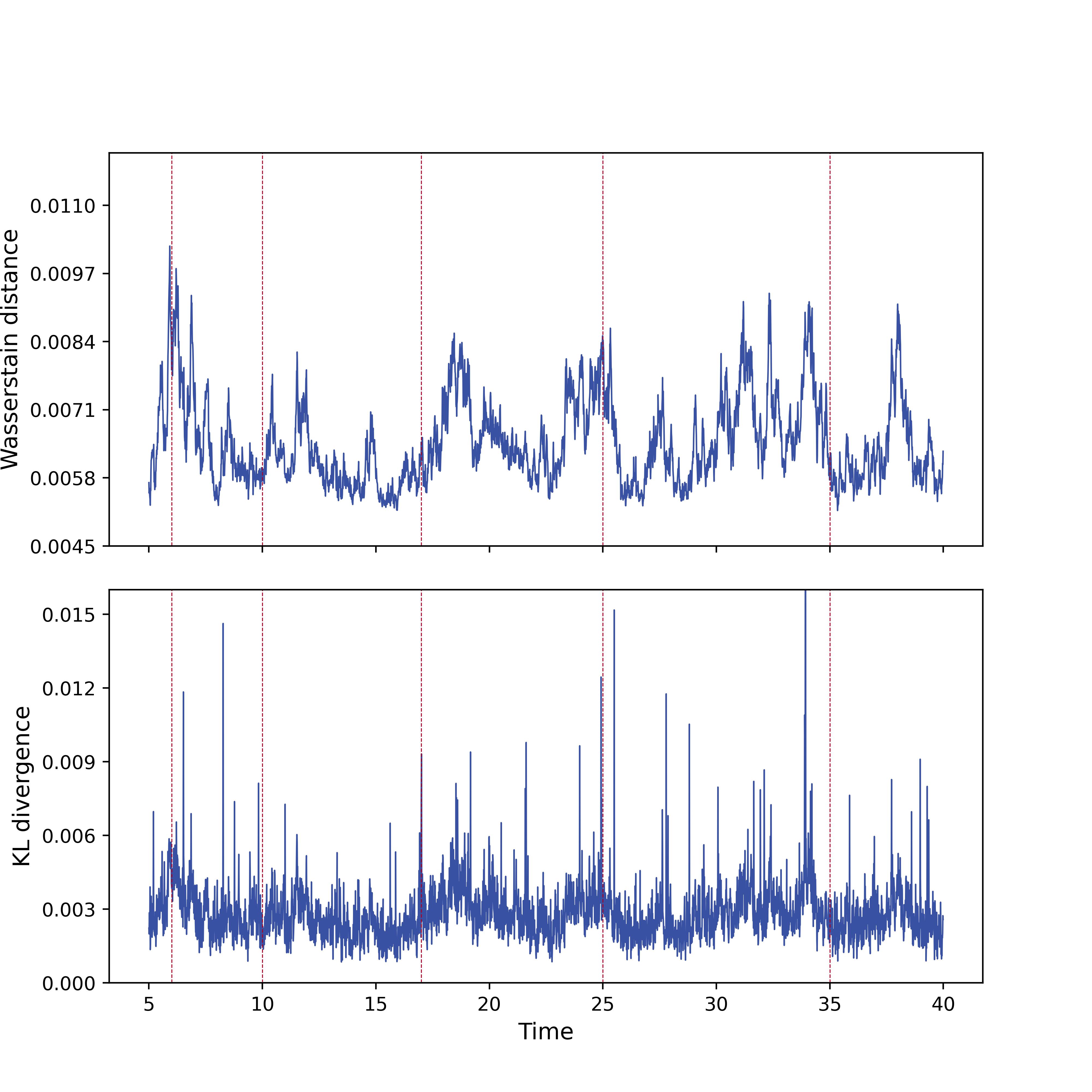}}
  \makebox[\textwidth][c]{\subfigure[PDFs. Reference (solid blue lines) and RC-NF results (red dashed lines) for several snapshots selected on the test dataset]{
    \label{fig: subfig: esno pdf select} 
    \includegraphics[width=1.12\textwidth]{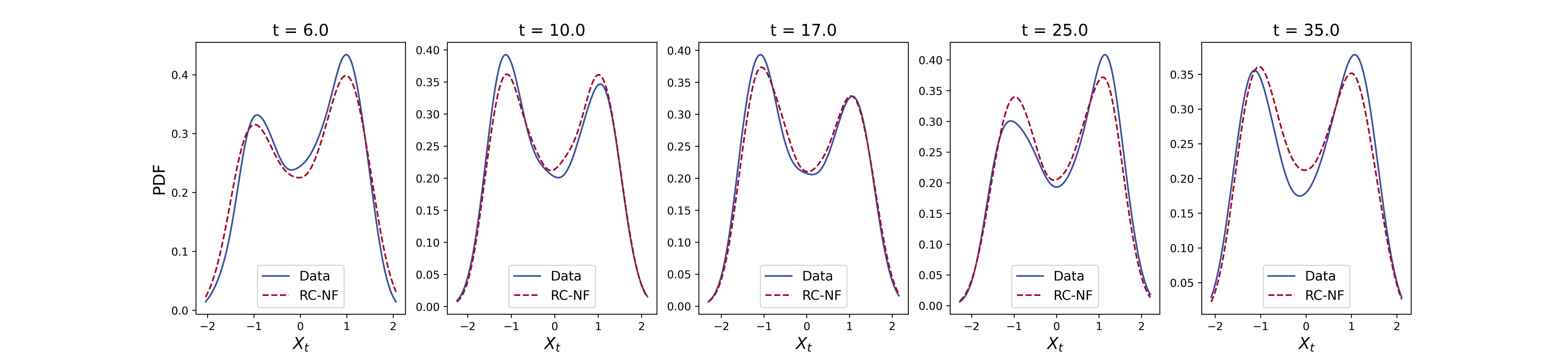}} }
  \subfigure[Time correlation function $C_{AB}(t)$]{
    \includegraphics[width=0.6\textwidth]{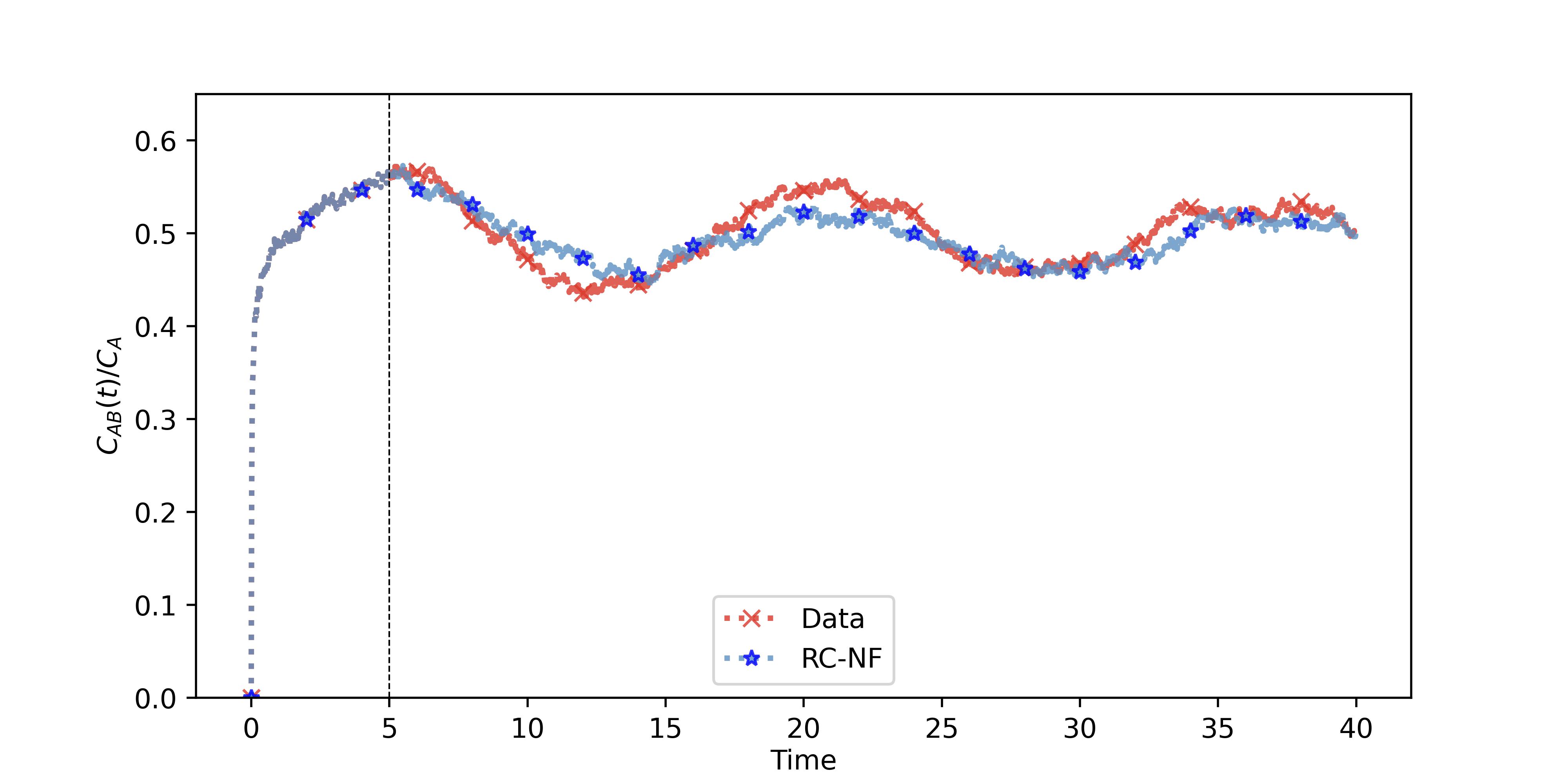}
    \label{fig: subfig: esnokba}}
  \caption{Comparison results of ENSO simplified model. PDFs of the trajectory data and the RC-NF rolling predictions are shown in (a). (b) Wasserstein distance (top) and KL divergence (bottom) at different times $t$, the red dashed lines represent snapshots where we draw the PDFs in (c) and display the values in Tables \ref{tab: wasser} and \ref{tab: kl}. (d) Time correlation function $C_{AB}(t)$, the red cross dot line is calculated from the trajectory data, and the blue asterisk line is calculated from trajectories generated by the RC-NF model. The black dotted vertical line represents the split between the warm-up stage and the generation stage.}
  \label{fig: esno heatmap,wass, kl, pdf} 
\end{figure}

The ESNO simplified model, whose corresponding SDDE (\ref{eq: enso}) is nonlinear and whose solution is a non-Markov, non-stationary stochastic process, is the most complex one-dimensional stochastic dynamical system we consider in this article. Its evolution depends on initial values and time delay. Nevertheless, RC-NF successfully predicts the evolution of state variables in the probability distribution sense within the tolerable error and remains competitive in some qualitative and quantitative analyses.

\subsection{Stochastic Lorenz system: further comparison of RC and RC-NF}\label{sec: exper: stochastic lorenz}
Lorenz system \cite{Lorenz63} is notable for having chaotic solutions depending on the parameter values and initial conditions. This demonstrates how physically deterministic systems can still have some inherent unpredictability. Moreover, the necessity of stochasticity in the climate system and some related studies can be found in \cite{palmer2019stochastic, AGARWAL2016142, Tapio2021}. One of the stochastic forms of Lorenz system describes by the following nonlinear SDEs:
\begin{equation} \label{eq: lorenz}
\begin{cases}
    dX_t = \sigma_0(Y_t - X_t) dt + g^1dB^1_t, \\
    dY_t = \left(X_t(\rho_0-Z_t)-Y_t\right) dt + g^2dB^2_t,\\
    dZ_t = (X_tY_t - \beta_0Z_t) dt + g^3dB^3_t,
\end{cases}
\end{equation}
where $g^1$, $g^2$, $g^3$ are positive constants on the diagonal of the diffusion coefficient matrix $\bm{g}$ and $[B^1_t,B^2_t,B^3_t]^T$ is a three-dimensional standard Brownian motion. The Prandtl Number $\sigma_0$, the Rayleigh Number $\rho_0$, and a domain geometric factor $\beta_0$ are control parameters. We select $\sigma_0 = 10$, $\rho_0 = 28$, $\beta = 8/3$, and $g^1=g^2=g^3=3$. 

For this system, we set the time step of the Euler-Maruyama scheme $\delta t =10^{-5} $ and the observation step $\Delta t = 0.01$. The initial value is $[X_0, Y_0, Z_0]^T = [0,1,0]^T$. We generate 1000 trajectories with a total length of 4000 using the Euler-Maruyama scheme, where the training length $T$, validation length $T_{valid}$, and prediction length $T_{test}$ are 2000, 100, and 1900 respectively, and the first 100 steps of training data are used for warm-up. The generated sample trajectories are standardized, i.e., $X_t^{(m)}(\text{new}) = \left(X_t^{(m)}- \text{Mean}_{t,m}(X_t^{(m)})\right)/\text{Std}_{t,m}(X_t^{(m)})$. The variables $Y_t^{(m)}(\text{new})$ and $Z_t^{(m)}(\text{new})$ are dealt with in similar ways.

RC can also replicate the strange attractor of the Lorenz system, so the main task of this subsection is to compare the performance of RC and RC-NF in the stochastic Lorenz system. We first plot the trajectory data in the testing phase, the trajectories predicted by the RC model, and the trajectories predicted by the RC-NF model in three-dimensional space, as shown in Fig. \ref{fig: subfig: lorenzattractor}. Unsurprisingly, the trajectories of all three exhibit butterfly-like strange attractors. Fig. \ref{fig: subfig: lorenzpdf} depicts the probability density functions of the trajectory data and the rolling predictions of RC and RC-NF on the test dataset. The predicted probability density functions of RC-NF and the probability density functions of the trajectory data are similar, but the predicted probability density functions of RC are quite different, although it exhibits a butterfly-like strange attractor. Due to the ergodicity of the Lorenz system, we calculate the Wasserstein distances and KL divergences between the datasets (trajectory data, rolling predictions of the RC model, and rolling predictions of the RC-NF model) composed of all states in testing, which are displayed in Table \ref{tab: lorenz wasser and kl}. Results from RC-NF outperform those from RC.
\begin{figure}[!ht]
    \centering
    \subfigure[Trajectories in testing, references (left), generated by RC (middle), generated by RC-NF (right)]{
    \includegraphics[width=0.98\textwidth]{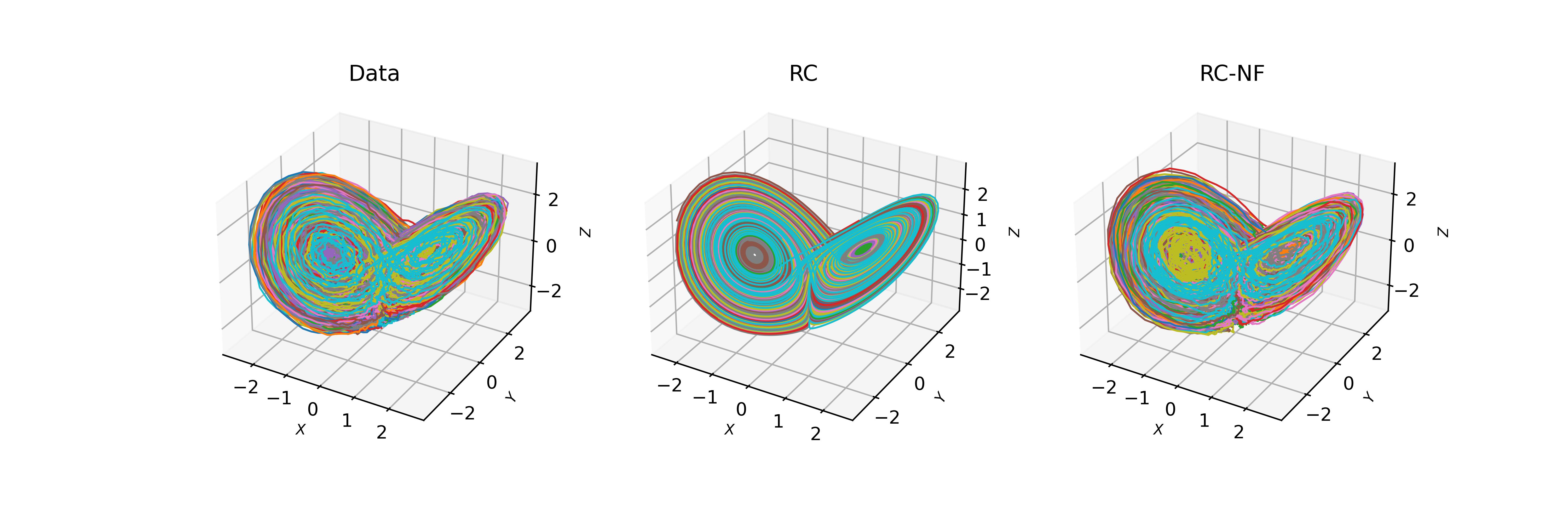}
    \label{fig: subfig: lorenzattractor}
    }
    \hspace{-0.2in}
    \subfigure[Marginal PDFs over time. References (top), RC (middle), and RC-NF (bottom)]{
    \includegraphics[width=0.98\textwidth]{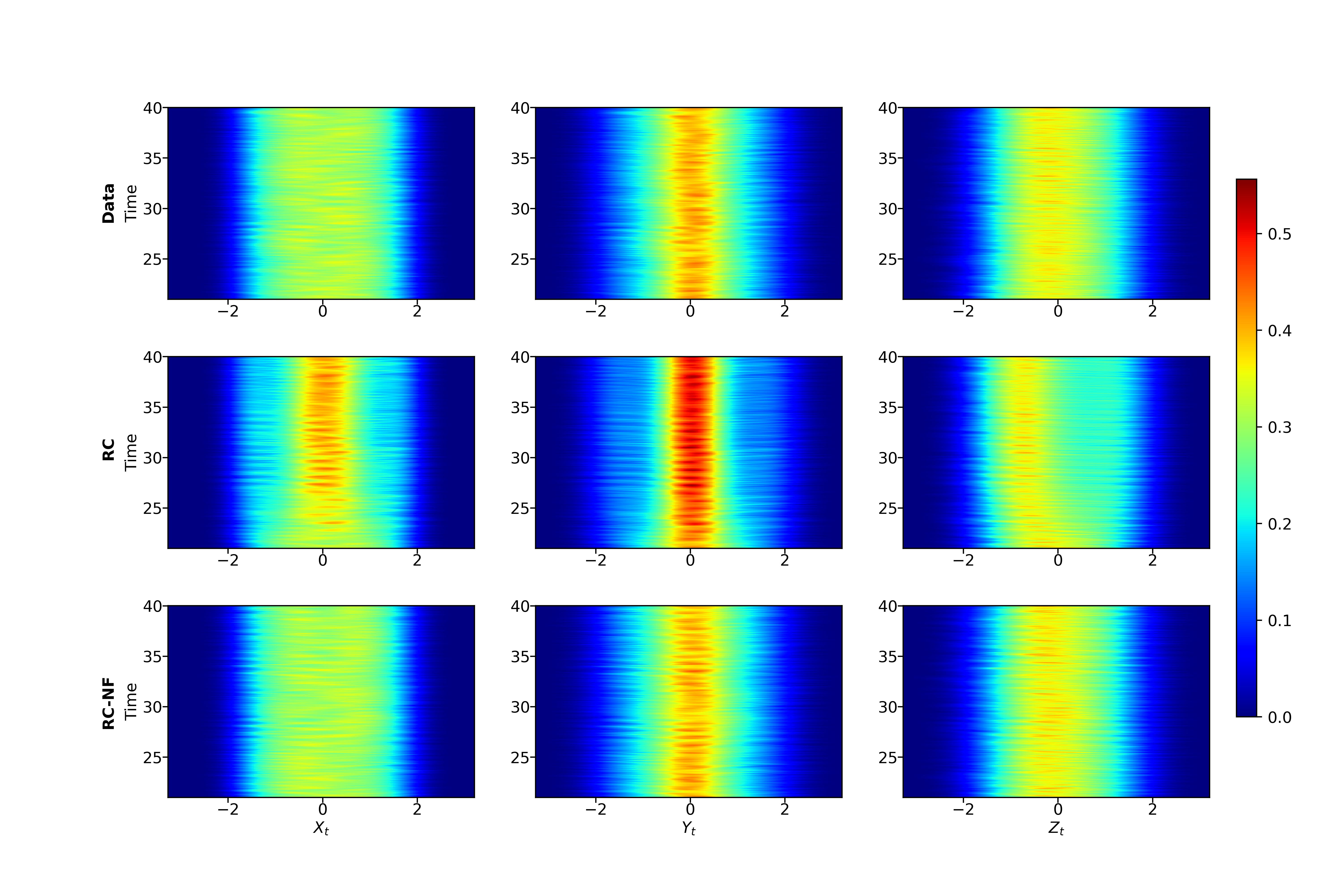}
    \label{fig: subfig: lorenzpdf}
    }
    \caption{Data, rolling predictions of RC and RC-NF on the test dataset for stochastic Lorenz system. (a) Strange attractors; (b) PDFs in the testing phase.}
    \label{fig: lorenz prediction}
\end{figure}

\begin{table}[!ht] 
\caption{Stochastic Lorenz model: Wasserstein distances and KL divergences are calculated for RC and RC-NF predictions, respectively.}
\label{tab: lorenz wasser and kl}
\centering
\begin{tabular}{l|cc} 
\bottomrule
\textbf{Criteria} & \textbf{RC} & \textbf{RC-NF}\\
\hline
Wasserstein distance & $4.7918\times 10^{-2}$ & $2.6603\times 10^{-2}$\\
KL divergence & $4.6666\times 10^{-4}$ & $7.5146\times 10^{-5}$\\
\toprule
\end{tabular}
\end{table}

We use two criteria to further characterize the difference between RC and RC-NF to highlight the performance of RC-NF. The metric approach is maximum Lyapunov exponents \cite{pathak2018PhysRevLett, rosenstein1993practical, pathak2017lyapunov} and the topological approach is close returns \cite{AGARWAL2016142, MINDLIN1992229}. The Lyapunov exponent of a dynamical system is a quantity that characterizes how quickly two infinitesimally similar trajectories separate from one another. The largest one is called the maximal Lyapunov exponent (MLE). A positive MLE is typically seen as a criterion of chaos in the system. Using the technique described in \cite{rosenstein1993practical}, MLEs are computed for each trajectory in testing. MLEs of each trajectory in the dataset, generated by the RC model, and generated by the RC-NF model, are shown in the boxplots in Fig. \ref{fig: subfig: lorenzMLE}. Similar to the results of the probability density functions, the MLEs of the trajectory data and the prediction results of the RC-NF model are similar. Note that the median of MLEs in the RC model is 1.8525, it is also no longer consistent with the MLE of the noiseless Lorenz system under our parameter settings (MLE$=0.91$). 

Close returns are used to identify segments in the chaotic time series that can stand in for the unstable periodic orbits that surround the strange attractor. For a trajectory $\bm{X}_{T+T_{valid}+t}$, $t=1, \dots, T_{test}$ on the test dataset, close returns segments can be identified in the data by making a two-dimensional close returns plot of
\begin{equation}\label{eq: close returns}
\|\bm{X}_t-\bm{X}_{t+p}\|_2
\begin{cases}
    < \varepsilon_0,\quad \text{black},\\
    > \varepsilon_0,\quad \text{white},
\end{cases}
\end{equation}
where lag $p>0$. The constant $\varepsilon_0 > 0$ is a fixed threshold depending on the diameter of the attractor, that is, $\varepsilon_0 \sim 10^{-2} \times (\text{Max}(\bm{X}_t) - \text{Min}(\bm{X}_t))$. We can draw a two-dimensional binary image of $(t,p)$. Considering our testing settings, we set $t = 1,\dots,1000$ and $p = 1,\dots,900$. A randomly selected trajectory on the test dataset is used to draw close returns maps of the trajectory data, RC prediction results, and RC-NF prediction results in Fig. \ref{fig: subfig: lorenclosereturn}. There are few differences in close returns maps among the three. Fig. \ref{fig: subfig: lorenzhistgram} is histograms of $p$ formed by the black pixels of all predicted trajectories. Again, the peak locations and heights of the RC-NF prediction results match those of the trajectory data, while RC can only match a few peak locations.

\begin{figure}[!ht]
    \centering
    \subfigure[Boxplots of MLEs]{
    \includegraphics[width=0.48\textwidth]{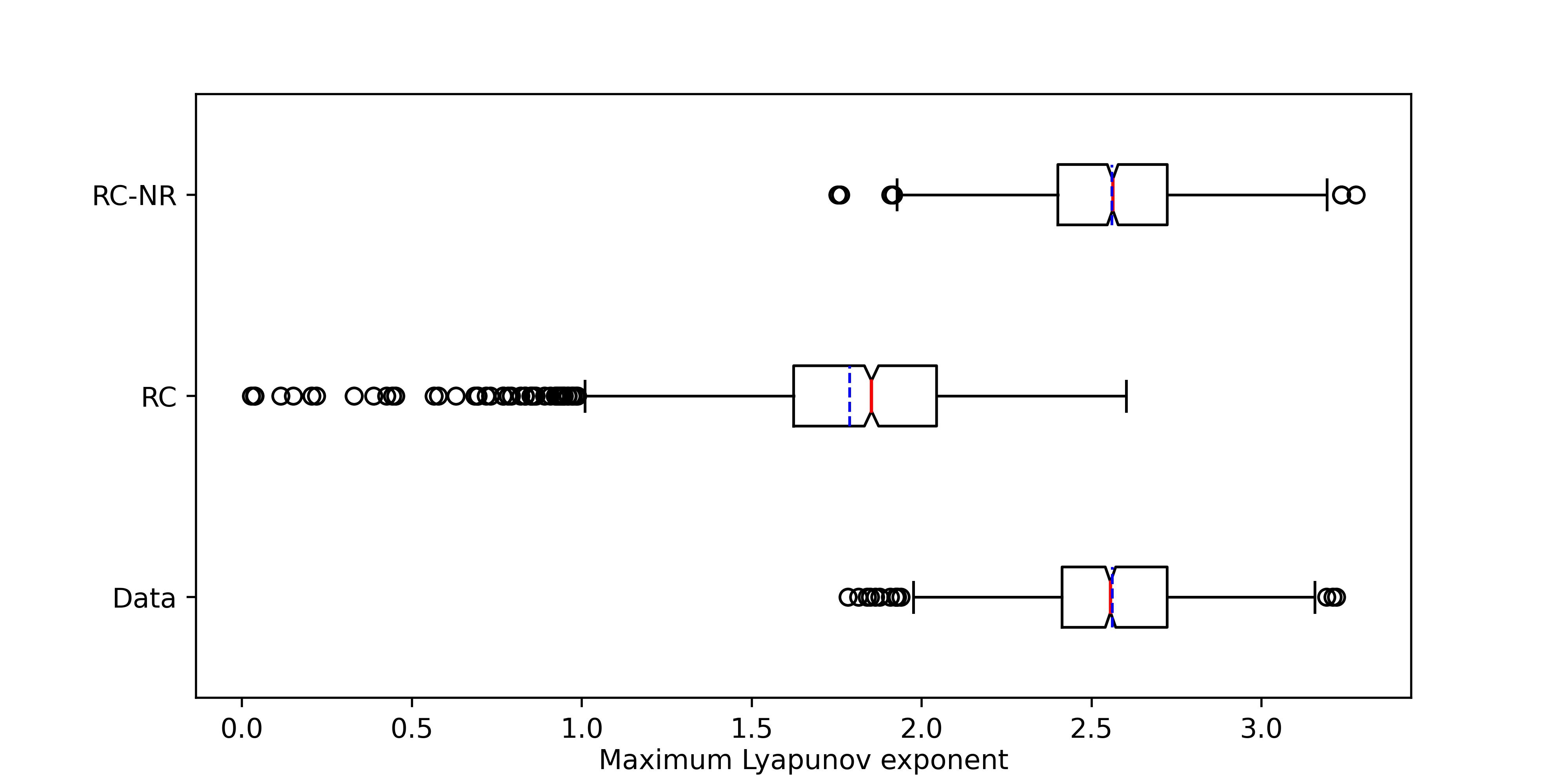}
    \label{fig: subfig: lorenzMLE}
    }
    \hspace{-0.35in}
    \subfigure[Histograms of close returns]{
    \includegraphics[width=0.48\textwidth]{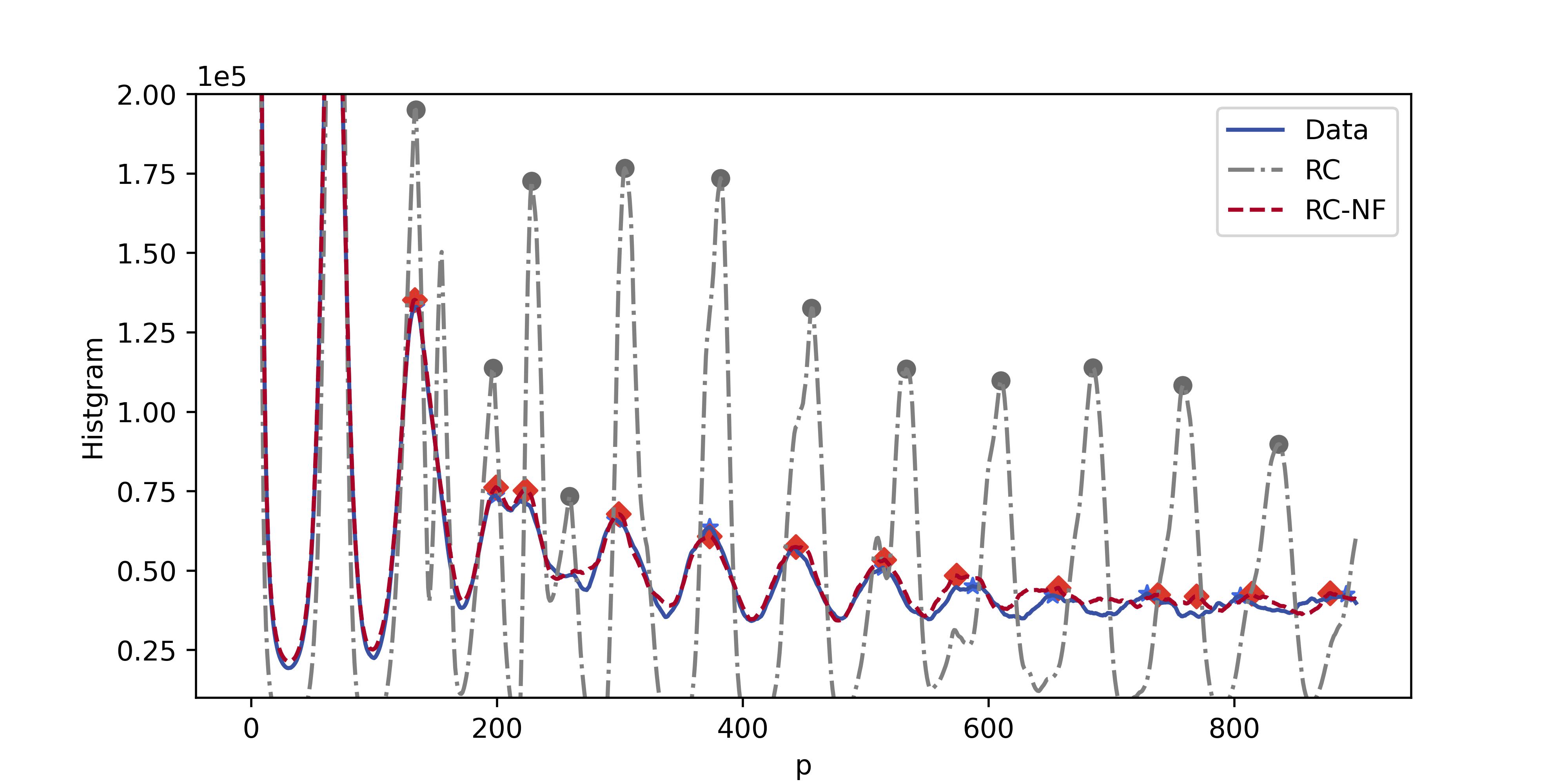}
    \label{fig: subfig: lorenzhistgram}
    }
    \makebox[\textwidth][c]{\subfigure[Close returns maps of a trajectory on the test dataset, data (left), RC (middle), RC-NF (right)]{
    \includegraphics[width=1.05\textwidth]{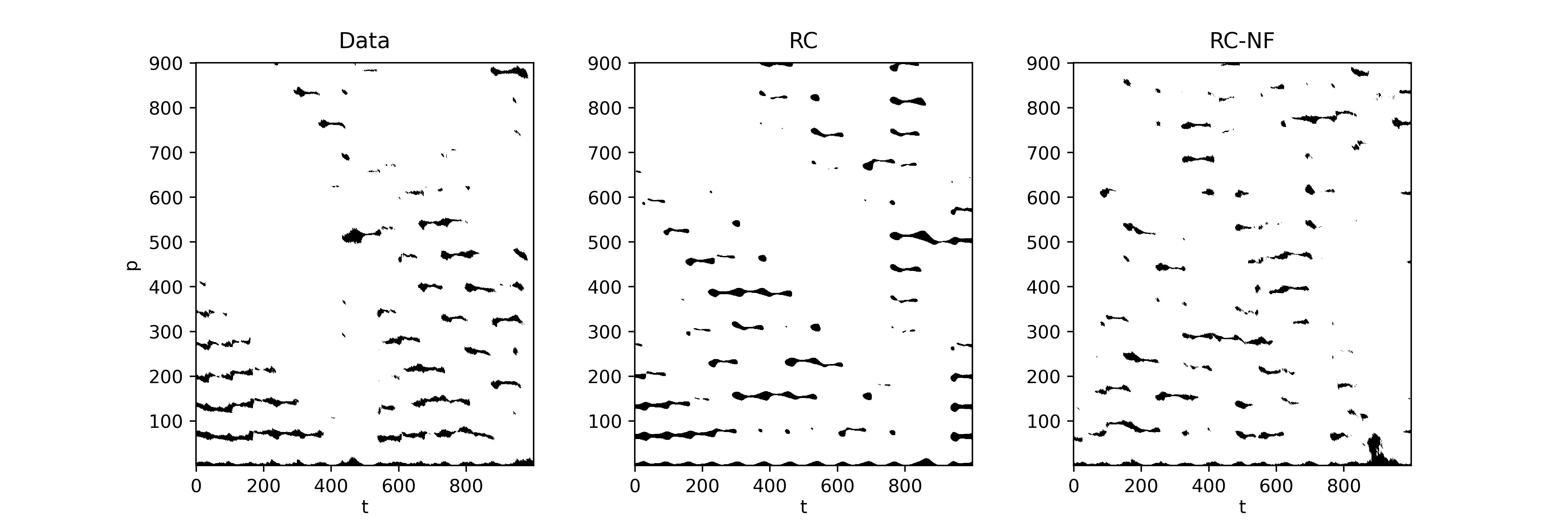}
    \label{fig: subfig: lorenclosereturn}
    }}
    \caption{Criteria of stochastic Lorenz systems. (a) Boxplots composed of MLEs of multiple trajectories. The red solid lines represent the medians of MLEs, and the blue dashed lines represent the means of MLEs. Medians: 2.5557 (data), 1.8525 (RC), 2.5622 (RC-NF). Means: 2.5609 (data), 1.7881 (RC), 2.5592 (RC-NF). (b) Close returns histograms. The solid blue line is calculated from the trajectory data, the gray dotted line is calculated from the RC prediction results, the red dashed line is calculated from the RC-NF prediction results, and their peaks are also marked with various markers. (c) Close returns maps of a trajectory in the testing phase.}
    \label{fig: lorenz indicator}
\end{figure}

We use the Euler-Maruyama scheme (time step $\delta t = 10^{-5}$), the RC model, and the RC-NF model to generate a long trajectory with the length of $10^{6}$ and observation time step $\Delta t = 0.01$, respectively. The warm-up time for RC and RC-NF models is $T_{warm}=500$. The probability density functions, autocorrelation functions (ACFs), and cross-correlation functions (CCFs) are compared using the long trajectories generated by the three methods, with results shown in Fig. \ref{fig: lorenz one track}. For these long trajectories, as depicted in Fig. \ref{fig: subfig: lorenzpdfonetrack}, the marginal probability density functions of the three dimensions indicate that the trajectory generated by the RC-NF model is closer to the trajectory generated by the numerical scheme, while the trajectory generated by RC is significantly different. The cross-correlation function between sequences $X$ and $Y$ is a deterministic function of lag $p$ defined as $r^{XY}_p = \text{Mean}_{t}\left((X_t-\bar{X})(Y_{t+p}-\bar{Y}))\right)/\sqrt{\text{Std}_{t}(X_t)\cdot\text{Std}_{t}(Y_t)}$, where $\bar{X}$ and $\bar{Y}$ represent the means of sequences $X$ and $Y$. In general, $r^{XY}_p \neq r^{YX}_p$, and $r^{XX}_p$ represents the ACF of sequence $X$. ACFs and CCFs of three long trajectories generated by different methods can be found in Fig. \ref{fig: subfig: lorenccf}. It is also clear that the results of the RC-NF model are closer to those of the numerical scheme.

\begin{figure}[!ht]
    \centering
    \subfigure[Marginal PDFs of a long trajectory]{
    \includegraphics[width=1.\textwidth]{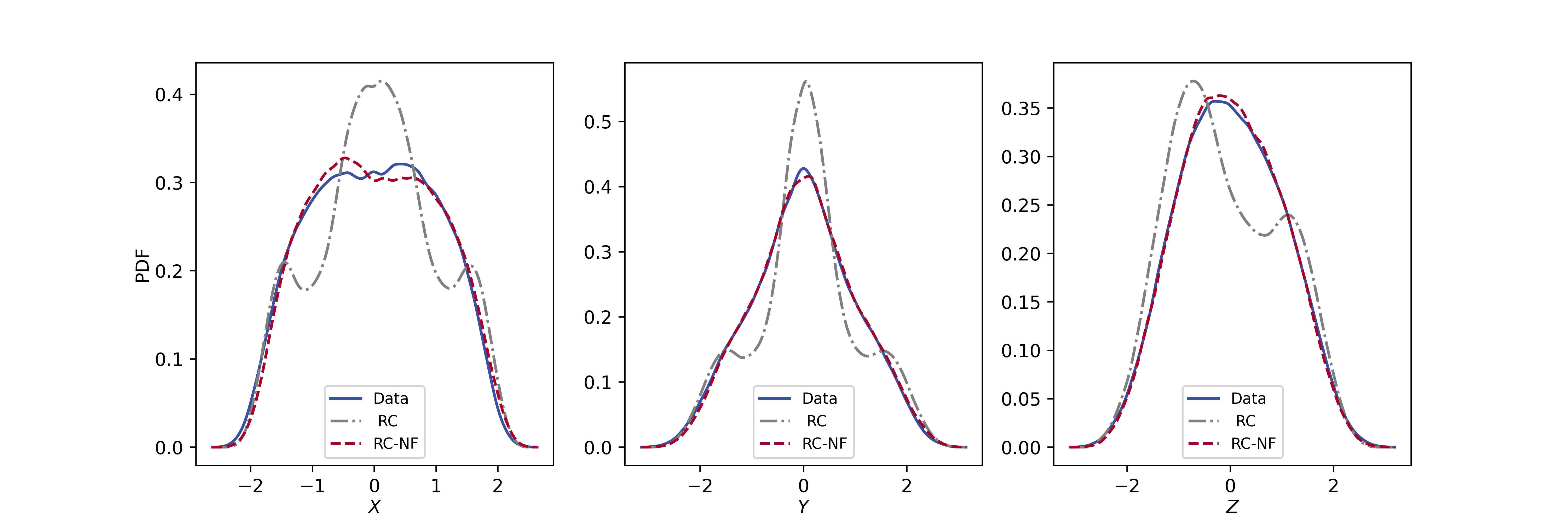}
    \label{fig: subfig: lorenzpdfonetrack}
    }\\
    \hspace{0.1in}
    \subfigure[ACFs (on the diagonal) and CCFs (off the diagonal) of a long trajectory]{
    \includegraphics[width=0.9\textwidth]{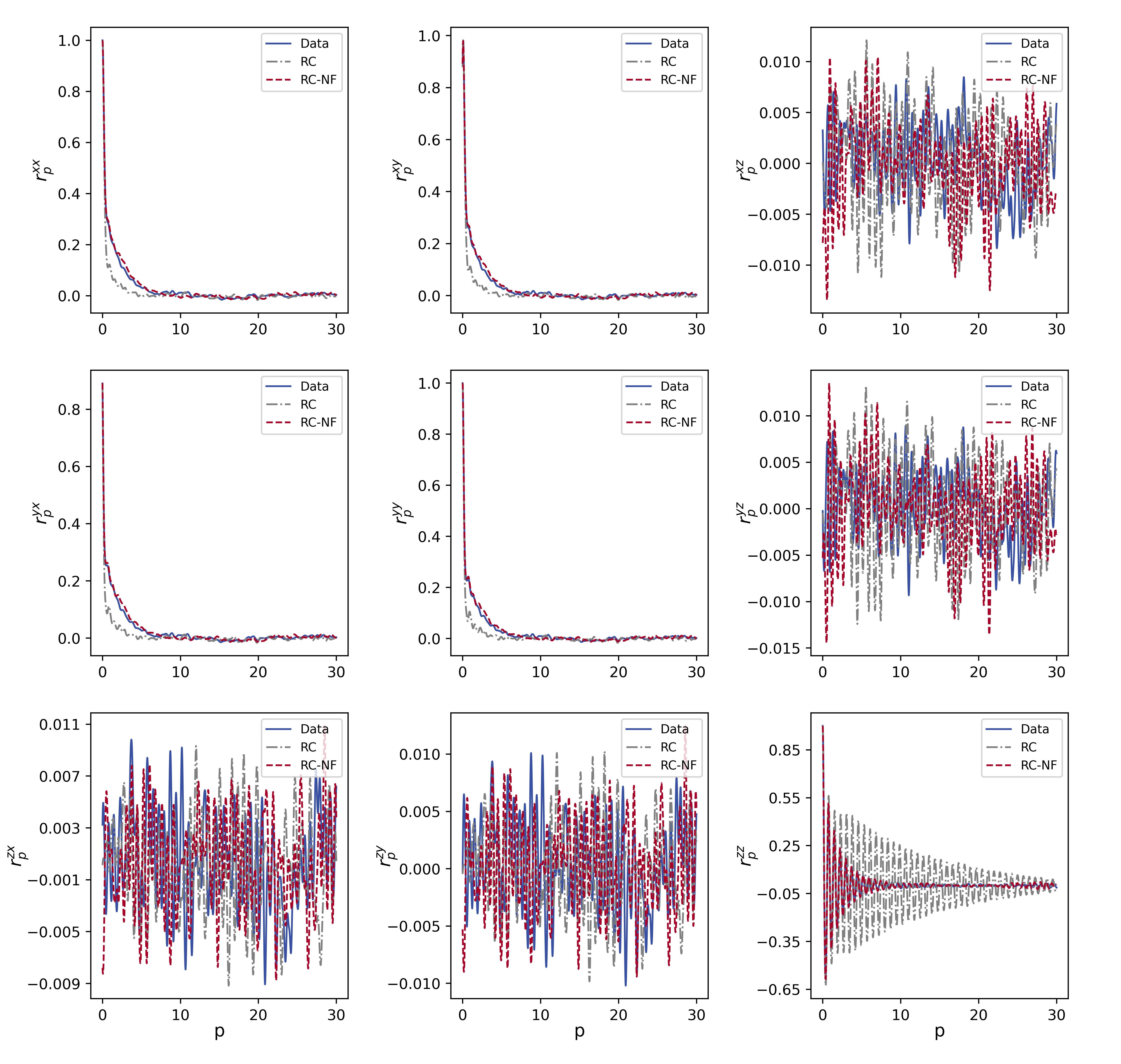}
    \label{fig: subfig: lorenccf}
    }
    \caption{Results of a long trajectory. (a) Marginal PDFs. (b) ACFs and CCFs. The blue solid lines are the results of the trajectory generated by the numerical scheme, the gray dotted lines are the results of the trajectory generated by the RC model, and the red dashed lines are the results of the trajectory generated by the RC-NF model.}
    \label{fig: lorenz one track}
\end{figure}

In summary, for the stochastic Lorenz system, we investigate the capacity of RC-NF to replicate strange attractors, the evolution of probability density, MLEs, and close returns maps on the test dataset. When comparing a long trajectory generated by different methods, we display the probability density functions, ACFs, and CCFs. RC-NF outperforms the traditional RC.

\section{Discussion} \label{sec: discussion}

In this paper, we devise a novel approach to predict the long-term evolution of stochastic dynamical systems and capture the corresponding dynamical behaviors. The framework (shorten as RC-NF) we proposed is model-free and takes only a few assumptions about the underlying stochastic dynamical systems (the approach is valid for SDEs/SDDEs, and it may be generalized to more types of systems). RC-NF integrates the virtues of Reservoir Computing for long-term prediction, as well as Normalizing Flow for producing samples from an approximated error distribution, both with low computational costs. We also illustrate the universality of the proposed framework under the discrete-time setting.

We verify the effectiveness of our framework with seven examples: the Ornstein-Uhlenbeck process, Double-Well system, stochastic Van der Pol oscillator, stochastic mixed-mode oscillation, linear SDDE, ENSO simplified model, and stochastic Lorenz system. Multiple criteria are reused in prediction tasks and generation tasks to manifest the effectiveness of RC-NF, namely, the probability density function, Wasserstein distance, KL divergence, transition rate, maximum Lyapunov exponent, close returns, autocorrelation function, and cross-correlation function. 
 
Despite the success of RC-NF, there are several directions that might be further explored. As we discussed in Sec. \ref{sec: convergence}, we consider the framework in discrete-time settings. While this meets the requirements of observing or generating data, it inevitably differs from the continuous time settings and causes errors. This relates to whether our proposed model can generate the true trajectory of the corresponding continuous stochastic systems. Besides, we shall extend the applicability of RC-NF to multiplicative noise or non-Gaussian noise to accommodate more general systems \cite{duan2015introduction}, which also implies a relaxation of Assumption \ref{assump: fix error}. More kinds of dynamical behaviors, such as stable/unstable/center manifolds, intermittency, bifurcation-induced tippings, and rate-induced tippings \cite{ashwin2012tipping}, should be investigated. Moreover, we have not taken into account more complicated RC frameworks, such as next generation reservoir computing\cite{gauthier2021next}, deep reservoir computing \cite{GALLICCHIO201787}, etc, whose effectiveness is unknown yet.

\section*{CRediT authorship contribution statement}
\textbf{Cheng Fang:} Conceptualization, Methodology, Software, Formal analysis, Writing - Original Draft. \textbf{Yubin Lu:} Conceptualization, Methodology, Formal analysis, Writing - Original Draft. \textbf{Ting Gao:} Conceptualization, Methodology, Funding acquisition, Writing - Review \& Editing. \textbf{Jinqiao Duan:} Discussion, Funding acquisition, Writing - Review \& Editing.

\section*{Declaration of competing interest}
The authors declare that they have no known competing financial interests or personal relationships that could have appeared to influence the work reported in this paper.

\section*{Data availability}
No data was used for the research described in the article. The code that supports the paper's findings can be found in GitHub \url{https://github.com/Fangransto/RC-NF}.

\section*{Acknowledgements}
We would like to thank Xu Sun, Ting Li, Wei Wei, and Luxuan Yang for helpful discussions. This work
was supported by the National Key Research and Development Program of China 2021ZD0201300, the National Natural Science Foundation of China 12141107, and Fundamental Research Funds for the Central Universities 5003011053. Lu was partially supported by the grant DOE DE-SC00222766 and NSF-2216926.

\appendix

\renewcommand\theequation{\Alph{section}.\arabic{equation}}
\setcounter{equation}{0}

\renewcommand\thetable{\Alph{section}.\arabic{table}}
\setcounter{table}{0}

\renewcommand\thefigure{\Alph{section}.\arabic{figure}}
\setcounter{figure}{0}

\renewcommand\theremark{\Alph{section}.\arabic{remark}}
\setcounter{remark}{0}

\renewcommand\thetheorem{\arabic{theorem}}
\setcounter{theorem}{0}

\section{Universality theory supplement} \label{appendix: convergence}

In this section, we supplement the theory required for the universality of Reservoir Computing, as well as the proofs of Theorem \ref{th: U lp uniqueness} and Theorem \ref{th: main}.

The following theorem states the universality of ESN (\ref{eq: esn}) in the $L^p$-sense. 

\begin{lemma} \label{lemma: ortega} (Gonon and Ortega \cite{Lukas2020}, Theorem 2)
Fix $p \in [1, \infty)$, let $\bm{Z}$ be a fixed $\mathbb{R}^d$-valued input process, and let $H$ be a functional such that $H(\bm{Z}) \in L^p(\Omega, \mathcal{F}, \mathbb{P})$. Suppose that the activation function $\sigma: \mathbb{R} \rightarrow \mathbb{R}$ is nonconstant, continuous, and has a bounded image. Then for every $\varepsilon > 0$, there exists $N \in \mathbb{N}$, $W_{in} \in \mathbb{M}_{N,d}$, $\bm{\zeta} \in \mathbb{R}^d$, $A \in \mathbb{M}_N$, $\bm{w} \in \mathbb{R}^N$ such that (\ref{eq: esn}) has the ESP, the corresponding filter is causal and time-invariant, the associated functional satisfies $H^{A, W_{in}, \bm{\zeta}}_{\bm{w}}(\bm{Z}) \in L^p(\Omega, \mathcal{F}, \mathbb{P})$ and 
\begin{equation} \label{functional Lp}
    \| H(\bm{Z}) - H^{A, W_{in}, \bm{\zeta}}_{\bm{w}}(\bm{Z}) \|_p < \varepsilon.
\end{equation}
\end{lemma}

\begin{remark} \label{remark: d-dimen fliter}
    The image of a filter map $U$ is a one-dimensional discrete time series. By stacking $d$ filters $U^{i}$, $i=1,2,\dots,d$, we can create a $d$-dimensional filter $\bm{U}: ((\mathbb{R}^d)^{\mathbb{Z}}, \otimes_{t \in \mathbb{Z}} \mathcal{B}(\mathbb{R}^d))\rightarrow ((\mathbb{R}^d)^{\mathbb{Z}}, \otimes_{t \in \mathbb{Z}} \mathcal{B}(\mathbb{R}^d))$. Additionally, by repeatedly utilizing Lemma \ref{lemma: ortega}, it is possible to create a $d$-dimensional $\bm{U}'$ whose components satisfy (\ref{eq: esn}).
\end{remark}

Based on the above Lemma \ref{lemma: ortega}, we show that if the time delay operator is bounded, then filters of the form (\ref{eq: esn}) are identical in the $L^p$-sense when the input sequences can be represented by each other using time delay operators.

\begin{theorem} 
Let $U$ be a causal and time-invariant filter and its associated functional is $H$. Fix $p \in [1, \infty)$, let $\bm{Z}$ be a fixed $\mathbb{R}^d$-valued input process, and $H(\pi_{\mathbb{Z}_-}(\bm{Z})) \in L^p(\Omega, \mathcal{F}, \mathbb{P})$. Assume the time delay operator $T_{-s}$ is a bounded operator for any $s \in \mathbb{Z}$. Then there exists causal and time-invariant filters $U'$, $U^{s'}$ satisfying (\ref{eq: esn}), constructed by input sequences $\pi_{\mathbb{Z}_-}(\bm{Z})$, $\pi_{\mathbb{Z}_-} \circ T_{-s} (\bm{Z})$, respectively, that are identical in the $L^p$-sense. That is, for every $\varepsilon > 0$, these causal and time-invariant filters $U'$, $U^{s'}$ satisfy $\| U'(\bm{Z})_s - U^{s'}(\bm{Z})_s \|_p < (\| T_{-s}\|_p + 1)\varepsilon$.
\end{theorem}

\begin{proof}

Without losing the generality, we consider the bounded time delay operator $T_{-1}$. Denoting $\bm{Z}^{(1)} = (\dots, \bm{Z}_{-1},\bm{Z}_{0}) \in ((\mathbb{R}^d)^{\mathbb{Z}_-}, \otimes_{t \in \mathbb{Z}_-} \mathcal{B}(\mathbb{R}^d))$, $\bm{Z}^{(2)} = (\dots, \bm{Z}_{-1}, \bm{Z}_{0}, \bm{Z}_{1}) \in ((\mathbb{R}^d)^{\mathbb{Z}_-}, \otimes_{t \in \mathbb{Z}_-} \mathcal{B}(\mathbb{R}^d))$, we have $\bm{Z}^{(2)} = T_{-1}(\bm{Z}^{(1)})$. Due to the boundedness of $T_{-1}$, $H(\bm{Z}^{(2)}) \in L^p(\Omega, \mathcal{F}, \mathbb{P}).$

Using Lemma \ref{lemma: ortega}, for every $\varepsilon > 0$, there exist causal and time-invariant filters of the form (\ref{eq: esn}) and associated functionals for input sequences $\bm{Z}^{(1)}$ and $\bm{Z}^{(2)}$, denoted as pairs ($U'$, $H'$), ($U''$, $H''$), respectively. In particular, for arbitrary $\varepsilon > 0$, $ \| H(\bm{Z}^{(1)}) - H'(\bm{Z}^{(1)}) \|_p < \varepsilon$, $ \| H(\bm{Z}^{(2)}) - H''(\bm{Z}^{(2)}) \|_p < \varepsilon$. Hence, for input sequence $\bm{Z}^{(2)}$,
\begin{equation*}
\begin{split}
    &\| H'(\bm{Z}^{(2)}) - H''(\bm{Z}^{(2)}) \|_p \\
    = & \| H'(\bm{Z}^{(2)}) - H(\bm{Z}^{(2)}) + H(\bm{Z}^{(2)}) - H''(\bm{Z}^{(2)}) \|_p \\
    \leq & \| H'(\bm{Z}^{(2)}) - H(\bm{Z}^{(2)}) \|_p + \| H(\bm{Z}^{(2)}) - H''(\bm{Z}^{(2)}) \|_p  \\
    = & \| H'(T_{-1}(\bm{Z}^{(1)})) - H(T_{-1}(\bm{Z}^{(1)})) \|_p + \| H(\bm{Z}^{(2)}) - H''(\bm{Z}^{(2)}) \|_p \\
    = & \| U'(T_{-1}(\bm{Z}^{(1)}))_0 - U(T_{-1}(\bm{Z}^{(1)}))_0 \|_p + \| H(\bm{Z}^{(2)}) - H''(\bm{Z}^{(2)}) \|_p \\  
    = & \| T_{-1}(U'(\bm{Z}^{(1)}))_0 - T_{-1}(U(\bm{Z}^{(1)}))_0 \|_p + \| H(\bm{Z}^{(2)}) - H''(\bm{Z}^{(2)}) \|_p \\ 
    \leq & \| T_{-1}\|_p\|(U'(\bm{Z}^{(1)}))_0 - (U(\bm{Z}^{(1)}))_0 \|_p + \| H(\bm{Z}^{(2)}) - H''(\bm{Z}^{(2)}) \|_p \\
    = & \| T_{-1}\|_p\|H'(\bm{Z}^{(1)}) - H(\bm{Z}^{(1)}) \|_p + \| H(\bm{Z}^{(2)}) - H''(\bm{Z}^{(2)}) \|_p \\ 
    < & (\| T_{-1}\|_p + 1)\varepsilon ,
\end{split}
\end{equation*}
 which is equivalent to $\| U'(\bm{Z}^{(2)})_0 - U''(\bm{Z}^{(2)})_0 \|_p < (\| T_{-1}\|_p + 1)\varepsilon$. For every $s \in \mathbb{Z}$, $\bm{Z}^{(s)}$ can be constructed similarly such that $\| U'(\bm{Z}^{(s)})_0 - U^{s'}(\bm{Z}^{(s)})_0 \|_p < (\| T_{-s}\|_p + 1)\varepsilon$. More generally, $\| U'(\bm{Z})_s - U^{s'}(\bm{Z})_s \|_p < (\| T_{-s}\|_p + 1)\varepsilon$ for any $s \in \mathbb{Z}$. The required result follows due to the arbitrariness of $\varepsilon$.  
\end{proof}

 We generalize the conclusions of Lemma \ref{lemma: ortega} and Theorem \ref{th: U lp uniqueness} to discrete time series $(\bm{X}_t^{\delta t})_{t \in  \mathbb{Z}}$ based on the numerical scheme.

\begin{theorem} 
Define a filter $\bm{U}: ((\mathbb{R}^d)^{\mathbb{Z}}, \otimes_{t \in \mathbb{Z}} \mathcal{B}(\mathbb{R}^d)) \rightarrow ((\mathbb{R}^d)^{\mathbb{Z}}, \otimes_{t \in \mathbb{Z}} \mathcal{B}(\mathbb{R}^d))$ by $\bm{U}((\bm{X}_t^{\delta t})_{t \in  \mathbb{Z}})_t = \bm{X}_{t+1}^{\delta t}$. 
Let $(\bm{X}_t^{\delta t})^{(1)}:=(\dots, \bm{X}^{\delta t}_{t-1}, \bm{X}^{\delta t}_{t})$ be a fixed $\mathbb{R}^d$-valued input process. 
Then for arbitrary $\varepsilon > 0$, there exists $N \in \mathbb{N}$, $W_{in} \in \mathbb{M}_{N,d}$, $\bm{\zeta} \in \mathbb{R}^d$, $A \in \mathbb{M}_N$, $\bm{w} \in \mathbb{R}^N$ such that (\ref{eq: esn}) has the ESP, the corresponding filter is causal and time-invariant, the associated functional satisfies $\bm{H}^{A, W_{in}, \bm{\zeta}}_{\bm{w}}((\bm{X}_t^{\delta t})^{(1)}) \in L^2(\Omega, \mathcal{F}, \mathbb{P})$ and 
\begin{equation} \label{appendix: functional Lp x}
    \| \bm{H}((\bm{X}_t^{\delta t})^{(1)}) - \bm{H}^{A, W_{in}, \bm{\zeta}}_{\bm{w}}((\bm{X}_t^{\delta t})^{(1)}) \|_2 < \varepsilon.
\end{equation}
Additionally, for any $s \in \mathbb{Z}$, consider input sequences $(\bm{X}_t^{\delta t})^{(1)}$ and $T_{-s}((\bm{X}_t^{\delta t})^{(1)})$, there exists causal and time-invariant filters $U'$, $U^{s'}$ satisfying (\ref{eq: esn}), constructed by input sequences $(\bm{X}_t^{\delta t})^{(1)}$, $T_{-s}((\bm{X}_t^{\delta t})^{(1)})$, that are identical in the $L^2$-sense.
\end{theorem}
\begin{proof}
Let's consider that the dimension of $X$ is $1$, and the $d$-dimension case can be easily generalized by Remark \ref{remark: d-dimen fliter}. For the filter $U$ defined in Theorem \ref{th: main}, the associated functional is $H((X_t^{\delta t})^{(1)}):= U((X_t^{\delta t})^{(1)})_0 = X_{t+1}^{\delta t} \in L^2(\Omega, \mathcal{F}, \mathbb{P})$. Then the formula (\ref{appendix: functional Lp x}) is obtained by Lemma \ref{lemma: ortega}.

For every $s \in \mathbb{Z}$, we have $T_{-s}(X_t^{\delta t}) = X_{t+s}^{\delta t} \in L^2(\Omega, \mathcal{F}, \mathbb{P})$, which means that $T_{-s}$ is a bounded operator. The required result follows by utilizing Theorem \ref{th: U lp uniqueness}.
\end{proof}

\section{The difference between RC and ESN on linear SDDE} \label{appendix: ESN}

We use the linear SDDE mentioned in Sec. \ref{sec: exper: linear sdde} to prove the difference between RC (\ref{eq: rc update}), (\ref{eq: rc readout}) and ESN (\ref{eq: esn}). The specific settings are identical to those in Sec. \ref{sec: exper: linear sdde}, with the exception that the ESN framework is simpler and does not require the selection of hyperparameter $\alpha$. We show some results that differ significantly between the two. 

The mean of Wasserstein distance of the ESN-NF method during the testing phase is $8.0540\times10^{-3}$, whereas the RC-NF method is $7.9967\times10^{-3}$. During the test period, the mean of KL divergence of the ESN-NF method in testing is $1.7927\times10^{-3}$, while the RC-NF method is $1.4012\times10^{-3}$. By imitating Fig. \ref{fig: ou rc and rc-nf} in Sec. \ref{sec: exper: ou}, we depict the rolling prediction results of the ESN-NF method and RC-NF method in Fig. \ref{fig: linear sdde rc and esn}.
\begin{figure}[!ht]
  \centering
   \subfigure[Rolling predictions of ESN-NF]{
    \label{fig: subfig: linear sdde esn} 
    \includegraphics[width=0.49\textwidth]{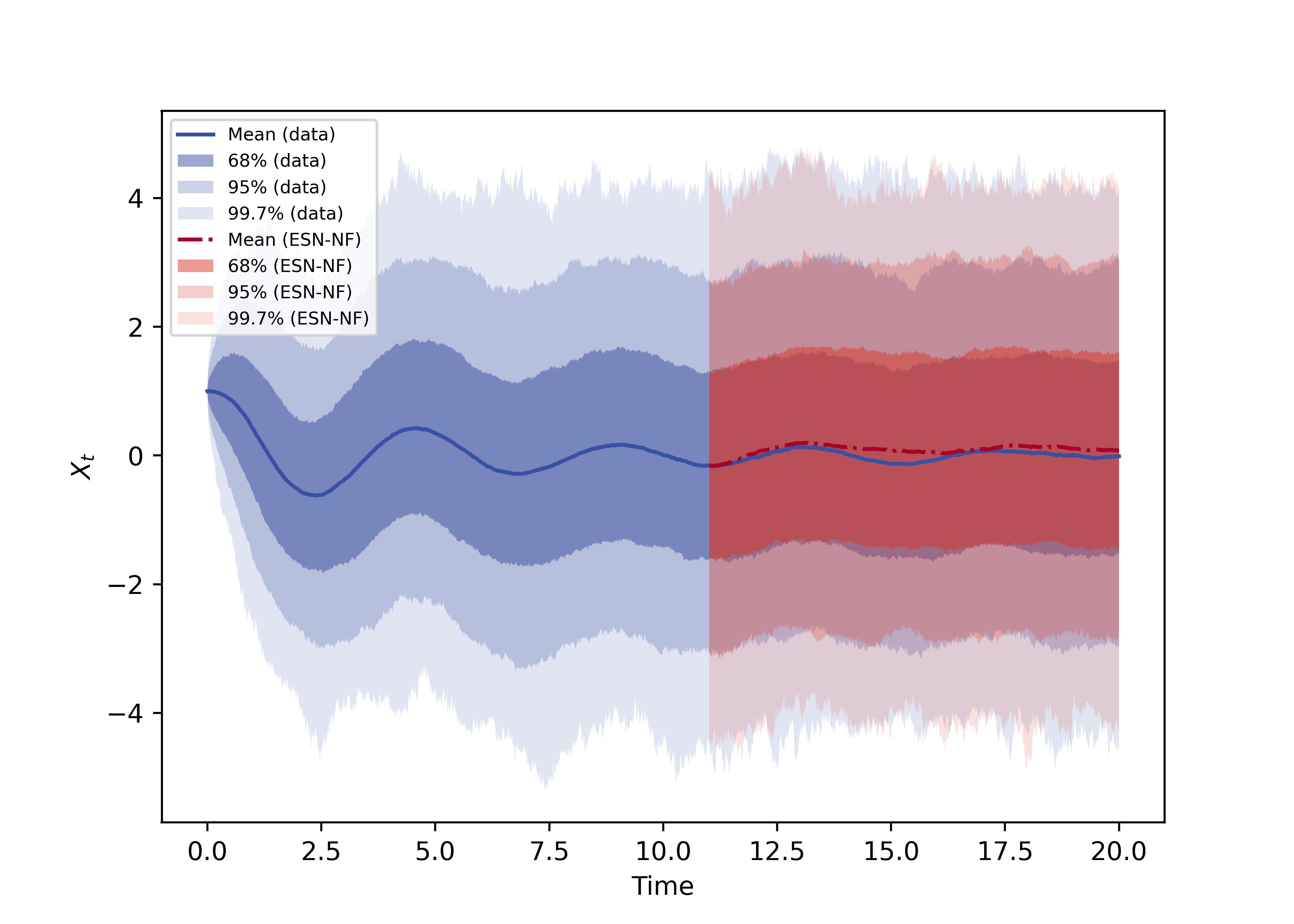}}
  \hspace{-0.1in}  
  \subfigure[Rolling predictions of RC-NF]{
    \label{fig: subfig: linear sdde rc-nf} 
    \includegraphics[width=0.49\textwidth]{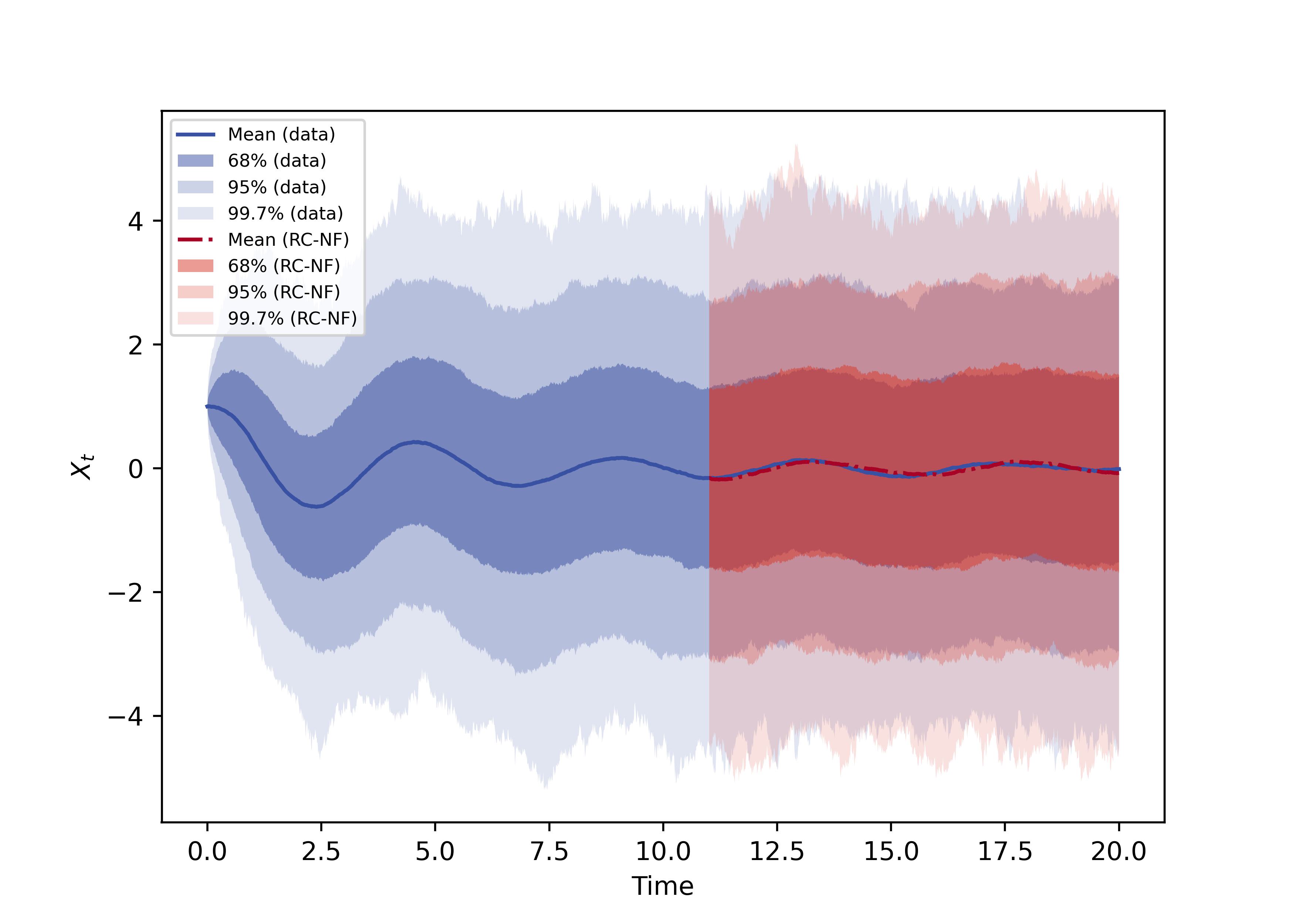}}
  \caption{Rolling predictions of ESN-NF and RC-NF. The solid blue line denotes the sample mean, and the shaded blue portions from dark to light represent the intervals calculated from the trajectory data. Rolling predictions using ESN-NF and RC-NF produce the red dotted dashes and shaded areas. (a) Rolling predictions of ESN-NF; (b) Rolling predictions of RC-NF.}
  \label{fig: linear sdde rc and esn} 
\end{figure}

Results-wise, the rolling prediction results of RC-NF are slightly superior to those of ESN-NF. Undoubtedly, the experiment presented in this section shows that ESN is capable of yielding satisfactory prediction results, which serves as evidence for the theories stated in Sec. \ref{sec: convergence}. However, RC produces a better result than ESN.

\section{Hyperparameters of various experiments} \label{appendix: hyper}

We initialize the Bayesian Optimization search process with the provided initial parameters and set the number of iterations to 50. After 50 iterations, the hyperparameters of RC corresponding to the smallest $\mathcal{L}_2$ error selected by BO on the validation dataset are shown in Table \ref{tab: hyper select}.

\begin{table}[!ht] 
\caption{The optimal hyperparameters selected by the Bayesian Optimization algorithm in each example.}
\label{tab: hyper select}
\centering
\begin{tabular}{lccccc} 
\bottomrule
\textbf{Experiments} & $\rho$ & $k$ & $\chi$ & $\alpha$ & $\lambda$ \\ 
\hline
 OU process & 0.5173 & 3 & 0.8933 & 0.8570 & $1.0000\times10^{0}$\\
 DW system & 0.8609 & 3 & 1.3469 & 0.9839 & $5.7206\times10^{-2}$\\
 Van der Pol oscillator & 0.5192 & 3 & 1.2345 & 0.6074 & $1.8232\times10^{-2}$\\
 stochastic MMO & 0.8609 & 3 & 1.3469 & 0.9839 & $5.7206\times10^{-2}$\\
 Linear SDDE & 0.9324 & 4 & 0.3655 & 0.2008 & $2.2073\times10^{-6}$\\
 ESNO simplified model & 0.8654 & 1 & 0.6024 & 0.0500 & $5.0616\times10^{-5}$\\
 stochastic Lorenz system & 0.3972 & 5 & 0.3817 & 0.9694 & $7.7623\times10^{-2}$\\
\toprule
\end{tabular}
\end{table}

\bibliographystyle{unsrt}
\bibliography{sample.bib}

\end{document}